\documentclass[twoside]{article}

\usepackage{amssymb}
\usepackage{amsmath}
\usepackage{mathrsfs}
\usepackage{amsthm}
\usepackage{amsfonts}
\usepackage{color}
\usepackage{latexsym,amssymb}

\allowdisplaybreaks

\def\ls{\lesssim}

\def\fz{\infty}

\def\r{\right}
\def\lf{\left}

\def\supp{{\mathop\mathrm{\,supp\,}}}

\def\rr{{\mathbb R}}

\def\rn{{{\rr}^n}}
\def\zz{{\mathbb Z}}
\def\nn{{\mathbb N}}
\def\cc{{\mathbb C}}

\newcommand{\wz}{\widetilde}

\def\cl{{\mathcal L}}
\def\az{\alpha}
\def\lz{\lambda}
\def\blz{\Lambda}
\def\bdz{\Delta}

\def\dz {\delta}

\def\bz{\beta}

\def\fai{\varphi}
\def\gz{{\gamma}}
\def\bgz{{\Gamma}}

\def\tz{\theta}
\def\sz{\sigma}

\def\wz{\widetilde}

\def\ls{\lesssim}

\def\pat{\partial}

\def\lp{{L^p(\rn)}}

\def\mol{{{\mathop\mathrm{mol}}}}

\def\hs{\hspace{0.3cm}}

\def\dsum{\displaystyle\sum}
\def\dint{\displaystyle\int}
\def\dfrac{\displaystyle\frac}
\def\dsup{\displaystyle\sup}

\newcommand{\R}{{\mathbb R}}

\newtheorem{theorem}{Theorem}[section]
\newtheorem{lemma}[theorem]{Lemma}
\newtheorem{corollary}[theorem]{Corollary}
\newtheorem{proposition}[theorem]{Proposition}
\theoremstyle{definition}
\newtheorem{remark}[theorem]{Remark}
\newtheorem{definition}[theorem]{Definition}
\numberwithin{equation}{section}
\def\supp{{\mathop\mathrm{\,supp\,}}}

\def\dist{{\mathop\mathrm{\,dist\,}}}

\numberwithin{equation}{section}

\begin{document}

\arraycolsep=1pt

\title{\bf\Large
Maximal Function Characterizations of Hardy Spaces Associated to Homogeneous
Higher Order Elliptic Operators
\footnotetext {\hspace{-0.35cm}
2010 {\it Mathematics Subject Classification}. Primary: 42B30;
Secondary: 42B20, 42B35, 42B37, 46E30, 35J48, 47B06, 47B38.
\endgraf {\it Key words and phrases}.
higher order elliptic operator, off-diagonal estimate, Hardy space,  maximal function,
square function, molecule, Riesz transform.
\endgraf Svitlana Mayboroda was partially supported by the NSF grants DMS 1220089
(CAREER), DMS 1344235 (INSPIRE), DMR 0212302 (UMN MRSEC Seed grant),
and the the Alfred P. Sloan Fellowship. Dachun Yang is supported by the National
Natural Science Foundation  of China (Grant Nos. 11171027 and
11361020) and the Specialized Research Fund for the Doctoral Program
of Higher Education of China (Grant No. 20120003110003) and the Fundamental Research Funds for Central
Universities of China (Grant Nos. 2013YB60 and 2014KJJCA10).}}
\author{Jun Cao, Svitlana Mayboroda and
Dachun Yang\footnote{Corresponding author}}
\date{ }
\maketitle

\vspace{-0.6cm}

\begin{center}
\begin{minipage}{13.5cm}
{\small {\bf Abstract} \quad Let $L$ be a homogeneous divergence form
higher order elliptic operator with complex bounded measurable coefficients
and $(p_-(L),\, p_+(L))$ be the maximal interval of exponents
$q\in[1,\,\infty]$ such that the semigroup $\{e^{-tL}\}_{t>0}$  is bounded on $L^q(\mathbb{R}^n)$.
In this article, the authors establish the non-tangential
maximal function characterizations of the associated Hardy spaces
$H_L^p(\mathbb{R}^n)$ for all $p\in(0,\,p_+(L))$,
which, when $p=1$, answers a question asked
by Deng et al. in [J. Funct. Anal. 263 (2012),
604-674]. Moreover, the authors characterize $H_L^p(\mathbb{R}^n)$ via various versions
of square functions and Lusin-area functions associated to the operator $L$.}
\end{minipage}
\end{center}

\vspace{0.2cm}

\section{Introduction\label{s1}}

\hskip\parindent Let $m\in\nn:=\{1,\,2,\,\ldots\}$ and
$L$ be a homogeneous higher order elliptic operator of the  form
\begin{eqnarray}\label{L homo}
L:=\dsum_{|\az|=m=|\bz|}(-1)^{m}
\pat^{\az}\lf(a_{\az,\,\bz}\pat^{\bz}\r),
\end{eqnarray}
where $\az:=(\az_1,\,\ldots,\,\az_n)$ and  $\bz:=(\bz_1,\,\ldots,\,\bz_n)$
belong to $(\zz_+)^n:=(\nn\cup \{0\})^n$, $|\az|:=\az_1+\cdots+\az_n$,
$|\bz|:=\bz_1+\cdots+\bz_n$ and
$\{a_{\az,\,\bz}\}_{|\az|=m=|\bz|}$
are bounded measurable functions mapping
$\rn$ into $\cc$ (see Subsection \ref{s2.1} below
for the exact definition of $L$ in \eqref{L homo}).
The aim of this article is to establish the maximal function characterizations
of the associated Hardy space $H_L^p(\rn)$ adapted to $L$,
which, when $p=1$, answers a question asked by Deng et al. in \cite{ddy12}.
It is now well known that such a Hardy space adapted to $L$, which has appeared in
\cite{ddy12,cy12}, is a good substitute of the
Lebesgue space $L^p(\rn)$, for smaller
$p$, when studying the regularity of the solution to the
corresponding elliptic equation (see, for example,
\cite{cks92,cks93,cds99,clms,dhmmy,mm04,ccyy,dy05-1,dxy,DL13,jy11,cmy}).

Notice that, if
$L\equiv -\bdz:=-\sum_{j=1}^n \frac{\pat^2}{\pat x_j^2}$
is the Laplace operator,
the Hardy space $H_{-\bdz}^p(\rn)$ is just the classical
Hardy space $H^p(\rn)$ which has been systematically studied by Fefferman and
Stein in their seminal paper \cite{fs72}. In the same paper,
Fefferman and Stein also established various real-variable
characterizations of  $H^p(\rn)$, including their non-tangential maximal function characterization
and Littlewood-Paley function characterizations.
Recall that, for all $f\in L^2(\rn)$ and $x\in\rn$, the \emph{non-tangential maximal
function} $\mathcal{N}_{\bdz}(f)(x)$ is defined by
\begin{eqnarray}\label{eqn CNmax}
\mathcal{N}_{\bdz}(f)(x):=\dsup_{(y,\,t)\in\bgz(x)}\lf|e^{-t\sqrt{\bdz}}(f)(y)\r|,
\end{eqnarray}
here and hereafter, for all $x\in\rn$,
\begin{eqnarray}\label{1.x1}
\bgz(x):=
\lf\{(y,\,t)\in\rr^{n+1}_+:\ |y-x|<t\r\}
\end{eqnarray}
denotes the \emph{cone with vertex} $x$ and
$\rr^{n+1}_+:=\rn\times(0,\,\fz)$.
Recall also that, if $n=1$, the nontangential maximal function characterization of
$H^p(\rn)$ was proved by Burkholder, Gundy and Silverstein \cite{bgs71}
more early,
which constitutes one of the motivations for Fefferman and Stein to study the
real-variable theory of $H^p(\rn)$.

Let $L\equiv-\rm{div}(A\nabla)$ be the second order elliptic operator, where
$\nabla:=(\frac{\pat}{\pat x_1},\,\ldots,\,\frac{\pat}{\pat x_n})$
and $A:=A(x)$ is an $n\times n$ matrix of complex bounded measurable coefficients
defined on $\rn$ and satisfies the \emph{ellipticity condition}
$$\lz |\xi|^2\le \Re e \lf(A\xi \cdot \overline{\xi}\r)\ \ \ \text{and}\ \ \
|A\xi \cdot \overline {\zeta}|\le \blz |\xi| |\zeta|$$
for all $\xi$, $\zeta\in\cc^n$ and for some positive constants  $0<\lz\le \blz<\fz$
independent of $\xi$ and $\zeta$. Hofmann and Mayboroda \cite{hm09} (for $p=1$),
and Jiang and Yang \cite{jy10} (for $p\in(0,\,1]$)
established the non-tangential maximal function characterization of the associated Hardy space
$H_L^p(\rn)$. Recall that, for all $f\in L^2(\rn)$ and $x\in\rn$, the \emph{non-tangential
maximal function} $\mathcal{N}_{L}(f)$ is defined by
\begin{eqnarray}\label{eqn EOmax}
\mathcal{N}_{L}(f)(x):=\dsup_{(y,\,t)\in\bgz(x)}
\lf\{\frac{1}{t^n} \dint_{B(y,\,t)} \lf|e^{-t^2L}(f)(z)\r|^2\,dz\r\}^{\frac{1}{2}},
\end{eqnarray}
here and hereafter, for all $(y,\,t)\in\rr_+^{n+1}$,
$B(y,\,t):=\{z\in\rn:\ |z-y|<t\}$. Observe that the non-tangential maximal function
\eqref{eqn EOmax} is a little bit different from \eqref{eqn CNmax}.
The main reason for adding an extra averaging in the spatial variable in
\eqref{eqn EOmax} is that  we need to compensate for the lack of pointwise
estimates of the heat semigroup (see \cite{hm09} for more details).

Now, let $L$ be a homogenous $2m$-th order elliptic operator
as in \eqref{L homo},
where $\{a_{\az,\,\bz}\}_{|\az|=m=|\bz|}$ are bounded measurable functions
mapping $\rn$ into $\cc$ satisfying the Ellipticity condition
$(\mathcal{E}_0)$ or the Strong ellipticity condition
$(\mathcal{E}_1)$ (see Subsection \ref{s2.1} for their definitions).
Some properties of Hardy spaces $H_L^p(\rn)$ associated with a
homogeneous higher order elliptic operator $L$ as in
\eqref{L homo}, for $p\in(0,\,1]$, have already been established in
 \cite{cy12,ddy12}.  To be precise,
let ${L}$ be the homogeneous higher order operator defined
as in \eqref{L homo} that satisfies the Ellipticity condition $(\mathcal{E}_0)$.
For all $f\in L^2(\rn)$ and $x\in\rn$,
the ${L}$-\emph{adapted square function} $S_{L}(f)$ is defined by
\begin{eqnarray}\label{4.1}
S_{L} (f)(x):=\lf\{\iint_{\Gamma(x)}\lf|t^{2m}{L}e^{-t^{2m}{L}}(f)(y)\r|^2
\frac{dy\,dt}{t^{n+1}}\r\}^{1/2}.
\end{eqnarray}

The following definition of Hardy spaces is motivated by
\cite{hm09,jy10,hmm}; see also \cite{ddy12,cy12} for the case
when $p\in(0,\,1]$.

\begin{definition}\label{Hardy space aoociated with L0 df}
Let $p\in(0,\,2]$, ${L}$ be as in \eqref{L homo} and satisfy the Ellipticity
condition $(\mathcal{E}_0)$. A function $f\in L^2(\rn)$ is said
to be in the \emph{space $\mathbb{H}_{L}^p(\rn)$} if $S_{L}(f)\in L^p(\rn)$; moreover,
let $\|f\|_{H_{L}^p(\rn)}:=\|S_{L}(f)\|_{L^p(\rn)}.$ The {\it Hardy
space} $H_{L}^p(\rn)$ is then defined as the completion of
$\mathbb{H}_{L}^p(\rn)$ with respect to the \emph{quasi-norm}
$\|\cdot\|_{H_{L}^p(\rn)}$.

For $p\in(2,\,\fz)$, the \emph{Hardy space} $H_{L}^p(\rn)$ is then defined as the
dual space of the Hardy space $H_{L^*}^{p'}(\rn)$, where $L^*$ denotes the \emph{adjoint
operator} of $L$ in $L^2(\rn)$ and $p':=\frac{p}{p-1}\in(1,\,2)$
denotes the \emph{conjugate exponent} of $p$.
\end{definition}

For the Hardy space $H_L^p(\rn)$ with $p\in(0,\,1]$,
the authors in \cite{cy12} established various characterizations of
$H_L^p(\rn)$ in terms of molecules, the
generalized square function or the Riesz transform.
Moreover, Deng et al. in \cite{ddy12} also established some other interesting
characterizations of these Hardy spaces
in the case of $p=1$. However, neither of the above articles gives the maximal
function characterizations of $H_L^p(\rn)$ even in the case of $p=1$ and it has been raised by
Deng et al. \cite{ddy12} as an open question
whether $H_L^1(\rn)$ has the maximal function characterizations or not.

Motivated by the above articles, the main purpose of this article
is to establish the maximal function characterizations
of the Hardy space $H_L^p(\rn)$ associated with $L$ as in \eqref{L homo}.
Based on \cite{hm09}, we first  introduce
the following versions of maximal functions associated with ${L}$.
For $\lz\in(0,\,\fz)$, $f\in L^2(\rn)$ and $x\in\rn$, the \emph{radial maximal
function}, $\mathcal{R}^\lz_{h,\,L}(f)$,
associated with the heat semigroup generated by ${L}$, is defined by
\begin{eqnarray}\label{radial maximal function for heat}
&&\mathcal{R}^\lz_{h,\,L}(f)(x):=\sup_{t\in(0,\fz)}\lf\{\frac{1}{(\lz t)^n}\dint_{B(x,\,\lz t)}
\lf|e^{-t^{2m}{L}}(f)(y)\r|^2\,dy\r\}^{\frac{1}{2}}.
\end{eqnarray}

Similarly, the \emph{non-tangential maximal
function}, $\mathcal{N}^\lz_{h,\,L}(f)$,
associated with the heat semigroup generated by ${L}$, is defined by
\begin{eqnarray}\label{non-tangential maximal function for heat}
&&\mathcal{N}^\lz_{h,\,L}(f)(x):=\sup_{(y,\,t)\in\Gamma^\lz(x)}
\lf\{\frac{1}{(\lz t)^n} \int_{B(y,\,\lz t)}\lf|e^{-t^{2m}{L}}(f)(z)\r|^2\,dz\r\}^{\frac{1}{2}},
\end{eqnarray}
where $\Gamma^\lz(x)$ for all $x\in\rn$ is defined by setting
\begin{eqnarray}\label{1.x2}
\Gamma^\lz(x):=
\lf\{(y,\,t)\in\rr^{n+1}_+:\ |y-x|<\lz t\r\}.
\end{eqnarray}
In what follows, when $\lz=1$, we remove the superscript
$\lz$ from $\mathcal{R}^\lz_{h,\,L}(f)$ and $\mathcal{N}^\lz_{h,\,L}(f)$
for simplicity. Observe also that, if $m=1$, then the maximal functions
defined in \eqref{radial maximal function for heat} and \eqref{non-tangential maximal function for heat}
coincide with those in \cite{hm09,jy10}.

\begin{definition}\label{maximal function Hardy space}
Let ${L}$ be as in \eqref{L homo} and satisfy the Ellipticity condition
$(\mathcal{E}_0)$. For all $p\in(0,\,\fz)$, the \emph{Hardy space}
$H_{\mathcal{N}_{h,\,L}}^p(\rn)$ is defined as the completion of
$\lf\{f\in L^2(\rn):\ \mathcal{N}_{h,\,L}(f)\in L^p(\rn)\r\}$
with respect to the \emph{quasi-norm} $$\lf\|f\r\|_{H_{\mathcal{N}_{h,\,L}}^p(\rn)}
:=\lf\|\mathcal{N}_{h,\,L}(f) \r\|_{L^p(\rn)}.$$

The \emph{Hardy space} $H_{\mathcal{R}_{h,\,L}}^p(\rn)$ is defined in the
way same as $H_{\mathcal{N}_{h,\,L}}^p(\rn)$ with $\mathcal{N}_{h,\,L}(f)$
in \eqref{non-tangential maximal function for heat} replaced
by $\mathcal{R}_{h,\,L}(f)$ in \eqref{radial maximal function for heat}.
\end{definition}

\begin{remark}\label{rem def MA}
By the argument that used in the proof of \cite[(6.50)]{hm09}
with a small modification, we know that, for all $p\in(0,\,\fz)$ and $f\in L^2(\rn)$,
$$\lf\|\mathcal{N}_{h,\,{L}}(f)\r\|_{L^p(\rn)}\sim
\lf\|\mathcal{R}_{h,\,{L}}(f)\r\|_{L^p(\rn)},$$
which implies that, for all $p\in(0,\,\fz)$,
$H_{\mathcal{R}_{h,\,L}}^p(\rn)=
H_{\mathcal{N}_{h,\,L}}^p(\rn)$ with equivalent quasi-norms.
\end{remark}

Now, let $(p_-(L),\, p_+(L))$ be the \emph{maximal interval} of exponents
$q\in[1,\,\fz]$ such that the family $\{e^{-tL}\}_{t>0}$
of operators  is bounded on $L^q(\rn)$. The following theorem gives out
the maximal function characterizations of $H_{L}^p(\rn)$ for all $p\in(0,\,p_+(L))$.

\begin{theorem}\label{non-tangential heat maximal functiona characterization}
Let ${L}$ be as in \eqref{L homo} and satisfy the Strong ellipticity condition
$(\mathcal{E}_1)$ (see Subsection \ref{s2.1} below for its definition).
Then, for all $p\in(0,\,p_+(L))$, $H_{L}^p(\rn)=
H_{\mathcal{N}_{h,\,L}}^p(\rn)=H_{\mathcal{R}_{h,\,L}}^p(\rn)$ with equivalent quasi-norms,
where $H_{\mathcal{N}_{h,\,L}}^p(\rn)$ and $H_{\mathcal{R}_{h,\,L}}^p(\rn)$
are defined as in Definition \ref{maximal function Hardy space}.
\end{theorem}

The proof of Theorem \ref{non-tangential heat maximal functiona characterization}
will be given in Section \ref{s3} of this article.

Before describing our method to
prove Theorem \ref{non-tangential heat maximal functiona characterization},
let us first recall some key points of the methods used
to establish the maximal function characterizations in \cite{fs72,hm09}.

For all $f\in L^2(\rn)$ and $x\in\rn$, let
\begin{eqnarray}\label{eqn Csquare}
S_{\bdz}(f)(x):=\lf\{\iint_{\bgz(x)}\lf|t\nabla e^{-t\sqrt{\bdz}}(f)(y)\r|^2\,
\frac{dy\,dt}{t^{n+1}}\r\}^{\frac{1}{2}}
\end{eqnarray}
be the \emph{Lusin-area function} of $f$ \emph{associated to} $\bdz$,
where $\bgz(x)$ for all $x\in\rn$ is as in \eqref{1.x1}. For convenience,
throughout the article, we distinguish in terminology the \emph{square function with gradient}
from the one without gradient via calling the former the \emph{Lusin-area function}.

Recall that Fefferman and Stein \cite{fs72} established the maximal function
characterizations of $H^p(\rn)$ by developing the equivalence of the $L^p(\rn)$
quasi-norms between $\mathcal{N}_{\bdz}(f)$ in \eqref{eqn CNmax} and
$S_{\bdz}(f)$ in \eqref{eqn Csquare}.
The heart of their proof is to control the integral
$\int_E \lf[S_{\bdz}(f)(x)\r]^2\,dx$
for some set $E$. By Fubini's theorem, this is reduced to the corresponding
estimates on a saw-tooth region $\mathcal{R}:=\cup_{x\in E}\bgz(x)$ based on $E$,
namely, we need to control
\begin{eqnarray}\label{1.1}
\iint_{\mathcal{R}}t\lf|\nabla e^{-t\sqrt{\bdz}}(f)(y)\r|^2\,dy\,dt.
\end{eqnarray}
The main tool that they used to estimate \eqref{1.1} is
Green's theorem. To this end, they first replaced the region $\mathcal{R}$
by an approximating family $\{\mathcal{R}_\epsilon\}_{\epsilon>0}$ of regions whose
boundaries have certain uniform smoothness, and then applied Green's theorem to
reduce the estimates on $\mathcal{R}_\epsilon$ to its boundary.
Finally,  they used some properties of harmonic functions to estimate
the corresponding integral on the boundary.

Hofmann and Mayboroda \cite{hm09} used the strategy similar to that of Fefferman
and Stein \cite{fs72}.
However, there do exist some differences between these two methods.
For all $f\in L^2(\rn)$ and $x\in\rn$, let
\begin{eqnarray}\label{eqn EOsquare}
S_{L} (f)(x):=\lf\{\iint_{\bgz(x)}
\lf|t^2Le^{-t^2L}(f)(y)\r|^2\,\frac{dy\,dt}
{t^{n+1}}\r\}^{\frac{1}{2}}
\end{eqnarray}
be the \emph{square function} of $f$ \emph{associated to} $L$, which was used
in \cite{hm09} to introduce the Hardy space $H_L^p(\rn)$ associated with
$L\equiv-\rm{div}(A\nabla)$.
Notice that this square function,  which is more convenient when introducing
$H_L^p(\rn)$, is different from the Lusin-area function \eqref{eqn Csquare}.
Then, to obtain the non-tangential maximal characterization
of $H_L^p(\rn)$, Hofmann and Mayboroda \cite{hm09} used a Caccioppoli's inequality to control
$S_{L} (f)$ by another, Lusin-type, area function defined in a way similar to
\eqref{eqn Csquare} with $e^{-t\sqrt{\bdz}}$ replaced by $e^{-t\sqrt{L}}$.
Furthermore, they used the truncated cone to approximate the cone in
\eqref{eqn EOsquare} before applying Fubini's theorem.
This reduces to estimating the following integral
\begin{eqnarray}\label{1.2}
\iint_{\mathcal{R}^{\az \epsilon,\,\az\epsilon,\,
\frac{1}{\az}}(E^*)}t\lf|\nabla e^{-t^2L}(f)(y)\r|^2\,dy\,dt,
\end{eqnarray}
where $\mathcal{R}^{\az \epsilon,\,\az\epsilon,\,
\frac{1}{\az}}(E^*)$ denotes a truncated saw-tooth region. Finally, in the estimate of
\eqref{1.2}, since $e^{-t^2L}(f)$ is no longer a harmonic function and hence
Green's theorem cannot be used directly,
Hofmann and Mayboroda \cite{hm09} made full use of the ellipticity condition
of the operator $-\rm{div}(A\nabla)$ and the divergence theorem to
reduce the corresponding estimates to the boundary of
$\mathcal{R}^{\az \epsilon,\,\az\epsilon,\,
\frac{1}{\az}}(E^*)$.

In the present article, to prove Theorem \ref{non-tangential heat maximal functiona characterization},
we first point out that the proof of the inclusion $$H_{L}^p(\rn)\subset
H_{\mathcal{N}_{h,\,L}}^p(\rn)$$ is relatively easy.
Indeed, for $p\in(0,\,1]$, by the molecular characterization of
$H_{L}^p(\rn)$ (see Theorem
\ref{molecular characterization for Hardy space L0} below),
we only need to consider the action of the non-tangential maximal function
$\mathcal{N}_{h,\,L}$ on each molecule of $H_{L}^p(\rn)$. For $p\in[2,\,p_+(L))$,
using the $L^2(\rn)$ off-diagonal estimates, we show that the radial maximal
function $\mathcal{R}_{h,\,L}$ is bounded on $L^p(\rn)$, which, together with
relations between $H_{L}^p(\rn)$ and $L^p(\rn)$
(see Lemma \ref{lem relation Hardy Lebesgue L0} below), the complex interpolation of
$H_L^p(\rn)$ (see Proposition \ref{pro interpolation Hardy L0} below)
and Remark \ref{rem def MA}, implies that, for all $p\in(1,\,p_+(L))$,
$\mathcal{N}_{h,\,L}$ is bounded from $H_L^p(\rn)$ to $L^p(\rn)$. This furnishes the proof
of the inclusion $H_{L}^p(\rn)\subset H_{\mathcal{N}_{h,\,L}}^p(\rn)$.

For the proof of the converse inclusion of Theorem
\ref{non-tangential heat maximal functiona characterization}, when $p\in(0,\,2]$,
we adapt the strategy of \cite{fs72,hm09}.  The higher order setting produces
new problems and requires new tools.
To be precise, let $S_{L}$ and
$S_{h,\,L}$ be, respectively, the square function and the Lusin-area function
as in \eqref{4.1} and
\eqref{gradient square function for heat}.
We obtain the converse inclusion by showing that, for all $p\in(0,\,2]$ and
$f\in L^2(\rn)$,
\begin{eqnarray}\label{eqn MainAM}
\| S_{L}(f)\|_{L^p(\rn)}\ls \|S_{h,\,L}(f)\|_{L^p(\rn)}
\ls \lf\|\mathcal{N}^\gz_{h,\,L}(f)\r\|_{L^p(\rn)},
\end{eqnarray}
where $\gz\in(0,\,\fz)$ and the implicit positive constants are independent of $f$.
More precisely, in the proof of the first inequality of \eqref{eqn MainAM}, we need a
new higher order  parabolic Caccioppoli's inequality
(see \eqref{caccioppoli's inequality3 with epsilon} below).
To obtain \eqref{caccioppoli's inequality3 with epsilon}, we first establish a
parabolic Caccioppoli's inequality with gradient terms on the right hand side of
the inequality (see \eqref{caccioppoli's inequality2 with epsilon} below). Then, by an induction
argument from Barton \cite{ba14}, we remove the gradient terms and obtain an improved
Caccioppoli's inequality in \eqref{caccioppoli's inequality3 with epsilon}.
Also, in the proof of the second inequality of
\eqref{eqn MainAM}, in order to avoid the estimates on the boundary
when applying the divergence theorem,
we use some special cut-off functions. In this argument, the
parabolic Caccioppoli's inequality \eqref{caccioppoli's inequality3 with epsilon} is also
needed. The case $p\in(2,\,p_+(L))$ of the first inequality
in \eqref{eqn MainAM} is obtained via duality; see Proposition \ref{cor HLP}
and Corollary \ref{pro domination of square functions K} below.

With the help of Theorem \ref{non-tangential heat maximal functiona characterization},
we point out that $H_L^p(\rn)$ can also be characterized by another kind of maximal functions.
To be precise, for $\lz\in(0,\,\fz)$, $f\in L^2(\rn)$ and $x\in\rn$, define the \emph{radial maximal
function}, $\wz{\mathcal{R}}^\lz_{h,\,L}(f)$,
associated with the heat semigroup generated by ${L}$, by setting
\begin{eqnarray}\label{radial maximal function for heat q}
&&\wz{\mathcal{R}}^\lz_{h,\,L}(f)(x):=\sup_{t\in(0,\fz)}\lf\{\frac{1}{(\lz t)^n}\dint_{B(x,\,\lz t)}
\dsum_{k=0}^{m-1}\lf|\lf(t\nabla\r)^ke^{-t^{2m}{L}}(f)(y)\r|^2\,dy\r\}^{\frac{1}{2}}.
\end{eqnarray}

Similarly, define the \emph{non-tangential maximal
function}, $\wz{\mathcal{N}}^\lz_{h,\,L}(f)$,
associated with the heat semigroup generated by ${L}$, by setting
\begin{eqnarray}\label{non-tangential maximal function for heat q}
&&\wz{\mathcal{N}}^\lz_{h,\,L}(f)(x):=\sup_{(y,\,t)\in\Gamma^\lz(x)}
\lf\{\frac{1}{(\lz t)^n} \int_{B(y,\,\lz t)}\dsum_{k=0}^{m-1}
\lf|\lf(t\nabla\r)^ke^{-t^{2m}{L}}(f)(z)\r|^2\,dz\r\}^{\frac{1}{2}},
\end{eqnarray}
where $\Gamma^\lz(x)$ for all $x\in\rn$ is defined as in \eqref{1.x2}.
In what follows, when $\lz=1$, we remove the superscript
$\lz$ from $\wz{\mathcal{R}}^\lz_{h,\,L}(f)$ and $\wz{\mathcal{N}}^\lz_{h,\,L}(f)$
for simplicity. Recall also that, if $m=1$, $\wz{\mathcal{R}}^\lz_{h,\,L}(f)$ and
$\wz{\mathcal{N}}^\lz_{h,\,L}(f)$ coincide, respectively, with ${\mathcal{R}}^\lz_{h,\,L}(f)$ and
${\mathcal{N}}^\lz_{h,\,L}(f)$ as in \eqref{radial maximal function for heat} and
\eqref{non-tangential maximal function for heat}.

The following theorem characterizes
$H_{L}^p(\rn)$ via maximal functions defined as in \eqref{radial maximal function for heat q}
and \eqref{non-tangential maximal function for heat q}.

\begin{theorem}\label{non-tangential heat maximal functiona characterization gradient}
Let ${L}$ be as in \eqref{L homo} and satisfy the Strong ellipticity condition
$(\mathcal{E}_1)$. Then, for all $p\in(0,\,p_+(L))$, $H_{L}^p(\rn)=
H_{\wz{\mathcal{N}}_{h,\,L}}^p(\rn)=H_{\wz{\mathcal{R}}_{h,\,L}}^p(\rn)$ with equivalent quasi-norms,
where $H_{\wz{\mathcal{N}}_{h,\,L}}^p(\rn)$ and $H_{\wz{\mathcal{R}}_{h,\,L}}^p(\rn)$ are
defined similarly as in Definition \ref{maximal function Hardy space} with $\mathcal{N}_{h,\,L}$
and $\mathcal{R}_{h,\,L}$ therein replaced, respectively, by $\wz{\mathcal{N}}_{h,\,L}$ and
$\wz{\mathcal{R}}_{h,\,L}$.
\end{theorem}

The proof of Theorem \ref{non-tangential heat maximal functiona characterization gradient}
will be given in Section \ref{s3} of this article.

Now we characterize $H_{L}^p(\rn)$ by using the non-tangential maximal function
with only the $(m-1)$-order gradients of the heat semigroup generated by $L$.

Let $\psi\in C_c^\fz (B(0,\,2))$ satisfy $0\le\psi\le 1$,
$\psi\equiv1$ on $B(0,\,1)$ and,
for all $k\in\{0,\,\ldots,\,m\}$,
$$\lf\|\nabla^k\psi\r\|_{L^\fz(\rn)}\ls 1.$$
For all $(x,\,t)\in\rr^{n+1}_+$ and $y\in\rn$,
let
\begin{eqnarray}\label{eqn COF}
\psi_{x,\,t}(y):=\frac{1}{t^n}\psi\lf(\frac{y-x}{t}\r).
\end{eqnarray}
Then $\psi_{x,\,t}\in C_c^\fz (B(x,\,2t))$ and $0\le\psi_{x,\,t}\le 1$,
$\psi_{x,\,t}\equiv1$ on $B(x,\,t)$ and,
for all $k\in\{0,\,\ldots,\,m\}$,
$$\lf\|\nabla^k\psi_{x,\,t}\r\|_{L^\fz(\rn)}\ls t^{-k}.$$

Having fixed any $\psi$ as above, for any $f\in L^2(\rn)$, $x\in\rn$
and $\lz\in(0,\,\fz)$, we introduce the following version of the
\emph{non-tangential maximal function}, $\mathcal{N}^\lz_{h,\,\psi,\,L}(f)$, associated with the
heat semigroup generated by $L$, by setting
\begin{eqnarray*}
&&\mathcal{N}^\lz_{h,\,\psi,\,L}(f)(x):=\sup_{(y,\,t)\in\Gamma^\lz(x)}
\lf\{\frac{1}{(\lz t)^n} \int_{B(y,\, \lz t)}
\lf|(t\nabla)^{m-1}\lf(\psi_{x,\,t}
e^{-t^{2m} L}(f)\r)(z)\r|^2\,dz\r\}^{\frac{1}{2}}.
\end{eqnarray*}
When $\lz=1$, we remove the superscript $\lz$
from $\mathcal{N}^\lz_{h,\,\psi,\,L}(f)$ for simplicity.

\begin{proposition}\label{pro NTMFC2 L0}
Let $L$ be as in \eqref{L homo}
and satisfy the Strong ellipticity condition $(\mathcal{E}_1)$, and
let $\psi$ be a cut-off function defined as in \eqref{eqn COF}.
For any $p\in(0,\,p_+(L))$, denote by $H_{\mathcal{N}_{h,\,\psi,\,L}}^p(\rn)$
the Hardy space defined as $H_{\mathcal{N}_{h,\,L}}^p(\rn)$ with
$\mathcal{N}_{h,\,L}$ replaced by $\mathcal{N}_{h,\,\psi,\,L}$.
Then $H_{L}^p(\rn)= H_{\mathcal{N}_{h,\,\psi,\,L}}^p(\rn)$
with equivalent quasi-norms.
In particular, different choices of $\psi$ in the definition of
$H_{\mathcal{N}_{h,\,\psi,\,L}}^p(\rn)$ above yield equivalent
quasi-norms.
\end{proposition}

The proof of Proposition \ref{pro NTMFC2 L0} will be given in Section
\ref{s3} of this article, where a higher Poincar\'e's inequality from \cite{Mo08}
is used.

By the method used in the proof of Theorem \ref{non-tangential heat maximal functiona characterization}, we are
able to characterize $H_L^p(\rn)$ via some more general square functions
and Lusin-area functions.

To be precise, for all $\lz\in(0,\,\fz)$,  $k\in\zz_+$ and $f\in L^2(\rn)$,
the ${L}$-\emph{adapted square function} $S^\lz_{L,\,k}(f)$
is defined by setting, for all $x\in\rn$,
\begin{eqnarray}\label{eqn KSF}
S_{L,\,k}^\lz (f)(x):=\lf\{\iint_{\Gamma^\lz(x)}\lf|\lf(t^{2m}{L}\r)^ke^{-t^{2m}{L}}(f)(y)\r|^2
\frac{dy\,dt}{t^{n+1}}\r\}^{1/2}
\end{eqnarray}
and the \emph{Lusin-area function} $S^\lz_{h,\,L,\,k}(f)$
by setting, for all $x\in\rn$,
\begin{eqnarray}\label{gradient square function for heat}
S^\lz_{h,\,L,\,k} (f)(x):=\lf\{\iint_{\Gamma^\lz(x)}\lf|\lf(t\nabla\r)^m
\lf(t^{2m}{L}\r)^ke^{-t^{2m}{L}}(f)(y)\r|^2
\frac{dy\,dt}{t^{n+1}}\r\}^{1/2},
\end{eqnarray}
where $\Gamma^\lz$ is as in \eqref{1.x2}.
For simplicity, if $k=1$, we remove the subscript
$k$ from $S_{L,\,k}^\lz (f)$ and,  if $k=0$, we
remove the subscript $k$ from
$S^\lz_{h,\,L,\,k} (f)$. Also, if $\lz=1$, we
remove the superscript $\lz$ from both
$S_{L,\,k}^\lz (f)$ and $S^\lz_{h,\,L,\,k} (f)$.

\begin{definition}\label{def SFLAF Hardy}
Let ${L}$ be as in \eqref{L homo} and satisfy the Ellipticity condition
$(\mathcal{E}_0)$. For all
$k\in\nn$ and $p\in(0,\,\fz)$, the \emph{Hardy space}
$H_{S_{L,\,k}}^p(\rn)$ is defined as the completion of
$$\lf\{f\in {L^2(\rn)}:\ S_{L,\,k}(f)\in L^p(\rn)\r\}$$
with respect to the \emph{quasi-norm} 
$$\lf\|f\r\|_{H_{S_{L,\,k}}^p(\rn)}
:=\lf\|S_{L,\,k}(f) \r\|_{L^p(\rn)}.$$

Moreover, for all $k\in\zz_+$ and $p\in(0,\,\fz)$,
the \emph{Hardy space} $H_{S_{h,\,L,\,k}}^p(\rn)$ is defined in the
way same as $H_{S_{L,\,k}}^p(\rn)$ with $S_{L,\,k}(f)$
in \eqref{eqn KSF} replaced
by $S_{h,\,L,\,k}(f)$ in \eqref{gradient square function for heat}.
\end{definition}

The following theorem establishes the characterization
of $H_L^p(\rn)$ via, respectively, some square functions and some
Lusin-area functions.

\begin{theorem}\label{thm SFLAF Hardy}
Let ${L}$ be as in \eqref{L homo} and satisfy the Ellipticity condition
$(\mathcal{E}_0)$. Then

\begin{itemize}
\item [{\rm (i)}] for all $k\in\nn$ and $p\in(0,\,p_+(L))$, $H_L^p(\rn)
=H_{S_{L,\,k}}^p(\rn)$ with equivalent quasi-norms;

\item [{\rm (ii)}] for all $k\in\nn$ and $p\in(0,\,p_+(L))$, $H_L^p(\rn)
=H_{S_{h,\,L,\,k}}^p(\rn)$ with equivalent quasi-norms.
\end{itemize}
\end{theorem}

The proof of Theorem \ref{thm SFLAF Hardy} will be given in Section \ref{s3}
of this article.

Let us end this section by making some conventions on the notation.
Throughout the paper, we always let $\nn:=\{1,2,\ldots\}$
and $\zz_+:=\nn\cup\{0\}$. Denote the {\it differential operator}
$\frac{\pat^{\az}}{\pat x_1^{\az_1}\cdots \pat x_n^{\az_n}}$
simply by $\pat^\az$, where $\az:=(\az_1,\ldots,\az_n)$ and
$|\az|:=\az_1+\cdots+\az_n$.  We use $C$ to denote a {\it positive
constant} that is independent of the main parameters involved but
whose value may differ from line to line, and $C_{(\az,\ldots)}$ to
denote a {\it positive constant} depending on the parameters $\az,$
$\ldots$. {\it Constants with subscripts}, such as $C_1$, do not
change in different occurrences. If $f\le Cg$, we then write $f\ls
g$ and, if $f\ls g\ls f$, we then write $f\sim g$. For any
$x\in\rr^n$, $r\in(0,\fz)$ and $\lz\in (0,\fz)$, let $B(x,r):=\{y\in\rr^n:|x-y|<r\}$
and $\lz B:=B(x,\lz r)$.
Also, for any set $E\subset\rn$, $\chi_E$ denotes its
{\it characteristic function} and, for all $z\in\cc$, $\Re e\,z$ denotes its
{\it real part}.

\section{The Hardy space $H_{L}^p(\rn)$}\label{s2}

\hskip\parindent In this section, we study the Hardy space $H_{L}^p(\rn)$
associated with the  homogeneous higher order elliptic operator $L$
in \eqref{L homo}. To this end, we first collect some known basic facts on $L$
in Subsection \ref{s2.1}; then, in Subsection \ref{s2.2}, we present some real-variable
properties of the Hardy space $H_{L}^p(\rn)$ associated with $L$ for $p\in(0,\,\fz)$.
Recall that, for $p\in(0,\,1]$, $H_{L}^p(\rn)$ has been studied in
\cite{cy12,ddy12}. Our results here also include the case $p\in(1,\,\fz)$.

\subsection{Homogeneous  higher order elliptic operators}\label{s2.1}

\hskip\parindent
Let $m\in\nn$ and  $\dot{W}^{m,\,2}(\rn)$
be the $m$-\emph{order homogeneous Sobolev space} equipped with the
usual \emph{norm}
\begin{eqnarray*}
\|f\|_{{\dot W}^{m,\,2}(\rn)}:
=\lf[\sum_{|\az|= m}\|\pat^\az f\|_{L^2(\rn)}^2\r]^{1/2}.
\end{eqnarray*}
For all multi-indices $\az$ and $\bz$ in $(\zz_+)^n$ satisfying $|\az|=m=|\bz|$,
let $a_{\az,\,\bz}$ be a complex valued $L^\fz$ function on $\rn$.
For all $f$ and $g\in
\dot W^{m,\,2}(\rn)$,  define the  \emph{sesquilinear form} $\mathfrak{a}_0$, mapping
$\dot W^{m,\,2}(\rn)\times \dot W^{m,\,2}(\rn)$ into $\cc$,
by
\begin{eqnarray}\label{form0}
\mathfrak{a}_0(f,\,g):=\dsum_{|\az|=m=|\bz|}\dint_{\rn} a_{\az,\,\bz}(x) \pat^{\bz}f(x)
\overline{\pat^\az g(x)}\,dx.
\end{eqnarray}

The following ellipticity condition on
$\{a_{\az,\,\bz}\}_{|\az|=m=|\bz|}$ is necessary.

\medskip

\noindent {\bf Ellipticity condition $(\mathcal{E}_{0})$.}
There exist constants
$0<\lz_0\le\Lambda_0<\fz$
such that, for all $f$ and $g\in \dot W^{m,\,2}(\rn)$,
\begin{eqnarray*}
\lf|\sum_{|\az|=m=|\bz|}
\int_\rn a_{\az,\,\bz}(x)\partial^\bz f(x)\overline{\partial^\az g(x)}\,dx\r|
\le &&\blz_0 \|\nabla^m f\|_{L^2(\rn)}\|\nabla^m g\|_{L^2(\rn)}
\end{eqnarray*}
and
$$\Re e \lf\{\sum_{|\az|=m=|\bz|}\int_\rn
 a_{\az,\,\bz}(x) \partial^\bz f(x)\overline{\partial^\az f(x)}\,dx\r\}\ge
\lz_0\,\|\nabla^m f\|_{L^2(\rn)}^2,$$
where
$$\|\nabla^m f\|_{L^2(\rn)}:=\lf[\sum_{|\az|=m} \int_{\rn} |\pat^\az f(x)|^2\,dx\r]^{1/2}.$$

\medskip

We also need the following strong ellipticity condition on
$\{a_{\az,\,\bz}\}_{|\az|=m=|\bz|}$.

\medskip

\noindent {\bf Strong ellipticity condition $(\mathcal{E}_{1})$.}
There exists a positive constant $\lz_1$ such that,
for all $\xi:=\{\xi_{\az}\}_{|\az|=m}$ with $\xi_\az\in\cc$
and almost every $x\in\rn$,
\begin{eqnarray*}
\Re e\lf\{ \dsum_{|\az|=m=|\bz|} a_{\az,\,\bz}(x)\xi_{\bz} \overline{\xi_{\az}}\r\}\ge
\lz_1 |\xi|^2=\lz_1\lf\{\dsum_{|\az|=m}\lf|\xi_\az\r|^2\r\}.
\end{eqnarray*}
Moreover, for all multi-indices $\az$ and $\bz$ with
$|\az|=m=|\bz|$, $a_{\az,\,\bz}\in L^\fz(\rn)$.

\medskip

\begin{remark}\label{r2.1}
It is easy to see that the Strong ellipticity condition $(\mathcal{E}_1)$
implies the Ellipticity condition $(\mathcal{E}_0)$. However,
the equivalence between $(\mathcal{E}_1)$ and $(\mathcal{E}_0)$
is only a specific feature of second order operators (see,
for example, \cite[p.\,15]{AT98}). For more relationships on
these two kinds of ellipticity conditions, we refer the reader to
\cite[p.\,365]{ahmt01}.
\end{remark}

Let us recall some basic facts on sesquilinear forms
from \cite[p.\,3, Section 1.2.1]{ou05}.

\begin{definition}[\cite{ou05}]\label{def form ou}
Assume that $\mathfrak{a}:
D(\mathfrak{a})\times D(\mathfrak{a})\to \cc$ is a sesquilinear form in the Hilbert space
$\mathcal{H}$.
\begin{itemize}
\item[(i)] $\mathfrak{a}$ is said to be \emph{densely defined} if the domain of $\mathfrak{a}$,
$D(\mathfrak{a})$, is dense in $\mathcal{H}$;

\item[(ii)] $\mathfrak{a}$  is said to be \emph{accretive} if, for all $u\in D(\mathfrak{a})$,
\begin{eqnarray*}
\Re e\, \lf(\mathfrak{a}(u,\,u)\r)\ge 0;
\end{eqnarray*}
\item[(iii)] $\mathfrak{a}$  is said to be  \emph{continuous} if there exists a nonnegative constant $M$
such that, for all $u$, $v\in D(\mathfrak{a})$,
\begin{eqnarray*}
\lf|\mathfrak{a}(u,\,v)\r|\le M \|u\|_{\mathfrak{a}}\|v\|_{\mathfrak{a}},
\end{eqnarray*}
where $\|u\|_{\mathfrak{a}}:=\sqrt{\Re e\,(\mathfrak{a}(u,\,u))+\|u\|_{\mathcal{H}}^2}$;

\item[(iv)] $\mathfrak{a}$  is said to be \emph{closed} if $(D(\mathfrak{a}),\,\|\cdot\|_{\mathfrak{a}})$
is a complete space.
\end{itemize}
\end{definition}

For a densely defined, accretive, continuous and closed sesquilinear form
in the Hilbert space $\mathcal{H}$,
we have the following conclusion from \cite[Proposition 1.22]{ou05}.
Recall that $\|\cdot\|_{\mathcal{H}}$ and $(\cdot,\,\cdot)_{\mathcal{H}}$
denote, respectively, the \emph{inner product} and the \emph{norm} of $\mathcal{H}$.

\begin{proposition}[\cite{ou05}]\label{pro Ouhabaz}
Assume that $\mathfrak{a}$ is a densely defined, accretive, continuous and closed
sesquilinear form in the Hilbert space $\mathcal{H}$. Then there
exists a densely defined operator $T$, defined by setting
\begin{eqnarray*}
D(T):=\lf\{u\in \mathcal{H}:\ \ \exists\ v\in \mathcal{H}\ \text{such that},
\ \text{for all}\ \phi\in D(\mathfrak{a}), \ \mathfrak{a}(u,\,\phi)=
(v,\,\phi)_{\mathcal{H}}\r\}
\end{eqnarray*}
and $Tu:=v$ for all $u\in D(T)$, such that, for all $\lz\in(0,\,\fz)$, $\lz I+T$
is invertible (from $D(T)$ into $\mathcal{H}$) and $(\lz I+T)^{-1}$
is bounded on $\mathcal{H}$.
Moreover, for all $\lz\in(0,\,\fz)$ and $f\in \mathcal{H}$,
\begin{eqnarray*}
\lf\|\lz\lf(\lz I+T\r)^{-1}(f)\r\|_{\mathcal{H}}\le \|f\|_{\mathcal{H}}.
\end{eqnarray*}
\end{proposition}

For $\mathfrak{a}_0$ defined as in \eqref{form0}, from the fact that
$\dot W^{m,\,2}(\rn)$ is dense in $L^2(\rn)$ and the Ellipticity condition
$(\mathcal{E}_0)$, we deduce that $\mathfrak{a}_0$ is a densely defined,
accretive and continuous sesquilinear form.

Moreover, let $W^{m,\,2}(\rn)$
be the $m$-\emph{order inhomogeneous Sobolev space}
equipped with the usual \emph{norm}
\begin{eqnarray}\label{eqn inhoSovNorm}
\|f\|_{W^{m,\,2}(\rn)}:
=\lf[\sum_{0\le |\az|\le m}\|\pat^\az f\|_{L^2(\rn)}^2\r]^{1/2}.
\end{eqnarray}
For all $f\in D(\mathfrak{a}_0)$, by the Ellipticity condition
$(\mathcal{E}_0)$ and Plancherel's theorem,
it is easy to see that
\begin{eqnarray*}
\|f\|_{\mathfrak{a}_0}:=
\sqrt{\Re e\,\lf(\mathfrak{a}_0(f,\,f)\r)+\|f\|_{L^2(\rn)}^2}
\sim \|f\|_{W^{m,\,2}(\rn)}.
\end{eqnarray*}
This, combined with the fact that $W^{m,\,2}(\rn)$ is a Banach space, further implies
that
$$(\dot W^{m,\,2}(\rn),\,\|\cdot\|_{\mathfrak{a}_0})$$
is complete.
Thus, $\mathfrak{a}_0$ is closed. Using Proposition \ref{pro Ouhabaz},
we know that there exists a densely defined operator
$L$ in $L^2(\rn)$ associated with $\mathfrak{a}_0$, which is formally
written as in \eqref{L homo}.

Let $\omega\in[0,\,\pi/2)$. Recall  that an operator
$T$ in the Hilbert space $\mathcal{H}$ is said to be
$m$-$\omega$-{\it accretive} (or \emph{maximal $\omega$-accretive}) if
\begin{enumerate}
\item[(i)] the range of the operator $T+I$, $R(T+I)$, is dense in $\mathcal{H}$;
\item[(ii)] for all $u\in D(T)$, $|\arg(T(u),\,u)_{\mathcal{H}}|\le\omega$,
\end{enumerate}
where $\arg (T(u),\,u)_{\mathcal{H}}$ denotes the
{\it argument} of $(T(u),\,u)_{\mathcal{H}}$; see \cite[p.\,173]{ha06}.

It is known that, by \cite[Proposition 7.1.1]{ha06},
every closed m-$\omega$-accretive operator is of
\emph{type} $\omega$ in $L^2(\rn)$, namely, the
spectrum of $T$, $\sz(T)$, is contained in the \emph{sector}
\begin{eqnarray*}
S_\omega:=\{z\in\cc:\
|\arg z| \le \omega\}
\end{eqnarray*}
and, for each
$\tz\in (\omega,\,\pi)$, there exists a nonnegative constant $C$
such that, for all $z\in\cc\setminus S_\tz$,
$\|(T-zI)^{-1}\|_{\cl(L^2(\rn))}\le C|z|^{-1}$, where
$\|S\|_{\cl(\mathcal H)}$ denotes the {\it operator norm}
of the linear operator $S$ on the normed linear space $\mathcal H$.

Moreover, by \cite{ha06}, we know  that, if $T$ is of type $\omega$, then
$-T$ generates a semigroup $\{e^{-tT}\}_{t>0}$, which can be extended
to a bounded holomorphic semigroup $\{e^{-zT}\}_{z\in
S_{\pi/2-\omega}^0}$ in the \emph{open sector}
\begin{eqnarray*}
S_{\pi/2-\omega}^0:=\{z\in\cc\setminus \{0\}:\
|\arg{z}|<\pi/2-\omega\}.
\end{eqnarray*}

Recall that, by the Ellipticity condition $(\mathcal{E}_0)$, we know that $L$ is an
m-$\arctan\frac{\Lambda}{\lz}$-accretive operator in $L^2(\rn)$.
Thus, $-L$ generates a bounded holomorphic semigroup in the open sector
$S_{\pi/2-\arctan\frac{\Lambda}{\lz}}^0$.

The following $L^2(\rn)$ off-diagonal estimates of $\{e^{-zL}\}_{z\in
S_{\pi/2-\arctan\frac{\Lambda}{\lz}}^0}$ are well known (see, for example,
\cite[p.\,66]{au07}, \cite[Theorem 3.2]{ddy12} or \cite[Lemma 3.1]{cy12}).

\begin{proposition}\label{pro Gaffney estiamtes L0}
Let $L$ be as in \eqref{L homo}  and satisfy the
Ellipticity condition $(\mathcal{E}_0)$, and let
$\omega:=\arctan\frac{\Lambda_0}{\lz_0}$, where $\blz_0$ and $\lz_0$
are as in the Ellipticity condition $(\mathcal{E}_0)$.
Then, for all $\ell\in(0,\,1)$, $ k\in\zz_+$, the family of
operators, $\{(zL)^ke^{-zL}\}_{z\in
S_{\ell(\frac{\pi}{2}-\omega)}^0}$, satisfies the $m$-Davies-Gaffney
estimates in $z$. That is, there exist positive
constants $C$ and $\wz C$ such that, for all $f\in L^2(\rn)$
supported in $E$ and $z\in S_{\ell(\frac{\pi}{2}-\omega)}^0$,
\begin{eqnarray*}
\|(z L)^ke^{-zL}(f)\|_{L^2(F)}\le C
\exp\lf\{-\wz C\frac{\lf[\dist(E,\,F)\r]^{2m/(2m-1)}}{|z|^{1/(2m-1)}}
\r\}\|f\|_{L^2(E)}.
\end{eqnarray*}
\end{proposition}

We now consider the $L^p(\rn)$ theory of $\{e^{-tL}\}_{t>0}$.
Let $(p_-(L),\, p_+(L))$ be the \emph{maximal interval} of exponents
$p\in[1,\,\fz]$ such that $\{e^{-tL}\}_{t>0}$ is bounded on $L^p(\rn)$.
Let $(q_-(L),\, q_+(L))$ be the \emph{maximal interval} of exponents
$q\in[1,\,\fz]$ such that $\{\sqrt{t}\nabla^me^{-tL}\}_{t>0}$ is bounded on $L^q(\rn)$.
By \cite[pp.\,66-67]{au07} and \cite[Theorem 3.2]{ddy12}, we have the following
conclusion.

\begin{proposition}[\cite{au07,ddy12}]\label{pro Lp semiL0}
Let $L$ be as in \eqref{L homo} and satisfy the
Ellipticity condition $(\mathcal{E}_0)$. Then

\begin{itemize}

\item[{\rm (i)}]
\begin{eqnarray*}
\begin{cases}
(p_-(L),\, p_+(L))=(1,\,\fz),\ \, &\text{when}\ n\le 2m,\\ \\
\lf[\dfrac{n}{n+2m},\,\dfrac{n}{n-2m}\r]\subset (p_-(L),\, p_+(L)),\ \,
&\text{when}\ n>2m.
\end{cases}
\end{eqnarray*}
\item[{\rm (ii)}]
$q_-(L)=p_-(L)$, $q_+(L)>2$ and $p_+(L)\ge (q_+(L))^{*m}$, where, for any $q\in(1,\,\fz)$,
\begin{eqnarray*}
q^*:=
\begin{cases}
\dfrac{np}{n-p},\ \, &\text{when}\ p< n,\\ \\
\fz,\ \,&\text{when}\ p\ge n
\end{cases}
\end{eqnarray*}
denotes the Sobolev exponent of $q$ and $q^{*m}$ means the $m$-th iteration of the operation
$q\mapsto q^*$.

\item[{\rm (iii)}] For all $k\in\zz_+$
and $p_-(L)<p\le q<p_+(L)$,
the family $\{(tL)^ke^{-tL}\}_{t>0}$ of operators
satisfies the following $m$-$L^p$-$L^q$ off-diagonal estimates:
there exist positive constants $C$ and $\wz C$ such that, for any
closed sets $E$, $F$ in $\rn$, $t\in(0,\,\fz)$ and $f\in L^2(\rn)\cap
L^p(\rn)$ supported in $E$,
\begin{eqnarray*}
\lf\|(t L)^ke^{-tL}(f)\r\|_{L^q(F)}\le
Ct^{\frac{n}{2m}(\frac{1}{q}-\frac{1}{p})}
\exp\lf\{-\wz C\frac{\lf[d(E,\,F)\r]^{\frac{2m}{2m-1}}}{t^{\frac{1}
{2m-1}}}\r\}\|f\|_{L^p(\rn)}.
\end{eqnarray*}

\item[{\rm (iv)}]
For all $p_-(L)<p\le q<q_+(L)$,
the family $\{(t^{1/(2m)}\nabla)^ke^{-tL}\}_{t>0}$ of operators
satisfies the following $m$-$L^p$-$L^q$ off-diagonal estimates:
there exist positive constants $C$ and $\wz C$ such that, for any
closed sets $E$, $F$ in $\rn$, $t\in(0,\,\fz)$ and $f\in L^2(\rn)\cap
L^p(\rn)$ supported in $E$,
\begin{eqnarray*}
\lf\|\lf(t^{1/(2m)}\nabla\r)^ke^{-tL}(f)\r\|_{L^q(F)}\le
Ct^{\frac{n}{2m}(\frac{1}{q}-\frac{1}{p})}
\exp\lf\{-\wz C\frac{\lf[d(E,\,F)\r]^{\frac{2m}{2m-1}}}{t^{\frac{1}
{2m-1}}}\r\}\|f\|_{L^p(\rn)}.
\end{eqnarray*}
\end{itemize}
\end{proposition}

Finally, we recall some results on the square root of $L$. Let $L$ be defined as
in \eqref{L homo}. It is known that $L$ is one-to-one and m-$\omega$-accretive.
By \cite[p.\,8]{AT98}, we know that $L$ has a bounded holomorphic functional
calculus in $L^2(\rn)$. Thus, its square root $L^{1/2}$ is well defined on $L^2(\rn)$.

Auscher et al.
proved the following result on Kato's square root problem of $L^{1/2}$
(see \cite[Theorem 1.1]{ahmt01}).

\begin{proposition}[\cite{ahmt01}]\label{pro kato root L0}
Let $L$ be as in \eqref{L homo} and satisfy the Ellipticity condition
$(\mathcal{E}_0)$. The square root of $L$ has a domain equal to
the Sobolev space $W^{m,\,2}(\rn)$ defined as in \eqref{eqn inhoSovNorm}.
Moreover, there exists a positive constant $C$ such that, for all
$f\in W^{m,\,2}(\rn)$,
\begin{eqnarray*}
\frac{1}{C}\lf\|\sqrt{L}(f)\r\|_{L^2(\rn)}\le \|\nabla^m f\|_{L^2(\rn)}\le C
\lf\|\sqrt{L}(f)\r\|_{L^2(\rn)}.
\end{eqnarray*}
\end{proposition}

Proposition \ref{pro kato root L0} implies immediately that the Riesz transform
$\nabla^mL^{-1/2}$ associated with $L$ is bounded on $L^2(\rn)$.
Moreover, Auscher proved the following boundedness of $\nabla^mL^{-1/2}$
on $L^p(\rn)$ (see \cite[p.\,68]{au07}).

\begin{proposition}[\cite{au07}]\label{pro bd RieszTFL0 Auscher}
Let $L$ be as in \eqref{L homo} and satisfy the Ellipticity condition
$(\mathcal{E}_0)$. Then, for all $p\in(q_-(L),\,q_+(L))$,
$\nabla^mL^{-1/2}$ is bounded on $L^p(\rn)$.
\end{proposition}

We also refer the reader to \cite[Theorem 1.2]{bk04} for a related result on the
boundedness of $\nabla^mL^{-1/2}$.

\subsection{The Hardy space $H_{L}^p(\rn)$}\label{s2.2}

\hskip\parindent  Let ${L}$ be the homogeneous higher order operator defined
as in \eqref{L homo} that satisfies the Ellipticity condition $(\mathcal{E}_0)$.
Let $H_L^p(\rn)$ be the Hardy space associated with $L$ defined as in Definition
\ref{Hardy space aoociated with L0 df}.
In this subsection, we give some real-variable properties of $H_L^p(\rn)$
for $p\in(0,\,\fz)$. Our first result is the following
complex interpolation of $H_{L}^p(\rn)$. Recall (\cite{cms85}) that,
for all $p\in(0,\,\fz)$,  a function $F$ on $\mathbb{R}^{n+1}_+$ is said to be in
the {\it tent space} $T^p(\mathbb{R}^{n+1}_+)$, if
$\|F\|_{T^p(\mathbb{R}^{n+1}_+)}=:\|\mathcal{A}(F)
\|_{L^p(\rn)}<\fz$, where
\begin{eqnarray}\label{meqnA}
\mathcal{A}(F)(x):=\lf\{\iint_{\Gamma(x)}\lf|F(y,\,t)\r|^2
\frac{dy\,dt}{t^{n+1}}\r\}^{\frac{1}{2}},
\end{eqnarray}
with $\Gamma(x)$ for all $x\in\rn$ as in \eqref{1.x1},
denotes the $\mathcal{A}$-{\it functional} of $F$ (see \cite{cms85}
for more properties of tent spaces).

\begin{proposition}\label{pro interpolation Hardy L0}
Let $L$ be as in \eqref{L homo} and satisfy the Ellipticity condition
$(\mathcal{E}_0)$.  Then, for each $\tz\in(0,\,1)$
and $0<p_1<p_2<\fz$,
\begin{eqnarray*}
\lf[H_{L}^{p_1}(\rn),\,H_{L}^{p_2}(\rn)\r]_\tz=H_{L}^p(\rn),
\end{eqnarray*}
where $p$ satisfies $\frac{1}{p}=\frac{1-\tz}{p_1}+\frac{\tz}{p_2}$
and $[\cdot,\,\cdot]_{\tz}$ denotes the complex interpolation (see,
for example, \cite[Section 7]{KMM07}).
\end{proposition}

\begin{proof}
The proof of Proposition \ref{pro interpolation Hardy L0} is a consequence of
the complex interpolation of tent spaces $T^p(\rr^{n+1}_+)$ and the fact that
$H_{L}^{p}(\rn)$ is a retract of $T^p(\rr^{n+1}_+)$ (see \cite[Lemma 4.20]{hmm}
for more details in the case when $m=1$), the details being omitted.
\end{proof}

For $H_{L}^p(\rn)$ with $p\in(0,\,1]$,
one of its most useful properties is its molecular
characterization. To state it, we first recall the following notion of
$(p,\,2,\,M,\,q)_{L}$-molecules.

\begin{definition}\label{def molecule L0}
Let ${L}$ be as in \eqref{L homo} and satisfy the Ellipticity condition $(\mathcal{E}_0)$,
$p\in(0,\,1]$, $\epsilon\in(0,\,\fz)$ and
$M\in\nn$.  A function $\az\in L^2(\rn)$ is called a
$(p,\,2,\,M,\,\epsilon)_{{L}}$-{\it molecule}
if there exists a ball $B\subset\rn$ such that,
for each $\ell\in\{0,\,\ldots,\,M\}$, $\az$
belongs to the range of $L^\ell$ in $L^2(\rn)$
and, for all $i\in\zz_+$ and $\ell\in\{0,\,\ldots,\,M\}$,
\begin{eqnarray*}
\lf\|\lf(r_B^{2m}{L}\r)^{-\ell} (\az)\r\|_{L^2(S_i(B))}
\le2^{-i\epsilon} |2^iB|^{\frac{1}{2}-\frac{1}{p}}.
\end{eqnarray*}

Assume that $\{\az_j\}_{j}$ is a sequence of
$(p,\,2,\,M,\,\epsilon)_{{L}}$-molecules and
$\{\lz_j\}_{j}\in l^p$. For
any $f\in L^2(\rn)$, if $f=\sum_{j} \lz_j \az_j$ in $L^2(\rn)$,
then $\sum_{j} \lz_j\az_j$ is called a {\it molecular}
$(p,\,2,\,M,\,\epsilon)_{{L}}$-{\it representation} of $f$.
\end{definition}

\begin{definition}\label{molecular Hardy space associated with L2+k}
Let ${L}$ be as in \eqref{L homo} and satisfy the Ellipticity
condition $(\mathcal{E}_0)$, and let
$p\in(0,\,1]$, $\epsilon\in(0,\,\fz)$ and
$M\in\nn$.  The {\it molecular Hardy space $H_{{L},\,\mol,\,M,\,\epsilon}^p(\rn)$}
is defined as the completion of the space
$$\mathbb{H}_{{L},\,\mol,\,M,\,\epsilon}^p(\rn):=
\{f\in L^2(\rn):\ f\ \text{ has a molecular}\
(p,\,2,\,M,\,\epsilon)_{L}\text{-representation}\}$$  with respect to the
\emph{quasi-norm}
\begin{eqnarray*}
\|f\|_{H_{{L},\,\mol,\,M,\,\epsilon}^p(\rn)} :=&&\inf\lf\{\lf(\dsum_{j} |\lz_j|^p\r)^{1/p}:\
f=\dsum_{j} \lz_j\az_j \ \text{is a molecular}\r.\\
&&\hspace{5cm}(p,\, 2,\,M,\,\epsilon)_{{L}}\text{-representation}\Bigg\},
\end{eqnarray*}
where the infimum is taken over all the molecular $(p,\,2,\,
M,\,\epsilon)_{{L}}$-representations of $f$ as above.
\end{definition}

\begin{theorem}[\cite{cy12,ddy12}]
\label{molecular characterization for Hardy space L0}
Let $L$ be as in \eqref{L homo} and satisfy the Ellipticity condition $(\mathcal{E}_0)$,
$p\in(0,\,1]$, $\epsilon\in(0,\,\fz)$ and
$M\in\nn$ such that $M>\frac{n}{2m}(\frac{1}{p}-\frac{1}{2})$. Then
$H_{L}^p(\rn)= H_{L,\,\mol,\,M,\,\epsilon}^p(\rn)$
with equivalent quasi-norms.
\end{theorem}

For more characterizations of $H_{L}^p(\rn)$ with $p\in(0,\,1]$,
we refer the reader to \cite{cy12,ddy12}.

We now study the relationship between $H_{L}^p(\rn)$ and
the Lebesgue space $L^p(\rn)$.

\begin{lemma}\label{lem relation Hardy Lebesgue L0}
Let $L$ be as in \eqref{L homo} and satisfy the Ellipticity condition $(\mathcal{E}_0)$,
and let $p_-(L)$ and $p_+(L)$ be as in Proposition \ref{pro Lp semiL0}.
Then, for all $p\in(p_-(L),\,p_+(L))$,
$H_{L}^p(\rn)=L^p(\rn)$
with equivalent norms.
\end{lemma}

\begin{proof}
We prove Lemma \ref{lem relation Hardy Lebesgue L0} by borrowing some
ideas from the proof of \cite[Proposition 9.1(v)]{hmm}.
First, from \cite[Propositions 2.10 and 2.13]{ccyy13}, it follows that,
for all $p\in(p_-(L),\,p_+(L))$, $S_{L}$ is bounded on $L^p(\rn)$.
This, together with Definition \ref{Hardy space aoociated with L0 df},
shows that, for all $p\in(p_-(L),\,2]$ and $f\in L^2(\rn)\cap L^p(\rn)$,
\begin{eqnarray*}
\|f\|_{H_{L}^p(\rn)}:=\|S_{L}(f)\|_{L^p(\rn)}\ls \|f\|_{L^p(\rn)},
\end{eqnarray*}
which immediately implies that, for all $p\in(p_-(L),\,2]$,
$$\lf(L^2(\rn)\cap L^p(\rn)\r)\subset \lf(L^2(\rn)\cap H_{L}^p(\rn)\r).$$

On the other hand, recall the following Calder\'on reproducing
formula for $L$ (since $L$ has a bounded holomorphic functional calculus
in $L^2(\rn)$): for all $g\in L^2(\rn)$,
\begin{eqnarray}\label{eqn Calderon repro L0}
g=\wz C\dint_0^\fz(t^{2m} L)^{M+2}e^{-2t^{2m}L}(g)\,\frac{dt}{t}
=:\wz C\pi_{L,\,M}\circ Q_{L,\,1,\,t}(g)
\end{eqnarray}
holds true in $L^2(\rn)$,
where  $\wz C$ is a positive constant such that $\wz C \int_0^\fz
t^{2m(M+2)}e^{-2t^{2m}}\,\frac{dt}{t}=1$, $M\in\nn$ is sufficiently large,
$$\pi_{L,\,M}:=\dint_0^\fz(t^{2m} L)^{M+1}e^{-t^{2m}L}\,\frac{dt}{t}$$
and,  for all $k\in\nn$,
$$Q_{L,\,k,\,t}:=(t^{2m} L)^ke^{-t^{2m}L}.$$

Thus, for $p\in(p_-(L),\,2]$, if $f\in L^2(\rn)\cap H_{L}^p(\rn)$, then,
for all $g\in L^2(\rn)\cap L^{p'}(\rn)$, by \eqref{eqn Calderon repro L0},
duality between $T^p(\rr^{n+1}_+)$ and $T^{p'}(\rr^{n+1}_+)$ with $1/p+1/p'=1$,
and H\"older's inequality, we see that
\begin{eqnarray*}
\lf|\dint_\rn f(x)\overline{g(x)}\,dx\r|&&=\wz C\lf|\dint_\rn \pi_{L,\,M}\circ
Q_{L,\,1,\,t}(f)(x)\overline{g(x)}\,dx\r|\\
&&=\wz C\lf|\iint_{\rr^{n+1}_+}
Q_{L,\,1,\,t}(f)(x)\overline{Q_{L^*,\,M+1,\,t}(g)(x)}\,\frac{dx\,dt}{t}\r|\\
&&\ls \lf\|Q_{L,\,1,\,t}(f)\r\|_{T^p(\rr^{n+1}_+)}
\lf\|Q_{L^*,\,M+1,\,t}(g)\r\|_{T^{p'}(\rr^{n+1}_+)}\\
&&\sim \lf\|f\r\|_{H_{L}^p(\rn)}
\lf\|\lf\{\iint_{\Gamma(\cdot)}\lf|\lf(t^{2m}L^*\r)^{M+1}
e^{-t^{2m}L^*}(g)(y)\r|^2
\frac{dy\,dt}{t^{n+1}}\r\}^{1/2}\r\|_{L^{p'}(\rn)},
\end{eqnarray*}
where $T^p(\rr^{n+1}_+)$ denotes the tent space and
$L^*$ the adjoint operator of $L$ in $L^2(\rn)$.
Since $p'\in [2,\,p_+(L^*))$,
similar to the boundedness of $S_{L}$ on $L^{p'}(\rn)$ for all $p'\in (p_-(L),\,p_+(L))$,
we have
$$\lf\|Q_{L^*,\,M+1,\,t}(g)\r\|_{T^{p'}(\rr^{n+1}_+)}\ls \|g\|_{L^{p'}(\rn)},$$
which, combined with the arbitrariness of $g$, implies that $f\in L^p(\rn)$ and
\begin{eqnarray*}
\|f\|_{L^p(\rn)}\ls \lf\|f\r\|_{H_{L}^p(\rn)},
\end{eqnarray*}
and hence $(L^2(\rn)\cap H_{L}^p(\rn))\subset (L^2(\rn)\cap L^p(\rn))$.
By density, this finishes the proof of Lemma
\ref{lem relation Hardy Lebesgue L0} for $p\in(p_-(L),\,2]$. The case
$p\in[2,\,p_+(L))$ follows from Definition \ref{Hardy space aoociated with L0 df}
and a dual argument, the details being omitted. This finishes the proof of Lemma
\ref{lem relation Hardy Lebesgue L0}.
\end{proof}

Combining Lemma \ref{lem relation Hardy Lebesgue L0},
Propositions \ref{pro bd RieszTFL0 Auscher} and \ref{pro interpolation Hardy L0},
together with the fact that $\nabla^mL^{-1/2}$ is bounded from $H_{L}^p(\rn)$ to
the classical Hardy space $H^p(\rn)$ for all $p\in(\frac{n}{n+m},\,1]$
(see \cite[Theorem 6.2]{cy12}), we conclude the following proposition,
the details being omitted.

\begin{proposition}\label{pro bd RieszTF Hardy space L0}
Let $L$ be as in \eqref{L homo} and satisfy the Ellipticity condition $(\mathcal{E}_0)$.
Then, for all $p\in(\frac{n}{n+m},\,q_+(L))$, $\nabla^mL^{-1/2}$ is bounded from
$H_{L}^p(\rn)$ to $H^p(\rn)$.
\end{proposition}

Now, we establish the generalized square function characterization of $H_L^p(\rn)$,
which is available in \cite{hmm} for $m=1$ and $p\in(0,\,\fz)$
and in \cite{cy12} for $m\in\nn$ and $p\in(0,\,1]$.
Let $p\in(0,\,\fz)$, $\omega\in[0,\,\pi/2)$ be the type of $L$,
$\az\in(0,\,\fz)$, $\bz\in(\frac{n}{2m}(\max\{\frac{1}{p},\,1\}-\frac{1}{2}),\,\fz)$ and
$\psi\in\Psi_{\az,\bz}(S_{\mu}^0)$ with $\mu\in(\omega,\,\pi/2)$, where
$$S_\mu^0:=\lf\{z\in\cc\setminus\{0\}:\ |\arg z|<\mu\r\}$$
and
\begin{eqnarray*}
\Psi_{\az,\bz}(S_\mu^0)&&:=\Big\{f\ \text{is analytic on}\ S_\mu^0:\
\text{there exists a positive constant $C$ such that} \  \\
&& \quad\qquad \lf.|f(\xi)|\le C
\min\{|\xi|^\az,\,|\xi|^{-\bz}\}\ \text{for all}\ \xi\in
S_\mu^0\r\}.
\end{eqnarray*}
For all $f\in L^2(\rn)$ and $(x,\,t)\in \mathbb{R}^{n+1}_{+}$,
define the {\it operator} $Q_{\psi,L}(f)$ by
\begin{eqnarray*}
Q_{\psi,L}(f)(x,\,t):=\psi\lf(t^{2m}L\r)(f)(x).
\end{eqnarray*}

\begin{definition}\label{def SFHS}
Let $p\in(0,\,\fz)$, $L$ be as in \eqref{L homo} and
satisfy the Ellipticity condition $(\mathcal{E}_0)$, $\az\in(0,\,\fz)$,
$\bz\in(\frac{n}{2m}(\max\{\frac{1}{p},\,1\}-\frac{1}{2}),\,\fz)$,
$\mu\in(\omega,\,\pi/2)$ and
\begin{eqnarray*}
\psi\in
\begin{cases}
\Psi_{\az,\bz}(S_{\mu}^0),\ \, &\text{when}\ p\in(0,\,2],\\
\Psi_{\bz,\az}(S_{\mu}^0),\ \,&\text{when}\ p\in(2,\,\fz).
\end{cases}
\end{eqnarray*}
The {\it generalized square
function Hardy space} $H_{\psi,L}^p(\rn)$ is defined as the
completion of the space
\begin{eqnarray*}
\mathbb{H}_{\psi,L}^p(\rn):=\lf\{f\in L^2(\rn):\ Q_{\psi,L}(f)\in
T^p(\mathbb{R}^{n+1}_{+})\r\}
\end{eqnarray*}
with respect to the {\it quasi-norm} $\|f\|_{H_{\psi,L}^p(\rn)}
:=\|Q_{\psi,L}(f)\|_{T^p(\mathbb{R}^{n+1}_{+})}$.
\end{definition}

The following result establishes the generalized square
function characterization of $H_L^p(\rn)$ for $p\in(0,\,\fz)$.

\begin{proposition}\label{pro GSFC}
Let $p\in(0,\,\fz)$, $L$ be as in \eqref{L homo} and
satisfy the Ellipticity condition $(\mathcal{E}_0)$, $\az\in(0,\,\fz)$,
$\bz\in(\frac{n}{2m}(\max\{\frac{1}{p},\,1\}-\frac{1}{2}),\,\fz)$,
$\mu\in(\omega,\,\pi/2)$ and
\begin{eqnarray*}
\psi\in
\begin{cases}
\Psi_{\az,\bz}(S_{\mu}^0),\ \, &\text{when}\ p\in(0,\,2],\\
\Psi_{\bz,\az}(S_{\mu}^0),\ \,&\text{when}\ p\in(2,\,\fz).
\end{cases}
\end{eqnarray*}
Then the Hardy space
$H_L^p(\rn)=H_{\psi,L}^p(\rn)$ with equivalent quasi-norms.
\end{proposition}

\begin{proof}
If $p\in(0,\,1]$, Proposition \ref{pro GSFC} is just
\cite[Theorem 5.2]{cy12}, where $\bz\in(\frac{n}{2m}(\frac{1}{p}-\frac{1}{2}),\,\fz)$
is needed to guarantee $H_L^p(\rn)\subset H_{\psi,L}^p(\rn)$, via an application of the Calder\'on
reproducing formula.

If $p\in(1,\,\fz)$ and $m=1$,
Proposition \ref{pro GSFC} is just \cite[Corollary 4.17]{hmm}, where
$\bz\in(\frac{n}{4},\,\fz)$ is used to guarantee
$H_L^p(\rn)\subset H_{\psi,L}^p(\rn)$, via an application of
the Calder\'on reproducing formula. If $p\in(1,\,\fz)$ and $m\in\nn\cap[2,\fz)$,
an argument similar to that used in the proof of \cite[Corollary 4.17]{hmm},
together with an application of the Calder\'on
reproducing formula, also gives us the desired conclusion of
Proposition \ref{pro GSFC}, where we need $\bz\in(\frac{n}{4m},\,\fz)$ to guarantee
$H_L^p(\rn)\subset H_{\psi,L}^p(\rn)$, the details being omitted,
which completes the proof of Proposition \ref{pro GSFC}.
\end{proof}

\section{Proofs of Theorems \ref{non-tangential heat maximal functiona characterization},
\ref{non-tangential heat maximal functiona characterization gradient} and
\ref{thm SFLAF Hardy}, and
Proposition \ref{pro NTMFC2 L0}}\label{s3}

\hskip\parindent In this section, we give the proofs of Theorems
\ref{non-tangential heat maximal functiona characterization},
\ref{non-tangential heat maximal functiona characterization gradient} and
\ref{thm SFLAF Hardy}, and
Proposition \ref{pro NTMFC2 L0}.
To this end, we first establish the following parabolic
Caccioppoli's inequalities, resonating with \cite[Lemma 2.8]{hm09} and,
in a different way, with \cite[Proposition 40]{aq00}.

\begin{proposition}\label{caccioppoli's inequality}
Let $L$ be as in \eqref{L homo} and satisfy the Ellipticity condition
$(\mathcal{E}_0)$, and let $f\in L^2(\rn)$, $t\in(0,\,\fz)$ and $u(x,\,t)
:=e^{-t^{2m}{L}}(f)(x)$ for all $x\in\rn$.
For all $\epsilon\in(0,\,\fz)$, there exist positive constants
$C_{(\epsilon)}$, depending on $\epsilon$, and $C$ such that,
for all $x_0\in\rn$, $r\in(0,\,\fz)$ and $t_0\in(3r,\,\fz)$,
\begin{eqnarray}\label{caccioppoli's inequality1 with epsilon}
&&\dint_{t_0-r}^{t_0+r} \dint_{B(x_0,\,r)}\lf|\nabla^m u(x,\,t)\r|^2
\,dx\,dt\nonumber\\
&&\hs\le \epsilon
\dint_{t_0-2r}^{t_0+2r} \dint_{B(x_0,\,2r)}\lf|\nabla^m u(x,\,t)\r|^2
\,dx\,dt  +\frac{C_{(\epsilon)}}{r^{2m}}
\dint_{t_0-2r}^{t_0+2r} \dint_{B(x_0,\,2r)}\lf|u(x,\,t)\r|^2
\,dx\,dt
\end{eqnarray}
and
\begin{eqnarray}\label{caccioppoli's inequality2 with epsilon}
&&\dint_{t_0-r}^{t_0+r} \dint_{B(x_0,\,r)}\lf|\nabla^m u(x,\,t)\r|^2
\,dx\,dt\le \dsum_{j=0}^{m-1}\frac{C}{r^{2(m-j)}}
\dint_{t_0-2r}^{t_0+2r} \dint_{B(x_0,\,2r)}\lf|\nabla^{j} u(x,\,t)\r|^2
\,dx\,dt,
\end{eqnarray}
where $C_{(\epsilon)}$ and $C$ are independent of $f$.
\end{proposition}

\begin{proof}
We first prove \eqref{caccioppoli's inequality2 with epsilon}.
To this end, we
introduce two smooth cut-off functions. Let $\eta\in C_c^\fz
(B(x_0,\,2r))$ satisfy $0\le\eta\le 1$,
$\eta\equiv1$ on $B(x_0,\,r)$ and,
for all $k\in\{0,\,\ldots,\,m\}$,
$$\lf\|\nabla^k\eta\r\|_{L^\fz(\rn)}\ls r^{-k}.$$
Let $\gz\in C_c^\fz (t_0-2r,\,t_0+2r)$ satisfy $0\le \gz\le 1$,
$\gz\equiv 1$ on $(t_0-r,\,t_0+r)$
and 
$$\lf\|\pat_t\gz\r\|_{L^\fz(\rr)}\ls \frac{1}{r}.$$
By the properties of $\eta$ and $\gz$, and the
Ellipticity condition $(\mathcal{E}_0)$, we first  write
\begin{eqnarray}\label{0.1}
&&\dint_{t_0-r}^{t_0+r}\dint_{B(x_0,\,r)}\lf|\nabla^m u(x,\,t)\r|^2\,dx\,dt\nonumber\\
&&\nonumber\hs\le\dint_{t_0-2r}^{t_0+2r}
\dint_{\rn}\lf|\nabla^m (u\eta^m)(x,\,t)\r|^2\,dx\,\gz(t)\,dt\\
&&\hs\le\frac{1}{\lz_0}\Re e\lf\{\dint_{t_0-2r}^{t_0+2r}
\dint_{\rn}A(x)\nabla^m (u\eta^m)(x,\,t)\overline{\nabla^m (u\eta^m)}(x,\,t)\,dx\,\gz(t)\,dt\r\}=:\frac{1}{\lz_0}\mathcal{A},
\end{eqnarray}
where
\begin{eqnarray}\label{a-in}
A(x):=\{a_{\az,\,\bz}(x)\}_{|\az|=m=|\bz|}\ \
\text{for all}\ x\in\rn
\end{eqnarray}
is a (properly arranged) coefficient matrix of $L$ so that,
for all $f,\,g\in \dot W^{m,2}(\rn)$ and $x\in\rn$,
\begin{eqnarray*}
A(x) \nabla^m f(x)
\overline{\nabla^m g(x)}:=\dsum_{|\az|=m=|\bz|}
a_{\az,\,\bz}(x)\pat^\bz f(x)\overline{\pat^\az g(x)}.
\end{eqnarray*}

To bound $\mathcal{A}$, let
\begin{eqnarray*}
\mathcal{B}:=\Re e\lf\{\dint_{t_0-2r}^{t_0+2r}
\dint_{\rn}A(x)\nabla^m (u)(x,\,t)\overline{\nabla^m (u\eta^{2m})}(x,\,t)\,dx\,\gz(t)\,dt\r\}.
\end{eqnarray*}
We first bound $\mathcal{B}$. For all $(x,\,\wz t)\in\rr^{n+1}_+$,
let $F(x,\,\wz t):=e^{-\wz tL}(f)(x)\overline{e^{-\wz tL}(f)(x)}[\eta(x)]^{2m}$.
Using $\pat_{\wz t} e^{-\wz tL}=-Le^{-\wz tL}$ and
$$\overline{Le^{-\wz tL}(f)(x)\overline{e^{-\wz tL}(f)(x)}[\eta(x)]^{2m}}
=e^{-\wz tL}(f)(x)\overline{Le^{-\wz tL}(f)(x)}[\eta(x)]^{2m},$$
we know that
\begin{eqnarray*}
\pat_{\wz t} F(x,\,\wz t)&&=-Le^{-\wz tL}(f)(x)\overline{e^{-\wz tL}(f)(x)}[\eta(x)]^{2m}-
e^{-\wz tL}(f)(x)\overline{Le^{-\wz tL}(f)(x)}[\eta(x)]^{2m}\\
&&=-2 \Re e\lf\{Le^{-\wz tL}(f)(x)\overline{e^{-\wz tL}(f)(x)}[\eta(x)]^{2m}\r\},
\end{eqnarray*}
which, together with integration by parts, and the definition of the cut-off function $\gz$,
shows that
\begin{eqnarray}\label{eqn estJ}
\mathcal{B}&&=\frac{1}{\lz_1}\Re e\lf[\dint_{t_0-2r}^{t_0+2r}
\lf\{\dint_{\rn}Le^{-t^{2m}L}(f)(x) \overline{e^{-t^{2m}L}(f)(x)\lf[\eta(x)\r]^{2m}}
\,dx\r\}\gz(t)\,dt\r]\nonumber\\
&&\nonumber=\frac{1}{2m\lz_1}\Re e\lf\{\dint_{(t_0-2r)^{2m}}^{(t_0+2r)^{2m}}
\lf[\dint_{\rn}Le^{-\wz{t}L}(f)(x) \overline{e^{-\wz{t}L}(f)(x)\lf[\eta(x)\r]^{2m}}
\,dx\r]\gz\lf(\wz{t}^{\frac{1}{2m}}\r)
\wz{t}^{\frac{1}{2m}-1}\,d\wz{t}\r\}
\\&&\nonumber=-\frac{1}{4m\lz_1}\dint_{(t_0-2r)^{2m}}^{(t_0+2r)^{2m}}
\pat_{\wz{t}}\lf(\dint_{B(x_0,\,2r)}F(x,\,\wz{t})\,dx\r)
\gz\lf(\wz{t}^{\frac{1}{2m}}\r)\wz{t}^{\frac{1}{2m}-1}\,d\wz{t}\\
&&\nonumber=\frac{1}{4m\lz_1}\dint_{(t_0-2r)^{2m}}^{(t_0+2r)^{2m}}
\lf[\dint_{B(x_0,\,2r)}F(x,\,\wz{t})\,dx\r] \pat_{\wz{t}}
\lf(\gz\lf(\wz{t}^{\frac{1}{2m}} \r) \wz{t}^{\frac{1}{2m}-1}\r)\,d\wz{t}\\
&&\nonumber\hs-\frac{1}{4m\lz_1}\dint_{(t_0-2r)^{2m}}^{(t_0+2r)^{2m}}
\pat_{\wz{t}} \lf(\lf[\dint_{B(x_0,\,2r)} F(x,\,\wz{t})\,dx\r]\gz
\lf(\wz{t}^{\frac{1}{2m}}\r)\wz{t}^{\frac{1}{2m}-1}\r)\,d\wz{t}\\
 &&=\frac{1}{4m\lz_1}\dint_{(t_0-2r)^{2m}}^{(t_0+2r)^{2m}}
 \lf[\dint_{B(x_0,\,2r)}F(x,\,\wz{t})\,dx\r] \pat_{\wz{t}}\lf(\gz
 \lf(\wz{t}^{\frac{1}{2m}} \r)\wz{t}^{\frac{1}{2m}-1}\r)\,d\wz{t}.
\end{eqnarray}
This, combined with the change of variables, the size condition of $\gz$
and $t_0\in(3r,\,\fz)$, implies that
\begin{eqnarray}\label{prov cacci: est J}
\nonumber\lf|\mathcal{B}\r|&&\ls \dint_{(t_0-2r)^{2m}}^{(t_0+2r)^{2m}} \lf\{\dint_{B(x_0,\,2r)}
e^{-\wz{t}L}(f)(x) \overline{e^{-\wz{t}L}(f)(x)}\lf[\eta(x)\r]^{2m}\,dx\r\} \\
&&\nonumber\hs\times\lf[\lf|\pat_{\wz{t}}
\gz\lf(\wz{t}^{\frac{1}{2m}} \r)\r| \wz{t}^{\,2(\frac{1}{2m}-1)}+
\gz\lf(\wz{t}^{\frac{1}{2m}}\r)\wz{t}^{\frac{1}{2m}-2}\r]\,d\wz{t}\\
&&\ls \frac{1}{r^{2m}}\dint_{t_0-2r}^{t_0+2r} \dint_{B(x_0,\,2r)}
\lf|u(x,\,t)\r|^2\,dx\,dt,
\end{eqnarray}
which is desired.

On the other hand, by the definition of $\mathcal{B}$ and Leibniz's rule, we know that
\begin{eqnarray}\label{3.xx4}
\nonumber\mathcal{B}&&= \Re e\lf( \dsum_{|\az|=m=|\bz|}
\dint_{t_0-2r}^{t_0+2r}
\lf\{\dint_{\rn}a_{\az,\,\bz}(x)\pat^\bz u(x,\,t)
\overline{\pat^\az (u\eta^{2m})}(x,\,t)\,dx\r\}\gz(t)\,dt\r)\\
&&\nonumber=\Re e\lf( \dsum_{|\az|=m=|\bz|}
\dint_{t_0-2r}^{t_0+2r}
\lf\{\dint_{\rn}a_{\az,\,\bz}(x)\pat^\bz u(x,\,t)\eta^{m}
\overline{\pat^\az (u\eta^{m})}(x,\,t)\,dx\r\}\gz(t)\,dt\r)\\
&&\nonumber\hs+\Re e\lf( \dsum_{|\az|=m=|\bz|}
\dint_{t_0-2r}^{t_0+2r}
\lf\{\dint_{\rn}a_{\az,\,\bz}(x)\pat^\bz u(x,\,t)\r.\r.\\
&&\hs\hs\times \overline{\lf[\dsum_{\tz\le \xi<\az}C_{(\az,\,\xi)}\pat^\xi
(u\eta^{m})\pat^{\az-\xi}(\eta^m)\r]}(x,\,t)dx\Bigg\}\gz(t)dt\Bigg),
\end{eqnarray}
where $\tz:=(0,\,\ldots,\,0)\in\nn^n$ and,
for each $\az$ and $\xi$, $C_{(\az,\,\xi)}$ is a positive constant depending
on $\az$ and $\xi$, and
$\xi<\az$ means that each component of $\xi$ is not larger than the corresponding
component of $\az$ and $|\xi|<|\az|$.

Before going further, we make the following observation.
For all multi-indices $\gz$ with  $0\le |\gz|\le m$,
by Leibniz's rule, we see that there exists
a smooth function $\eta_{\gz}$ on $\rn$ such that
$\pat^{\gz}(\eta^{m})=\eta^{m-|\gz|} \eta_{\gz}$ and
\begin{eqnarray}\label{3.8xx}
\lf\|\eta_{\gz}\r\|_{L^\fz(\rn)}\ls \frac{1}{r^{|\gz|}}.
\end{eqnarray}
Indeed, if $\gz\equiv (2,\,0,\,\ldots,\,0)$, then we have
\begin{eqnarray*}
\pat^{\gz}(\eta^{m})&&=\pat^2_{x_1}(\eta^{m})=
\pat_{x_1}(m\eta^{m-1}\pat_{x_1}\eta)=m(m-1)\eta^{m-2}\lf(\pat_{x_1}\eta\r)^2
+m\eta^{m-1}\pat_{x_1}^2\eta\\
&&=\eta^{m-2}\lf[m(m-1)\lf(\pat_{x_1}\eta\r)^2+m\eta
\pat_{x_1}^2\eta\r]=:\eta^{m-2} \eta_{\gz},
\end{eqnarray*}
where the fact that $\eta_{\gz}$ satisfies \eqref{3.8xx} is an easy consequence of
properties of $\eta$. The general
cases follow from a similar calculation, the details being omitted. From this fact and
Leibniz's rule again, we deduce that, for each $\az$ and $\bz$ as in
\eqref{3.xx4},
\begin{eqnarray}\label{3.6xx}
\nonumber C_{(\az,\,\xi)}\pat^{\xi}(u\eta^{m})\pat^{\az-\xi}(\eta^m)
&&=C_{(\az,\,\xi)}\dsum_{\tz\le \zeta\le \xi}C_{(\xi,\,\zeta)} \pat^{\xi-\zeta}(\eta^m)\pat^\zeta u
\pat^{\az-\xi}(\eta^m)\\
\nonumber&&=C_{(\az,\,\xi)}\dsum_{\tz\le \zeta\le \xi}
C_{(\xi,\,\zeta)}\eta^{m-|\xi-\zeta|}\eta_{\xi-\zeta}\pat^\zeta u \eta^{m-|\az-\xi|}
\eta_{\az-\xi}\\
\nonumber&&=\eta^m C_{(\az,\,\xi)}\dsum_{\tz\le \zeta\le \xi}
C_{(\xi,\,\zeta)}\eta^{m-|\xi-\zeta|-|\az-\xi|}\eta_{\xi-\zeta} \eta_{\az-\xi}\pat^\zeta u\\
&&=:\eta^m \dsum_{\tz\le \zeta\le \xi}\eta_{\az,\,\zeta}\pat^\zeta u,
\end{eqnarray}
where $\xi\le \az$ means that each component of $\xi$ is not larger than the corresponding
component of $\az$ and, by \eqref{3.8xx}, we see that
\begin{eqnarray}\label{3.7xx}
\lf\|\eta_{\az,\,\zeta}\r\|_{L^\fz(\rn)}\ls \frac{1}{r^{|\az-\zeta|}}.
\end{eqnarray}
Combining \eqref{3.xx4} and \eqref{3.6xx}, we conclude
that
\begin{eqnarray*}
\mathcal{B}&&=\Re e\lf( \dsum_{|\az|=m=|\bz|}
\dint_{t_0-2r}^{t_0+2r}
\lf\{\dint_{\rn}a_{\az,\,\bz}(x)\pat^\bz u(x,\,t)\eta^{m}
\overline{\pat^\az (u\eta^{m})}(x,\,t)\,dx\r\}\gz(t)\,dt\r)\\
&&\nonumber\hs+\Re e\lf( \dsum_{|\az|=m=|\bz|}
\!\dint_{t_0-2r}^{t_0+2r}\!
\lf\{\dint_{\rn}a_{\az,\,\bz}(x)\pat^\bz u(x,\,t)\!\overline{\lf[\dsum_{\tz\le \xi<\az}\eta^m
\dsum_{\tz\le \zeta\le \xi}\eta_{\az,\,\zeta}\pat^\zeta u\r]}\!(x,\,t)\,dx\r\}\gz(t)dt\r),
\end{eqnarray*}
which further implies that
\begin{eqnarray*}
\mathcal{B}&&= \Re e\lf( \dsum_{|\az|=m=|\bz|}
\dint_{t_0-2r}^{t_0+2r}
\lf[\dint_{\rn}a_{\az,\,\bz}(x)\pat^\bz (u\eta^m)(x,\,t)
\overline{\pat^\az (u\eta^{m})}(x,\,t)\,dx\r]\gz(t)\,dt\r)\\
&&\hs-\Re\! e\lf( \dsum_{|\az|=m=|\bz|}
\dint_{t_0-2r}^{t_0+2r}\!
\lf[\dint_{\rn}a_{\az,\,\bz}(x) \dsum_{\tz\le \xi<\bz} C_{(\bz,\,\xi)}\pat^\xi u(x,\,t)
\pat^{\bz-\xi}(\eta^{m})
\overline{\pat^\az (u\eta^{m})}(x,\,t)\,dx\r]\r.\\
&&\hs\hs\times\gz(t)\,dt\Bigg)\\&&
\hs+\Re e\lf(\dsum_{|\az|=m=|\bz|}\!
\dint_{t_0-2r}^{t_0+2r}\!
\lf\{\dint_{\rn}a_{\az,\,\bz}(x)\pat^\bz (u\eta^m)(x,\,t)
\overline{\lf[\dsum_{\tz\le \xi<\az}\dsum_{\tz\le \zeta\le \xi}
\eta_{\az,\,\zeta}\pat^\zeta u\r]}
(x,\,t)\,dx\r\}\gz(t)\,dt\r)\\
&&\hs-\Re e\lf( \dsum_{|\az|=m=|\bz|}\dint_{t_0-2r}^{t_0+2r}
\lf\{\dint_{\rn}a_{\az,\,\bz}(x)\dsum_{\tz\le \xi<\bz} C_{(\bz,\,\xi)}\pat^\xi u(x,\,t)
\pat^{\bz-\xi}(\eta^{m}) \r.\r.\\
&&\hs\hs\times \overline{\lf[\dsum_{\tz\le \xi<\az}\dsum_{\tz\le \zeta\le \xi}
\eta_{\az,\,\zeta}\pat^\zeta u\r]}
(x,\,t)\,dx\Bigg\}\gz(t)\,dt\Bigg).
\end{eqnarray*}
This, together with the definition of $\mathcal{A}$,
H\"older's inequality, \eqref{3.8xx}, \eqref{3.7xx},
\eqref{prov cacci: est J}, H\"older's inequality with $\epsilon$ and
the Ellipticity condition $(\mathcal{E}_0)$, implies that, for all
$\epsilon\in(0,\,\fz)$, there exists a positive constant $C_{(\epsilon)}$ such that
\begin{eqnarray*}
\mathcal{A}&&\ls \mathcal{B}+ \lf\{\dint_{t_0-2r}^{t_0+2r}
\lf[\dint_{\rn}\lf|\nabla^m (u\eta^m)(x,\,t)\r|^2\,dx\r]\gz(t)\,dt\r\}^{1/2}\\
&&\hs\times
\lf\{\dsum_{j=0}^{m-1}\frac{1}{r^{2(m-j)}}\dint_{t_0-2r}^{t_0+2r}
\lf[\dint_{\rn}\lf|\nabla^j u(x,\,t)\r|^2\,dx\r]\gz(t)\,dt\r\}^{1/2}\\
&&\hs+\dsum_{j=0}^{m-1}\frac{1}{r^{2(m-j)}}\dint_{t_0-2r}^{t_0+2r}
\lf[\dint_{\rn}\lf|\nabla^j u(x,\,t)\r|^2\,dx\r]\gz(t)\,dt\\
&&\ls \mathcal{B}+ \epsilon\dint_{t_0-2r}^{t_0+2r}
\lf[\dint_{\rn}\lf|\nabla^m (u\eta^m)(x,\,t)\r|^2\,dx\r]\gz(t)\,dt\\
&&\hs+C_{(\epsilon)} \dsum_{j=0}^{m-1}\frac{1}{r^{2(m-j)}}\dint_{t_0-2r}^{t_0+2r}
\lf[\dint_{B(x_0,\,2r)}\lf|\nabla^j u(x,\,t)\r|^2\,dx\r]\gz(t)\,dt\\
&&\ls \dsum_{j=0}^{m-1}\frac{1}{r^{2(m-j)}}\dint_{t_0-2r}^{t_0+2r}
\lf[\dint_{B(x_0,\,2r)}\lf|\nabla^j u(x,\,t)\r|^2\,dx\r]\gz(t)\,dt+\frac{\epsilon}{\lz_0}\mathcal{A},
\end{eqnarray*}
which, combined with \eqref{0.1}, shows that \eqref{caccioppoli's inequality2 with epsilon}
holds true.

We now turn to the proof of \eqref{caccioppoli's inequality1 with epsilon}.
Recall that, from \cite[Theorem 5.2(3)]{af03} with some slight modifications, we easily
deduce that there exists a positive constant
$C_{(n,\,m)}$, depending only on $n$ and $m$, such that, for all balls $B$,
$f\in W^{m,\,p}(B)$ and $k\in\{0,\,\ldots,\,m\}$,
\begin{eqnarray}\label{interpolation of sobolev space}
\lf\|\nabla^kf\r\|_{L^2(B)}\le C_{(n,\,m)} \lf\|\nabla^mf\r\|_{L^2(B)}^{k/m}
\lf\|f\r\|_{L^2(B)}^{1-k/m},
\end{eqnarray}
which, together with  \eqref{caccioppoli's inequality2 with epsilon}, the interpolation
inequality and Young's inequality with $\epsilon$, immediately implies that
\eqref{caccioppoli's inequality1 with epsilon} holds true.
This finishes the proof of Proposition \ref{caccioppoli's inequality}.
\end{proof}

The following proposition improves Proposition \ref{caccioppoli's inequality}
by removing all the terms with gradients on the right hand side of
Caccioppoli's inequality \eqref{caccioppoli's inequality2 with epsilon},
which is motivated by a recent result of Barton
\cite{ba14}.

\begin{proposition}\label{caccioppoli's inequality2}
Let $L$ be as in \eqref{L homo} and satisfy the Ellipticity condition
$(\mathcal{E}_0)$, and let $f\in L^2(\rn)$, $t\in(0,\,\fz)$ and $u(x,\,t)
:=e^{-t^{2m}{L}}(f)(x)$ for all $x\in\rn$. Then there exists a positive constant $C$
such that, for all $x_0\in\rn$, $r\in(0,\,\fz)$ and $t_0\in(3r,\,\fz)$,
\begin{eqnarray}\label{caccioppoli's inequality3 with epsilon}
&&\dint_{t_0-r}^{t_0+r} \dint_{B(x_0,\,r)}\lf|\nabla^m u(x,\,t)\r|^2
\,dx\,dt\le \frac{C}{r^{2m}}
\dint_{t_0-2r}^{t_0+2r} \dint_{B(x_0,\,2r)}\lf|u(x,\,t)\r|^2
\,dx\,dt.
\end{eqnarray}
\end{proposition}

\begin{proof}
We prove this proposition
by borrowing some ideas from the proof of \cite[Theorem 3.10]{ba14}.
We first make the following {\it claim} that,
to finish the proof of Proposition \ref{caccioppoli's inequality2},
we only need to show that,
for all $j\in\{1,\,\ldots,\,m\}$ and {$0<\zeta<\xi\le 2r$},
\begin{eqnarray}\label{caccioppoli's inequality10 L}
\nonumber&&\dint_{t_0-\zeta}^{t_0+\zeta}\dint_{B(x_0,\,\zeta)}
\lf|\nabla^j u(x,\,t)\r|^2\,dx\,dt\\
&&\hs\ls \sum_{k=0}^{j-1}\frac{1}{(\xi-\zeta)^{2(j-k)}} \dint_{t_0-\xi}^{t_0+\xi}
\!\dint_{B(x_0,\,\xi)} \lf|\nabla^k u(x,\,t)\r|^2\,dx\,dt.
\end{eqnarray}
Indeed, if \eqref{caccioppoli's inequality10 L} holds true for all
$j\in\{1,\,\ldots,\,m\}$ and $0<\zeta<\xi<\fz$, then let
$$r=r_0<r_1<r_2<\cdots<r_m=2r$$
be an average decomposition of $(r,\,2r)$ and
\begin{eqnarray}\label{Arj}
A_{s,\,l}:=\dint_{t_0-s}^{t_0+s}\dint_{B(x_0,\,s)} \lf|\nabla^l u(x,\,t)\r|^2\,dx\,dt
\end{eqnarray}
with $l\in\{0,\,\cdots,\,m\}$ and $s\in(0,\,\fz)$. By repetitively using
\eqref{caccioppoli's inequality10 L} with $j\in\{m,\,\ldots,\,1\}$, we know that
\begin{eqnarray*}
\dint_{t_0-r}^{t_0+r}\dint_{B(x_0,\,r)}
\lf|\nabla^m u(x,\,t)\r|^2\,dx\,dt
&&=A_{r_0,\,m}\ls \dsum_{j=0}^{m-1}
\frac{1}{r^{2(m-j)}}A_{r_1,\,j}\\
&&\ls \frac{1}{r^{2m}} A_{r_1,\,0}+
\dsum_{j=1}^{m-1}\frac{1}{r^{2(m-j)}}\dsum_{k=0}^{j-1} \frac{1}{r^{2(j-k)}} A_{r_2,\,k}\\
&&\sim  \dsum_{k=0}^{m-2}\frac{1}{r^{2(m-k)}}  A_{r_2,\,k}
\ls\frac{1}{r^{2m}}A_{2r,\,0},
\end{eqnarray*}
which immediately implies that Proposition \ref{caccioppoli's inequality2} holds true.

We now turn to the proof of \eqref{caccioppoli's inequality10 L}.
{Observe that, if $j=m$, then  \eqref{caccioppoli's inequality10 L} can
be proved by using the same argument as the proof of the parabolic
Caccioppoli's inequality \eqref{caccioppoli's inequality2 with epsilon} with $r$
and $2r$ replaced, respectively, by $\zeta$ and $\xi$,
noticing that the assumption $0<\zeta<\xi\le2r$, together with $t_0\in(3r,\fz)$, implies that, for all
$t\in (t_0-\xi,\,t_0+\xi)$, $\frac{1}{t}<\frac{1}{t_0-\xi}<\frac{1}{r}<\frac{2}{\xi-\zeta}$.}

 Thus, by induction,
to finish the proof of  \eqref{caccioppoli's inequality10 L}, it remains to show that,
if  \eqref{caccioppoli's inequality10 L}  holds true for some $j+1$,
then   \eqref{caccioppoli's inequality10 L} also holds true for $j$.

Now, for all $i\in\zz_+$, let $\{\rho_i\}_{i\in\zz_+}$ be a sequence of increasing numbers satisfying
$$\zeta=\rho_0<\rho_1<\cdots<\xi,$$ $\dz_i:=\rho_{i+1}-\rho_{i}$ and $\wz{\rho}_i:=\rho_i+\frac{\dz_i}{2}$,
where the exact value of $\rho_i$ will be determined later.
Let $\fai_i\in C_c^\fz(B(x_0,\,\wz{\rho}_i))$ satisfy $\supp \fai_i\subset B(x_0,\,\wz{\rho}_i)$,
$\fai_i\equiv 1$ on $B(x_0,\,{\rho}_i)$, $\|\nabla\fai_i\|_{L^\fz(\rn)}\ls \frac{1}{\dz_i}$ and
$\|\nabla^2\fai_i\|_{L^\fz(\rn)}\ls \frac{1}{\dz_i^2}$. By properties of $\fai_i$, the Fourier transform
and H\"older's inequality, we know that, for all $i\in\zz_+$,
\begin{eqnarray*}
&&\dint_{t_0-\rho_i}^{t_0+\rho_i}\dint_{B(x_0,\,\rho_i)}
\lf|\nabla^j u(x,\,t)\r|^2\,dx\,dt\\
&&\hs\le\dint_{t_0-\wz\rho_i}^{t_0+\wz\rho_i}
\dint_{B(x_0,\,\wz\rho_i)}\lf|\nabla\lf(\fai_i\nabla^{j-1} u\r)(x,\,t)\r|^2\,dx\,dt
\\&&\hs\ls\lf\{\dint_{t_0-\wz\rho_i}^{t_0+\wz\rho_i}
\dint_{B(x_0,\,\wz\rho_i)}\lf|\nabla^{j-1} u(x,\,t)\r|^2\,dx\,dt\r\}^{\frac{1}{2}}\\
&&\hs\hs\times
\lf\{\dint_{t_0-\wz\rho_i}^{t_0+\wz\rho_i}
\dint_{B(x_0,\,\wz\rho_i)}\lf|\nabla^2\lf(\fai_i\nabla^{j-1} u\r)(x,\,t)\r|^2\,dx\,dt\r\}^{\frac{1}{2}}\\&&\hs\ls\lf\{\dint_{t_0-\wz\rho_i}^{t_0+\wz\rho_i}
\dint_{B(x_0,\,\wz\rho_i)}\lf|\nabla^{j-1} u(x,\,t)\r|^2\,dx\,dt\r\}^{\frac{1}{2}}\\
&&\hs\hs\times
\lf\{\dint_{t_0-\wz\rho_i}^{t_0+\wz\rho_i}
\dint_{B(x_0,\,\wz\rho_i)}\lf|\lf(\lf|\nabla^{j+1} u\r|+\frac{1}{\dz_i}\lf|\nabla^{j} u\r|+
\frac{1}{\dz_i^2}\lf|\nabla^{j-1} u\r|\r)(x,\,t)\r|^2\,dx\,dt\r\}^{\frac{1}{2}}\\
&&\hs\ls \lf[A_{\wz\rho_i,\,j-1}\r]^{1/2}
\lf[A_{\wz\rho_i,\,j+1}+\frac{1}{\dz_i^2}A_{\wz\rho_i,\,j}+\frac{1}{\dz_i^4}A_{\wz\rho_i,\,j-1}\r]^{1/2},
\end{eqnarray*}
where $A_{\wz\rho_i,\,j+1}$, $A_{\wz\rho_i,\,j}$ and $A_{\wz\rho_i,\,j-1}$ are defined as in \eqref{Arj}.

From the assumption that \eqref{caccioppoli's inequality10 L} holds true for $j+1$ and H\"older's inequality with
$\epsilon$, we further deduce that, for any $\epsilon\in(0,\,\fz)$, there exists a positive constant $C_{(\epsilon)}$
such that
\begin{eqnarray*}
&&\dint_{t_0-\rho_i}^{t_0+\rho_i}\dint_{B(x_0,\,\rho_i)}
\lf|\nabla^j u(x,\,t)\r|^2\,dx\,dt\\
&&\hs\ls \lf[A_{\rho_{i+1},\,j-1}\r]^{\frac{1}{2}}
\lf[\dsum_{k=0}^{j}\frac{1}{\dz_i^{2(j+1-k)}}A_{\rho_{i+1},\,k}
+\frac{1}{\dz_i^2}A_{\wz\rho_i,\,j}+\frac{1}{\dz_i^4}A_{\wz\rho_i,\,j-1}\r]^{\frac{1}{2}}\\
&&\hs\ls \frac{C_{(\epsilon)}}{\dz_i^2}A_{\rho_{i+1},\,j-1}+
\epsilon \dsum_{k=0}^{j}\frac{1}{\dz_i^{2(j-k)}}A_{\rho_{i+1},\,k}
+\epsilon A_{\rho_{i+1},\,j}.
\end{eqnarray*}
By letting $\epsilon$ small enough, we conclude that there exists a positive constant $\wz C$ such that,
for all $i\in\zz_+$,
\begin{eqnarray}\label{6.11}
A_{\rho_{i},\,j}&&=
\dint_{t_0-\rho_i}^{t_0+\rho_i}\dint_{B(x_0,\,\rho_i)}
\lf|\nabla^j u(x,\,t)\r|^2\,dx\,dt
\le \wz C \dsum_{k=0}^{j-1}\frac{1}{\dz_i^{2(j-k)}}A_{\rho_{i+1},\,k}
+\frac{1}{2}A_{\rho_{i+1},\,j}\nonumber\\
&&=: \wz C B_{i+1,\,j}+\frac{1}{2}A_{\rho_{i+1},\,j},
\end{eqnarray}
which immediately implies that
\begin{eqnarray}\label{6.12}
A_{\rho_0,\,j}\le \wz C B_{1,\,j}+\frac{1}{2}A_{\rho_{1,\,j}}\le \wz C B_{1,\,j}+\frac{1}{2}
\lf[\wz C B_{2,\,j}+\frac{1}{2}A_{\rho_{2,\,j}}\r]\le \wz C \dsum_{i=0}^\fz 2^{-i}B_{i+1,\,j}.
\end{eqnarray}

Moreover, for all $i\in\zz_+$, take $\tau\in(2^{-1/(2m)},\,1)$,
$\rho_i:=\zeta+(\xi-\zeta)(1-\tau)\sum_{s=1}^i\tau^s$ and
hence $\dz_i=(\xi-\zeta)(1-\tau)\tau^{i+1}$,
we then have
\begin{eqnarray}\label{6.15}
\nonumber\dsum_{i=0}^\fz 2^{-i}B_{i+1,\,j}&&=\dsum_{i=0}^\fz 2^{-i}
\wz C \dsum_{k=0}^{j-1} \frac{1}{\dz_i^{2(j-k)}}A_{\rho_{i+1},\,k}\\
&&\nonumber\ls \dsum_{i=0}^\fz 2^{-i} \wz C \dsum_{k=0}^{j-1}
\frac{1}{[(\xi-\zeta)(1-\tau)\tau^i]^{2(j-k)}}A_{\rho_{i+1},\,k}\\
&&\ls \dsum_{k=0}^{j-1}
\dsum_{i=0}^\fz \frac{1}{[2\tau^{2(j-k)}]^{i}}
\frac{1}{(\xi-\zeta)^{2(j-k)}}A_{\rho_{i+1},\,k}
\ls \dsum_{k=0}^{j-1} \frac{1}{(\xi-\zeta)^{2(j-k)}}A_{\xi,\,k},
\end{eqnarray}
where the implicit positive constants depend on $m$ and $\tau$, but
are independent of $j$, $\xi$ and $\zeta$.

Combining the estimates \eqref{6.12} and \eqref{6.15},
and using the definition of  $B_{i+1,\,j}$
in \eqref{6.11}, we conclude that
\begin{eqnarray*}
\dint_{t_0-\zeta}^{t_0+\zeta}\dint_{B(x_0,\,\zeta)}
\lf|\nabla^j u(x,\,t)\r|^2\,dx\,dt=A_{\rho_0,\,j}\ls \dsum_{i=0}^\fz 2^{-i}B_{i+1,\,j}
\ls\dsum_{k=0}^{j-1} \frac{1}{(\xi-\zeta)^{2(j-k)}}A_{\xi,\,k},
\end{eqnarray*}
which immediately implies that \eqref{caccioppoli's inequality10 L} holds true for $j$.
Thus, by induction, \eqref{caccioppoli's inequality10 L}  holds true for all $j\in\{1,\,\ldots,\,m\}$,
which completes the proof of Proposition \ref{caccioppoli's inequality2}.
\end{proof}

Now, for $p\in(0,\,p_+(L))$, we want to control the $H_L^p(\rn)$ quasi-norm
by the $L^p(\rn)$ quasi-norm of $S_{h,\,L}$ in \eqref{gradient square function for heat},
via the parabolic Caccioppoli's inequality
\eqref{caccioppoli's inequality3 with epsilon}.
Before going further, we point out that, in the remainder of this section,
including the proofs of Propositions \ref{domination of square functions}
and \ref{domination of heat by maximal},
and Theorem \ref{non-tangential heat maximal functiona characterization},
we borrow some ideas from the corresponding parts of \cite{hm09}, in which the authors
considered the case when $m=1$ and $p=1$.

We first need the following notation. For all $\lz\in(0,\,\fz)$,  $k\in\zz_+$ and
$f\in L^2(\rn)$, let $S^\lz_{L,\,k}(f)$ and $S^\lz_{h,\,L,\,k}(f)$ be the same,
respectively, as in \eqref{eqn KSF} and \eqref{gradient square function for heat}.
For any $0<\epsilon\ll R<\fz$ and $x\in\rn$, let $\bgz^{\epsilon,\,R,\,\lz}(x)$
be the \emph{truncated cone} defined by setting
\begin{eqnarray}\label{contracted cone}
\bgz^{\epsilon,\,R,\,\lz}(x):=\lf\{(y,\,t)\in\rn\times (\epsilon,\,R):\
|x-y|<\lz t\r\}.
\end{eqnarray}
We write $S_{L,\,k}^{\lz} (f)(x)$
and $S_{L,\,h,\,k}^\lz (f)(x)$, respectively, by $S_{L,\,k}^{\epsilon,\,R,\,\lz} (f)(x)$
and $S_{L,\,h,\,k}^{\epsilon,\,R,\,\lz} (f)(x)$ when the cone $\Gamma^\lz(x)$,
in \eqref{eqn KSF} and \eqref{gradient square function for heat}, is replaced by
$\bgz^{\epsilon,\,R,\,\lz}(x)$.

From \cite[Proposition 4]{cms85}, it follows that, for all $k\in\zz_+$,
$\lz\in(0,\,\fz)$, $p\in(0,\,\fz)$ and $f\in L^2(\rn)$,
\begin{eqnarray}\label{3.xx18}
\lf\| S_{h,\,L,\,k}^\lz (f)\r\|_{L^p(\rn)}\sim
\lf\| S_{h,\,L,\,k}(f)\r\|_{L^p(\rn)}
\end{eqnarray}
and
\begin{eqnarray}\label{3.xxx18}
\lf\| S_{L,\,k}^\lz (f)\r\|_{L^p(\rn)}\sim\lf\| S_{L,\,k}(f)\r\|_{L^p(\rn)},
\end{eqnarray}
where the implicit positive constants are independent of $f$.

For $p\in(0,\,\fz)$, we can control the $L^p(\rn)$ quasi-norm of $S_L$ in \eqref{eqn KSF}
by that of $S_{h,\,L}$ in \eqref{gradient square function for heat}
as follows.

\begin{proposition}\label{domination of square functions}
Let $L$ be as in \eqref{L homo} and satisfy the Ellipticity condition
$(\mathcal{E}_0)$, and let $p\in(0,\,\fz)$.
Then there exists a positive constant $C$ such that, for all {$f\in {L}^2(\rn)$},
\begin{eqnarray*}
\lf\|S_{L}(f)\r\|_{L^p(\rn)}\le C \lf\|S_{h,\,L}(f)\r\|_{L^p(\rn)}.
\end{eqnarray*}
\end{proposition}

In what follows, for all $k\in\zz_+$ and suitable functions $H$ on $\rr^{n+1}_+$ and $x\in\rn$, let
\begin{eqnarray}\label{meqnAk}
\mathcal{A}_k(H)(x):=\lf\{\iint_{\bgz^{2^k}(x)}
\lf|H(y,\,t)\r|^{2}\,\frac{dy\,dt}{t^{n+1}}\r\}^{\frac{1}{2}}.
\end{eqnarray}
Observe that $\mathcal{A}_0(H)$ is just the $\mathcal{A}$-functional $\mathcal{A}(H)$
defined as in \eqref{meqnA} with $F$ replaced by $H$.

To prove Proposition \ref{domination of square functions}, we need the following technical
lemma, which is due to  Steve Hofmann (a personal communication with the second author).

\begin{lemma}\label{mlem 1}
{Let $F$, $G\in T^2(\rr^{n+1}_+)$. If there exists a positive constant $C_0$ such that,
for all $k\in \zz_+$ and almost every $x\in\rn$,
\begin{eqnarray}\label{meqn1}
\mathcal{A}_k(F)(x)\le C_0 \lf[\mathcal{A}_{k+1}(G)(x)\r]^{1/2}\lf[\mathcal{A}_{k+1}(F)(x)\r]^{1/2},
\end{eqnarray}
then, for all $p\in(0,\,\fz)$, there exists a positive constant $C_1$, independent of $F$ and $G$,
such that
\begin{eqnarray*}
\|F\|_{T^p(\rr^{n+1}_+)}\le C_1 \|G\|_{T^p(\rr^{n+1}_+)}.
\end{eqnarray*}}
\end{lemma}

\begin{proof}
From \cite[Theorem 1.1]{Au11}, it follows that there exists a positive constant
$C_{(n,\,p)}\in[1,\,\fz)$, depending on $n$ and $p$, but being independent of $F$,
such that, for all $F\in T^2(\rr^{n+1}_+)$,
\begin{eqnarray}\label{meqn2}
\lf\|\mathcal{A}_k(F)\r\|_{L^p(\rn)}\le C_{(n,\,p)} \lf\|\mathcal{A}_{k-1}(F)\r\|_{L^p(\rn)}\le
\lf[C_{(n,\,p)}\r]^k \lf\|\mathcal{A}(F)\r\|_{L^p(\rn)}.
\end{eqnarray}
Moreover, $C_{(n,\,p)}\ge C_{(n,\,2)}$ for all $p\in(0,\,\fz)$.

Let $R\in (2C_{(n,\,p)},\,\fz)$ and $\mathcal{A}_*(F):=\sum_{k=0}^\fz \frac{1}{R^k}\mathcal{A}_k(F)$.
By \eqref{meqn2}, we know that
\begin{eqnarray*}
\lf\|\mathcal{A}_*(F)\r\|_{L^2(\rn)}\le \dsum_{k=0}^\fz\frac{1}{R^k}
\lf\|\mathcal{A}_k(F)\r\|_{L^2(\rn)}\le
\dsum_{k=0}^\fz\lf[\frac{C_{(n,\,2)}}{R}\r]^k \lf\|\mathcal{A}(F)\r\|_{L^2(\rn)}= 2
\lf\|F\r\|_{T^2(\rr^{n+1}_+)}<\fz,
\end{eqnarray*}
which immediately implies that $\mathcal{A}_*(F)(x)<\fz$ almost everywhere in $\rn$.

On the other hand, using \eqref{meqn1} and Cauchy's inequality, we find that, for all $k\in\zz_+$
and almost every $x\in\rn$,
\begin{eqnarray*}
\mathcal{A}_k(F)(x)\le \frac{1}{2}C_0^2R \mathcal{A}_{k+1}(G)(x)+
\frac{1}{2R} {\mathcal{A}_{k+1}(F)(x)},
\end{eqnarray*}
which, together with the definition of $\mathcal{A}_*$, shows that
\begin{eqnarray*}
\mathcal{A}_*(F)(x)\le \frac{C_0^2R^2}{2}\mathcal{A}_*(G)(x)+\frac{1}{2}\mathcal{A}_*(F)(x).
\end{eqnarray*}
This, combined with the fact $\mathcal{A}_*(F)(x)<\fz$ almost everywhere
in $\rn$, further implies that
\begin{eqnarray*}
\mathcal{A}_*(F)(x)\le C_0^2R^2\mathcal{A}_*(G)(x).
\end{eqnarray*}
Thus, by the fact that $\sum_{k=0}^\fz\frac{1}{R^k}=\frac{R}{R-1}$,
the definition of the $\mathcal{A}$-functional in \eqref{meqnA}
and \eqref{meqn2},  we have
\begin{eqnarray*}
\|F\|_{T^p(\rr^{n+1}_+)}&&=\|\mathcal{A}(F)\|_{L^p(\rn)}= \frac{R-1}{R}
\lf\|\dsum_{k=0}^\fz \frac{1}{R^k}\mathcal{A}(F)\r\|_{L^p(\rn)}\le \frac{R-1}{R}
\lf\|\dsum_{k=0}^\fz \frac{1}{R^k}\mathcal{A}_k(F)\r\|_{L^p(\rn)}\\
&&\le \frac{R-1}{R}\lf\|\mathcal{A}_*(F)\r\|_{L^p(\rn)}\le \frac{R-1}{R}
\lf\|{C_0^2 R^2}\mathcal{A}_*(G)\r\|_{L^p(\rn)}\\
&&\le \frac{R-1}{R}
\lf\|\dsum_{k=0}^\fz \frac{C_0^2R^2}{R^k}\mathcal{A}_k(G)\r\|_{L^p(\rn)}\le
\frac{(R-1)C_0^2R^2}{R}\dsum_{k=0}^\fz \frac{1}{2^k}\lf\|\mathcal{A}(G)\r\|_{L^p(\rn)}\\
&&\le \frac{2(R-1)C_0^2R^2}{R}\lf\|\mathcal{A}(G)\r\|_{L^p(\rn)}=
\frac{2(R-1)C_0^2R^2}{R}\lf\|G\r\|_{T^p(\rr^{n+1}_+)},
\end{eqnarray*}
which completes the proof of Lemma \ref{mlem 1}.
\end{proof}

With the help of Lemma \ref{mlem 1}, we now prove
Proposition \ref{domination of square functions}.

\begin{proof}[Proof of Proposition \ref{domination of square functions}.]
We begin the proof of this proposition
by first introducing some smooth cut-off functions supported
in truncated cones. For all $0<\epsilon \ll R<\fz$, $\lz\in(0,\,\fz)$
and $x\in\rn$, let
$\bgz^{\epsilon,\,R,\,\lz}(x)$ be the truncated cone defined as in
\eqref{contracted cone}.

Let $\eta\in C_c^\fz(\bgz^{\epsilon/2,\,2R,\,3/2}(x))$ satisfy $\eta\equiv1$
on $\bgz^{\epsilon,\,R,\,1}(x)$, $0\le \eta\le 1$ and, for all $k\in\nn$ with $k\le m$
and $(y,\,t)\in \bgz^{\epsilon/2,\,2R,\,3/2}(x)$,
$$\lf|\nabla^k\eta(y,\,t)\r|\ls \frac{1}{t^{k}}.$$
From the definition of ${L}$ and Minkowski's inequality, we deduce that, for all $x\in\rn$,
\begin{eqnarray}\label{I0-Im}
\nonumber&&\lf\{\iint_{\bgz^{\epsilon,\,R,\,1}(x)} \lf|t^{2m}{L}e^{-t^{2m}{L}}(f)(y)\r|^2
\frac{dy\,dt}{t^{n+1}}\r\}^{\frac{1}{2}}\\
&&\nonumber\hs\le \lf|\iint_{\bgz^{\epsilon/2,\,2R,\,3/2}(x)}
t^{2m}{L}e^{-t^{2m}{L}}(f)(y)
\overline{t^{2m}{L}e^{-t^{2m}{L}}(f)(y)}\eta(y,\,t)\,
\frac{dy\,dt}{t^{n+1}}\r|^{\frac{1}{2}}\\
&&\nonumber\hs= \lf|\dsum_{|\az|=m=|\bz|}
\iint_{\bgz^{\epsilon/2,\,2R,\,3/2}(x)} a_{\az,\,\bz}(y)
t^{m}\pat^{\bz}\lf(e^{-t^{2m}{L}}(f)\r)(y)\r.\\
&&\nonumber\hs\hs\times
\overline{t^m\pat^{\az}\lf(t^{2m}{L}e^{-t^{2m}{L}}(f)\eta\r)(y,\,t)}
\,\frac{dy\,dt}{t^{n+1}}\Bigg|^{\frac{1}{2}}\\
&&\nonumber\hs\ls \dsum_{k=0}^{m}\lf|\dsum_{|\az|=m=|\bz|}
\iint_{\bgz^{\epsilon/2,\,2R,\,3/2}(x)} a_{\az,\,\bz}(y)
t^{m}\pat^{\bz}\lf(e^{-t^{2m}{L}}(f)\r)(y)\r.\\
&&\hs\hs\times\lf.\lf[t^m \dsum_{|\wz\az|=k,\,\wz\az\le \az}
C_{(\az,\,\wz\az)}\overline{\pat^{\wz\az}\lf(t^{2m}{L}e^{-t^{2m}{L}}(f)\r)(y)
\, \pat^{\az-\wz\az}\eta(y,\,t)}\r]\,\frac{dy\,dt}{t^{n+1}}\r|^{\frac{1}{2}}
=:\dsum_{k=0}^{m}\mathrm{I}_k.
\end{eqnarray}

We first bound $\mathrm{I}_0$. By H\"older's inequality,
the size condition of $\eta$ and the Ellipticity condition $(\mathcal{E}_0)$,
we see that, for all $x\in\rn$,
\begin{eqnarray}\label{estimates of I0}
\nonumber\mathrm{I}_0
&&\ls \lf\{\iint_{\bgz^{\epsilon/2,\,2R,\,3/2}(x)}
\lf|t^{m}\nabla^{m}\lf(e^{-t^{2m}{L}}(f)\r)(y)\r|^2\frac{dy\,dt}
{t^{n+1}}\r\}^{\frac{1}{4}}\\
&&\nonumber\hs\times\lf\{\iint_{\bgz^{\epsilon/2,\,2R,\,3/2}(x)}
\lf|t^{2m}{L}e^{-t^{2m}{L}}(f)(y)\r|^2
\frac{dy\,dt}{t^{n+1}}\r\}^{\frac{1}{4}}\\
&&\ls\lf[S^{\epsilon/2,\,2R,\,3/2}_{h,\,L}(f)(x)\r]^{\frac{1}{2}}
\lf[S^{\epsilon/2,\,2R,\,3/2}_{L}(f)(x)\r]^{\frac{1}{2}},
\end{eqnarray}
where $S^{\epsilon/2,\,2R,\,3/2}_{h,\,L}(f)$ and
$S^{\epsilon/2,\,2R,\,3/2}_{L}(f)$
are defined, respectively, similar to $S_{h,\,L}(f)$
in \eqref{gradient square function for heat} and $S_{L}(f)$ in \eqref{eqn KSF},
with $\bgz(x)$ replaced by
$\bgz^{\epsilon/2,\,2R,\,3/2}(x)$.

To bound $\mathrm{I}_m$, similar to \eqref{estimates of I0},
for all $x\in\rn$, we have
\begin{eqnarray}\label{2.x1}
\nonumber\mathrm{I}_m &&\ls \lf\{\iint_{\bgz^{\epsilon/2,\,2R,\,3/2}(x)}
\lf|t^{m}\nabla^{m}\lf(e^{-t^{2m}{L}}(f)\r)(y)\r|^2\frac{dy\,dt}
{t^{n+1}}\r\}^{\frac{1}{4}}\\ \nonumber
&&\hs\times\lf\{\iint_{\bgz^{\epsilon/2,\,2R,\,3/2}(x)}
\lf|t^m\nabla^m\lf(t^{2m}Le^{-t^{2m}{L}}(f)\r)(y)\r|^2
\frac{dy\,dt}{t^{n+1}}\r\}^{\frac{1}{4}}\\ 
&&\sim\lf[S^{\epsilon/2,\,2R,\,3/2}_{h,\,L}(f)(x)\r]^{\frac{1}{2}}
\lf[S^{\epsilon/2,\,2R,\,3/2}_{h,\,L,\,1}(f)(x)\r]^{\frac{1}{2}}.
\end{eqnarray}
To bound $S^{\epsilon/2,\,2R,\,3/2}_{h,\,L,\,1}(f)(x)$, let $Q(z,\,2r)$ be the cube with center $z$
and sidelength $2r$ in $\rr^{n+1}_+$. Write $z:= (z^*,\,t)$ with $z^*\in\rn$
and $t\in(0,\,\fz)$.
Assume that $\{Q(z_j,\,2r_j)\}_{j\in\nn}$ is a covering of
$\bgz^{\epsilon/2,\,2R,\,3/2}(x)$ satisfying
\begin{eqnarray*}
\bgz^{\epsilon/2,\,2R,\,3/2}(x)\subset \bigcup_{j\in\nn}
Q(z_j,\,2r_j)\subset \bigcup_{j\in\nn} Q(z_j,\,4\sqrt nr_j) \subset
\bgz^{\epsilon/4,\,3R,\,2}(x),
\end{eqnarray*}
\begin{eqnarray*}
d\lf(z_j,\,\lf(\bgz^{\epsilon/4,\,3R,\,2}(x)\r)^{\complement}\r)
\sim r_j\sim d\lf(z_j,\,\{t=0\}\r),\hs\hs j\in\nn
\end{eqnarray*}
and the collection $\{B(z_j^*,\,\sqrt n r_j) \times (t_j-\sqrt nr_j,\,t_j+\sqrt nr_j)\}_{j\in\nn}$
has a bounded overlap, where, for all $j\in\nn$, $z_j:=(z_j^*,\,t_j)$.
 This kind of covering is
based on Whitney's decomposition; see \cite[(5.26)]{hm09} for a
covering of similar nature consisting of balls.

It is easy to see that
\begin{eqnarray*}
\bgz^{\epsilon/2,\,2R,\,3/2}(x)&&\subset \bigcup_{j\in\nn}
Q\lf(z_j,\,2r_j\r)\subset \bigcup_{j\in\nn}B\lf(z_j^*,\,\sqrt n r_j\r)
\times \lf(t_j-\sqrt nr_j,\,t_j+\sqrt nr_j\r)\\
&&\subset \bigcup_{j\in\nn} B\lf(z_j^*,\,2\sqrt nr_j\r)
\times \lf(t_j-2\sqrt nr_j,\,t_j+2\sqrt nr_j\r)
\subset \bigcup_{j\in\nn} Q\lf(z_j,\,4\sqrt nr_j\r)\\
&&\subset\bgz^{\epsilon/4,\,3R,\,2}(x).
\end{eqnarray*}
From these and the parabolic Caccioppoli's inequality
\eqref{caccioppoli's inequality3 with epsilon},
we deduce that, for all $0<\epsilon\ll R<\fz$ and $x\in\rn$,
\begin{eqnarray}\label{estimates of IIm}
\nonumber\lf[S^{\epsilon/2,\,2R,\,3/2}_{h,\,L,\,1}(f)(x)\r]^2
&&\ls\dsum_{j\in\nn}\dint_{t_j-\sqrt n r_j}^{t_j+\sqrt n r_j}
\dint_{B(z^*_j,\,\sqrt n r_j)}
\lf|t^m\nabla^m\lf(t^{2m}{L}e^{-t^{2m}{L}}(f)\r)(y)\r|^2
\frac{dy\,dt}{t^{n+1}}\\
&&\nonumber\ls \dsum_{j\in\nn}\frac{1}{r_j^{2m}}
\dint_{t_j-2\sqrt n r_j}^{t_j+2\sqrt n r_j}
\dint_{B(z^*_j,\,2\sqrt n r_j)}
\lf|t^{3m}{L}e^{-t^{2m}{L}}(f)(y)\r|^2
\frac{dy\,dt}{t^{n+1}}\\
&&\ls\iint_{\bgz^{\epsilon/4,\,3R,\,2}(x)}
\lf|t^{2m}{L}e^{-t^{2m}{L}}(f)(y)\r|^2
\frac{dy\,dt}{t^{n+1}}\nonumber\\
&&\sim\lf[S^{\epsilon/4,\,3R,\,2}_{L}(f)(x)\r]^2,
\end{eqnarray}
which, together with letting $\epsilon\to 0$ and $R\to \fz$,
\eqref{3.xx18} and \eqref{3.xxx18},
implies that, for all $q\in (0,\,\fz)$ and $f\in L^2(\rn)$,
\begin{eqnarray}\label{eqn 3.xxx25}
\lf\|S_{h,\,L,\,1}(f)\r\|_{L^q(\rn)}\ls \lf\|S_{L}(f)\r\|_{L^q(\rn)}.
\end{eqnarray}

Moreover, by \eqref{2.x1} and \eqref{estimates of IIm}, we conclude that,
for all $x\in\rn$,
\begin{eqnarray}\label{estimates of Im}
\mathrm{I}_m&&\le \lf[S^{\epsilon/2,\,2R,\,3/2}_{h,\,L}(f)(x)\r]^{\frac{1}{2}}
\lf[S^{\epsilon/4,\,3R,\,2}_{L}(f)(x)\r]^{\frac{1}{2}}.
\end{eqnarray}

We now turn to the estimates of $\mathrm{I}_k$ for all $k\in\{1,\,\ldots,\,m-1\}$.
Again, by H\"older's inequality, the Ellipticity condition $(\mathcal{E}_0)$ and
the size condition of $\eta$, we see that
\begin{eqnarray*}
\mathrm{I}_k&&\ls \lf\{\iint_{\bgz^{\epsilon/2,\,2R,\,3/2}(x)}
\lf|t^{m}\nabla^{m}\lf(e^{-t^{2m}{L}}(f)\r)(y)\r|^2\frac{dy\,dt}
{t^{n+1}}\r\}^{\frac{1}{4}}\\
&&\hs\times\lf\{\iint_{\bgz^{\epsilon/2,\,2R,\,3/2}(x)}
\lf|t^k\nabla^k\lf(t^{2m}{L}e^{-t^{2m}{L}}(f)\r)(y)\r|^2
\frac{dy\,dt}{t^{n+1}}\r\}^{\frac{1}{4}}\\
&&=:\lf[S^{\epsilon/2,\,2R,\,3/2}_{h,\,L}(f)(x)\r]^{\frac{1}{2}}
\times(\mathrm{II}_k)^{\frac{1}{4}}.
\end{eqnarray*}
To bound $\mathrm{II}_k$, using again the interpolation inequality
\eqref{interpolation of sobolev space} and H\"older's inequality, we conclude that
\begin{eqnarray*}
\mathrm{II}_k&&\sim\dint_{\epsilon/2}^{2R}
\lf\|t^{k}\nabla^k\lf(t^{2m}{L}e^{-t^{2m}{L}}(f)\r)\r\|^2_{L^2(B(x,\,(3/2)t))}\,
\frac{dt}{t^{n+1}}\\
&&\ls \dint_{\epsilon/2}^{2R}
\lf\|t^{m}\nabla^m\lf(t^{2m} {L}e^{-t^{2m}{L}}(f)\r)\r\|^{2k/m}_{L^2(B(x,\,(3/2)t))}
\lf\|t^{2m} {L}e^{-t^{2m}{L}}(f)\r\|^{2(m-k)/m}_{L^2(B(x,\,(3/2)t))}\,
\frac{dt}{t^{n+1}}\\
&&\ls \lf\{\dint_{\epsilon/2}^{2R}
\lf\|t^{m}\nabla^m\lf(t^{2m} {L}e^{-t^{2m}{L}}(f)\r)\r\|^{2}_{L^2(B(x,\,(3/2)t))}
\frac{dt}{t^{n+1}} \r\}^{k/m}\\
&&\hs\times\lf\{\dint_{\epsilon/2}^{2R} \lf\|t^{2m} {L}e^{-t^{2m}{L}}
(f)\r\|^{2}_{L^2(B(x,\,(3/2)t))}
\frac{dt}{t^{n+1}}\r\}^{(m-k)/m}\\
&&\sim\lf[S^{\epsilon/2,\,2R,\,3/2}_{h,\,L,\,1}(f)(x)\r]^{2k/m}
\lf[S^{\epsilon/2,\,2R,\,3/2}_{L}(f)(x)\r]^{2(m-k)/m}.
\end{eqnarray*}
From \eqref{estimates of IIm}, we
deduce that
\begin{eqnarray*}
\mathrm{II}_k&&\ls \lf[S^{\epsilon/4,\,3R,\,2}_{L}(f)(x)\r]^{2k/m}
\lf[S^{\epsilon/2,\,2R,\,3/2}_{L}(f)(x)\r]^{2(m-k)/m}.
\end{eqnarray*}
Thus, it holds true that
\begin{eqnarray}\label{estimates of Ik}
\mathrm{I}_k&&\ls \lf[S^{\epsilon/2,\,2R,\,3/2}_{h,\,L}(f)(x)\r]^{\frac{1}{2}}
\lf[S^{\epsilon/4,\,3R,\,2}_{L}(f)(x)\r]^{k/(2m)}
\lf[S^{\epsilon/2,\,2R,\,3/2}_{L}(f)(x)\r]^{(m-k)/(2m)}.
\end{eqnarray}
Combining the estimates of \eqref{I0-Im} through \eqref{estimates of Ik}, $f\in L^2(\rn)$
and letting $\epsilon\to 0$ and $R\to \fz$, we conclude that, for almost every $x\in\rn$,
\begin{eqnarray*}
S_{L}(f)(x)&&\ls \lf[S_{h,\,L}^{3/2}(f)(x)\r]^{\frac{1}{2}}
\lf[S_{L}^{3/2}(f)(x)\r]^{\frac{1}{2}}+ \lf[{S}^{3/2}_{h,\,L}(f)(x)\r]^{\frac{1}{2}}
\lf[S^{2}_{L}(f)(x)\r]^{\frac{1}{2}}\\
&&\hs+\dsum_{k=1}^{m-1}\lf[S_{h,\,L}^{3/2}(f)(x)\r]^{\frac{1}{2}}
\lf[S^{2}_{L}(f)(x)\r]^{\frac{k}{2m}}\lf[S^{3/2}_{L}(f)(x)\r]^{\frac{m-k}{2m}},
\end{eqnarray*}
which immediately shows that there exists a positive constant
$C_0$ such that, for almost every $x\in\rn$,
\begin{eqnarray}\label{estimates of SL2}
S_{L}(f)(x)\le  C_0\lf[ S_{h,\,L}^{2}(f)(x)\r]^{\frac{1}{2}}
\lf[S^{2}_{L}(f)(x)\r]^{\frac{1}{2}}.
\end{eqnarray}

Similarly, by following the same line of the proof of \eqref{estimates of SL2},
we conclude that, for all $k\in\zz_+$ and almost every $x\in\rn$,
\begin{eqnarray*}
S^{2^k}_{L}(f)(x)\le  C_0\lf[ S_{h,\,L}^{2^{k+1}}(f)(x)\r]^{\frac{1}{2}}
\lf[S^{2^{k+1}}_{L}(f)(x)\r]^{\frac{1}{2}},
\end{eqnarray*}
which, combined with the definition of $\mathcal{A}_k$ in \eqref{meqnAk},
implies that, for all $f\in L^2(\rn)$ and almost every $x\in\rn$,
\begin{eqnarray*}
\mathcal{A}_k(F)(x)\le  C_0\lf[ \mathcal{A}_{k+1}(G)(x)\r]^{\frac{1}{2}}
\lf[\mathcal{A}_{k+1}(F)(x)\r]^{\frac{1}{2}},
\end{eqnarray*}
where $F:=t^{2m}Le^{-t^{2m}L}(f)$ and $G:=(t\nabla)^m e^{-t^{2m}L}(f)$.
Moreover, since $f\in L^2(\rn)$, we know that $F$, $G\in T^2(\rr^{n+1}_+)$.
Thus, by Lemma \ref{mlem 1}, we conclude that, for all $p\in (0,\fz)$ and $f\in L^2(\rn)$,
$\|F\|_{T^p(\rr^{n+1}_+)}\ls \|G\|_{T^p(\rr^{n+1}_+)}$, which implies that, for all $p\in (0,\fz)$ and $f\in L^2(\rn)$,
$\|S_L(f)\|_{L^p(\rn)}\ls \|S_{h,\,L}(f)\|_{L^p(\rn)}$ and hence completes the proof of Proposition
\ref{domination of square functions}.
\end{proof}

We also need the boundedness
of $S_{L,\,k}$ and $S_{h,\,L,\,k}$ in $L^q(\rn)$ as follows, which, when $k=1$,
was pointed out in \cite[p.\,68]{au07} without any details.

\begin{lemma}\label{lem 3.3}
Let $L$ be as in \eqref{L homo} and satisfy the Ellipticity condition
$(\mathcal{E}_0)$. Let $S_{L,\,k}$ and $S_{h,\,L,\,k}$ be the same, respectively, as
in \eqref{eqn KSF} and \eqref{gradient square function for heat}. Then
\begin{enumerate}
\item [{\rm (i)}] for all $k\in\nn$ and $q\in (p_-(L),\, p_+(L))$,
there exists a positive constant $C$ such that, for all
$f\in L^2(\rn)\cap L^q(\rn)$,
\begin{eqnarray*}
\lf\|S_{L,\,k} (f)\r\|_{L^q(\rn)}\le C\|f\|_{L^q(\rn)};
\end{eqnarray*}
\item [{\rm (ii)}] for all $k\in\zz_+$ and
$q\in (q_-(L),\, q_+(L))$,
there exists a positive constant $C$ such that, for all
$f\in L^2(\rn)\cap L^q(\rn)$,
\begin{eqnarray*}
\lf\|S_{h,\,L,\,k} (f)\r\|_{L^q(\rn)}\le C\|f\|_{L^q(\rn)}.
\end{eqnarray*}
\end{enumerate}
\end{lemma}

\begin{proof}
Observe that, by Proposition \ref{pro Lp semiL0}(iii),
we know that $L$ satisfies all the assumptions of \cite[Theorem 2.13]{bckyy13}
and, as a consequence, we obtain (i) of Lemma \ref{lem 3.3}.

The proof of (ii) of this lemma is similar to that of (i). We only need to
replace the $m$-$L^p$-$L^q$ off-diagonal estimates from Proposition \ref{pro Lp semiL0}(iii),
in the proof of \cite[Theorem 2.13]{bckyy13}, by the corresponding
$m$-$L^p$-$L^q$ off-diagonal estimates of the gradient semigroups
from Proposition \ref{pro Lp semiL0}(iv),
the details being omitted. This finishes
the proof of Lemma \ref{lem 3.3}.
\end{proof}

The next proposition presents an equivalence between the $H_L^p(\rn)$ norm,
defined via the square function $S_L$,
and the $L^p(\rn)$ norm when $p\in(p_+(L),\,p_+(L))$. Recall that
this conclusion was pointed out in \cite[p.\,68]{au07} without giving
any details.

\begin{proposition}\label{cor HLP}
Let $L$ be as in \eqref{L homo} and satisfy the Ellipticity condition
$(\mathcal{E}_0)$, and let $p\in(p_-(L),\,p_+(L))$, with $p_-(L)$ and $p_+(L)$ as in
Proposition \ref{pro Lp semiL0}, and $S_{L}$  be  as
in \eqref{eqn KSF} with $k=1$ and $\lz=1$.
Then there exists a positive constant $C$ such that,
for all $f\in L^2(\rn)\cap L^p(\rn)$,
\begin{eqnarray}\label{eqn cor}
\frac{1}{C}\lf\|f\r\|_{L^p(\rn)}\le \lf\|S_{L}(f)\r\|_{L^p(\rn)}\le C\lf\|f\r\|_{L^p(\rn)}.
\end{eqnarray}
\end{proposition}

\begin{proof}
It is easy to see that the second inequality of \eqref{eqn cor} is a direct consequence of
Lemma \ref{lem 3.3}(i) in the case when $k=1$.

We now prove the first inequality of \eqref{eqn cor}. Let $f\in L^2(\rn)\cap L^p(\rn)$.
For all $p\in(p_-(L),\,p_+(L))$ and
$g\in L^2(\rn)\cap L^{p'}(\rn)$ satisfying $\|g\|_{L^{p'}(\rn)}\le 1$
with $\frac{1}{p}+\frac{1}{p'}=1$,
by the Calder\'on reproducing formula \eqref{eqn Calderon repro L0}
with $M=0$, duality, Fubini's theorem, H\"older's inequality
and Lemma \ref{lem 3.3}(i), we know that
\begin{eqnarray}\label{eqn DA}
\nonumber\lf|\langle f,\,g \rangle_{L^2(\rn)}\r|&&\sim\lf|\lf\langle \dint_{0}^\fz
\lf(t^{2m}L\r)^2e^{-2t^{2m}L}(f)\,\frac{dt}{t},\,g \r\rangle_{L^2(\rn)}\r|\\
&&\nonumber\sim\lf|\dint_{0}^\fz\lf\langle
t^{2m}Le^{-t^{2m}L}(f),\,t^{2m}L^*e^{-t^{2m}L^*}(g) \r\rangle_{L^2(\rn)}\,\frac{dt}{t}\r|\\
&&\sim\lf\|S_L(f)\r\|_{L^p(\rn)}\lf\|S_{L^*}(g)\r\|_{L^{p'}(\rn)}
\ls\lf\|S_L(f)\r\|_{L^p(\rn)}
\end{eqnarray}
and hence
\begin{eqnarray*}
\lf\|f\r\|_{L^p(\rn)}\ls \|S_{L}(f)\|_{L^p(\rn)}.
\end{eqnarray*}
This finishes the  proof of the first inequality of \eqref{eqn cor} and hence
Proposition \ref{cor HLP}.
\end{proof}

\begin{remark}\label{mod rem1}
For all $p\in(0,\,\fz)$, let $\mathbb{H}^p_L(\rn)$ be the space defined
as in Definition \ref{Hardy space aoociated with L0 df}. From Proposition \ref{cor HLP} and
the definition of $\mathbb{H}^p_L(\rn)$, it follows that, for all $p\in(p_-(L),\,p_+(L))$,
$$\mathbb{H}^p_L(\rn)=L^2(\rn)\cap L^p(\rn).$$
\end{remark}

With the help of Propositions \ref{domination of square functions} and
\ref{cor HLP}, we obtain the following corollary.

\begin{corollary}\label{pro domination of square functions K}
Let $L$ be as in \eqref{L homo} and satisfy the Ellipticity condition
$(\mathcal{E}_0)$.
Then, for all $p\in(0,\,p_+(L))$,
there exists a positive constant $C_{(p)}$, depending on $p$, such that, for all $f\in L^2(\rn)$,
\begin{eqnarray}\label{slk}
\lf\|f\r\|_{H_L^p(\rn)}\le C_{(p)} \lf\|S_{h,\,L}(f)\r\|_{L^p(\rn)}.
\end{eqnarray}
\end{corollary}

\begin{proof}
If $p\in(0,\,2]$,
Corollary \ref{pro domination of square functions K} is an immediately
consequence of Proposition \ref{domination of square functions} and
Definition \ref{Hardy space aoociated with L0 df}.

If $p\in(2,\,p_+(L))$, Corollary \ref{pro domination of square functions K}
follows from Lemma \ref{lem relation Hardy Lebesgue L0}
and Propositions \ref{domination of square functions} and
\ref{cor HLP}. This finishes the proof of Corollary
\ref{pro domination of square functions K}.
\end{proof}

The next proposition shows that the $L^p(\rn)$ quasi-norm of $S_{h,\,L}$,
as in \eqref{gradient square function for heat} with $\lz=1$,
can be controlled by that of the non-tangential maximal function
$\mathcal{N}^\gz_{h,\,L}$ as in
\eqref{non-tangential maximal function for heat}.

\begin{proposition}\label{domination of heat by maximal}
Let $L$ be as in \eqref{L homo} and satisfy the
Strong ellipticity condition $(\mathcal{E}_1)$, and let {$p\in(0,\,2)$}.
Then there exist positive constants $\gz$ and $C$ such that, for all $f\in L^2(\rn)$,
\begin{eqnarray*}
\lf\|S_{h,\,L}(f)\r\|_{L^p(\rn)}\le C \lf\|\mathcal{N}^\gz_{h,\,L}(f)\r\|_{L^p(\rn)}.
\end{eqnarray*}
\end{proposition}

\begin{proof}
To prove Proposition \ref{domination of heat by maximal}, we first introduce
some notation. Let  $\sz\in(0,\,\fz)$.  Assume that $\gz\in(0,\,\fz)$,
whose exact value will be determined later. Let
$$E:=\lf\{x\in\rn:\ \mathcal{N}_{h,\,L}^{\gz}(f)(x)\le \sz\r\}.$$
Its \emph{subset $E^*$ of global $1/2$ density} is defined by
\begin{eqnarray*}
E^*:=\lf\{x\in\rn:\ \text{for all balls}\ B(x,\,r) \ \text{in}\
\rn,\ \  \frac{|E\cap B(x,\,r)|}{|B(x,\,r)|}\ge \frac{1}{2}\r\}.
\end{eqnarray*}
For all $0<\epsilon\ll R<\fz$, let $\mathcal{R}^{\epsilon,\,R,\,\gz}(E^*):=
\cup_{x\in E^*}\bgz^{\epsilon,\,R,\,\gz}(x)$ be the
\emph{sawtooth region based on} $E^*$ and $\mathcal{B}^{\epsilon,\,R,\,\gz}(E^*)$
the \emph{boundary} of $\mathcal{R}^{\epsilon,\,R,\,\gz}(E^*)$.
Moreover, for all $y\in\rn$ and $t\in(0,\,\fz)$, let $u(y,\,t):=e^{-t^{2m}L}(f)(y)$.
By Fubini's theorem, we find that
\begin{eqnarray*}
\dint_{E^*}\lf[S_{h,\,L}^{\epsilon,\,R,\,1/2}(f)(x)\r]^{2}\,dx\sim
\iint_{\mathcal{R}^{\epsilon,\,R,\,1/2}(E^*)}
t^{2m}\lf|\nabla^m u(y,\,t)\r|^{2}\,\frac{dy\,dt}{t},
\end{eqnarray*}
where $S_{h,\,L}^{\epsilon,\,R,\,1/2}(f)(x)$ is defined as in \eqref{2.x1}.

Now, let $\eta\in C_c^\fz(\mathcal{R}^{\epsilon/2,\,2R,\,3/2}(E^*))$
be a smooth cut-off function satisfying $0\le \eta\le 1$, $\eta\equiv1$
on $\mathcal{R}^{\epsilon,\,R,\,1/2}(E^*)$ and, for all $k\in\nn$ with $k\le m$
and $(x,\,t)\in \mathcal{R}^{\epsilon/2,\,2R,\,3/2}(E^*)$,
$$\lf|\nabla_x^k\eta(x,\,t)\r|\ls \frac{1}{t^{k}}$$ and
$|\pat_t\eta(x,\,t)|\ls \frac{1}{t}$.
These assumptions, together with the Strong ellipticity condition
$(\mathcal{E}_1)$, imply that
\begin{eqnarray*}
&&\iint_{\mathcal{R}^{\epsilon,\,R,\,1/2}(E^*)}
t^{2m}\lf|\nabla^m u(y,\,t)\r|^{2}\,\frac{dy\,dt}{t}\\
&&\hs\ls \Re e
\lf\{\iint_{\mathcal{R}^{\epsilon,\,R,\,1/2}(E^*)} \lf[t^{2m} \dsum_{|\az|=
m=|\bz|}a_{\az,\,\bz}(y)\pat^\bz u(y,\,t)
\overline{\pat^\az u(y,\,t)}\r]\,\frac{dy\,dt}{t} \r\}\\
&&\hs\ls \Re e
\lf\{\iint_{\rr^{n+1}_+} \lf[t^{2m} \dsum_{|\az|=
m=|\bz|}a_{\az,\,\bz}(y) \pat^\bz u(y,\,t)
\overline{\pat^\az u(y,\,t)}\eta(y,\,t)\r]\,\frac{dy\,dt}{t} \r\}.
\end{eqnarray*}
From this,  we further deduce that
\begin{eqnarray}\label{divide into Jk}
\nonumber&&\iint_{\mathcal{R}^{\epsilon,\,R,\,1/2}(E^*)}
t^{2m}\lf|\nabla^m u(y,\,t)\r|^{2}\,\frac{dy\,dt}{t}\\
&&\nonumber\hs\ls \lf|\Re e \lf\{\iint_{\rr^{n+1}_+} \lf[t^{2m} \dsum_{|\az|=
m=|\bz|}(-1)^{m} \pat^{\az}\lf(\eta a_{\az,\,\bz} \pat^\bz u\r)(y,\,t)
\overline{u(y,\,t)}\r]\,\frac{dy\,dt}{t} \r\}\r|\\
&&\nonumber\hs\ls \dsum_{k=0}^m \lf|\Re e \lf\{\iint_{\rr^{n+1}_+}
\lf[t^{2m} \dsum_{|\az|=m=|\bz|}\dsum_{|\wz\az|=k,\,\wz\az\le \az}
C_{(\az,\,\wz\az)}(-1)^{m} \pat^{\wz\az}\eta(y,\,t)\r.\r.\r.\\
&&\lf.\lf.\hs\hs\times \pat^{\az-\wz\az}\lf(a_{\az,\,\bz}
\pat^\bz u\r)(y,\,t)\overline{u(y,\,t)}\Bigg]\,\frac{dy\,dt}{t} \r\}\r|
=:\dsum_{k=0}^m\mathrm{J}_k,
\end{eqnarray}
where, for any multi-indices $\az$ and $\wz\az$ as above, $C_{(\az,\,\wz\az)}$
denotes a positive constant depending on $\az$ and $\wz\az$.

We first bound $\mathrm{J}_0$. Since, for all $(y,\,t)\in\rr^{n+1}_+$,
$\frac{\pat}{\pat t}u(y,\,t)=-2m t^{2m-1} {L}(u)(y,\,t)$, we know that
\begin{eqnarray*}
\frac{\pat}{\pat t}\lf|u(y,\,t)\r|^2&&=-2m t^{2m-1} {L}(u)(y,\,t)\overline{u(y,\,t)}
-2m t^{2m-1} u(y,\,t)\overline{L(u)(y,\,t)}\\
&&=-4m t^{2m-1}\Re e\lf\{ {L}(u)(y,\,t)\overline{u(y,\,t)}\r\},
\end{eqnarray*}
which, together with integration by parts and properties of the cut-off
function $\eta$,  shows that
\begin{eqnarray*}
\mathrm{J}_0&&\sim \lf|\Re e \lf\{\iint_{\rr^{n+1}_+}
t^{2m-1}  \eta(y,\,t) {L} (u)(y,\,t)\overline{u(y,\,t)}\,dy\,dt \r\}\r|\\
&&\sim \lf|\iint_{\rr^{n+1}_+} \eta(y,\,t) \frac{\pat}{\pat t}\lf|u(y,\,t)\r|^2\,
dy\,dt\r|\ls \iint_{\mathcal{R}^{\epsilon/2,\,2R,\,3/2}(E^*)\setminus
\mathcal{R}^{\epsilon,\,R,\,1/2}(E^*)}  \lf|u(y,\,t)\r|^2\,\frac{dy\,dt}{t}.
\end{eqnarray*}
To estimate the last term in the above formulae, we let
\begin{eqnarray}\label{2.x2}
\wz{\mathcal{B}}^\epsilon(E^*):=\lf\{(x,\,t)\in\rn\times (\epsilon/2,\,\epsilon):
\ d(x,\,E^*)<\frac{3}{2}t\r\},
\end{eqnarray}
\begin{eqnarray}\label{2.x3}
\wz{\mathcal{B}}^R(E^*):=\lf\{(x,\,t)\in\rn\times (R,\,2R):
\ d(x,\,E^*)<\frac{3}{2}t\r\}
\end{eqnarray}
and
\begin{eqnarray}\label{2.x4}
\wz{\mathcal{B}}_0(E^*):=\lf\{(x,\,t)\in\rn\times (\epsilon/2,\,2R):
\ \frac{1}{2}t\le d(x,\,E^*)<\frac{3}{2}t\r\}.
\end{eqnarray}
It is easy to see that $(\mathcal{R}^{\epsilon/2,\,2R,\,3/2}(E^*)\setminus
\mathcal{R}^{\epsilon,\,R,\,1/2}(E^*))\subset (\wz{\mathcal{B}}^\epsilon(E^*)\cup
\wz{\mathcal{B}}^R(E^*)\cup \wz{\mathcal{B}}_0(E^*))$.
For any $(y,\,t)\in \wz{\mathcal{B}}^\epsilon(E^*)$, we find that
there exists $x\in E^*$ such that $|x-y|<\frac{3}{2}t$. Moreover, from the definition
of $E^*$, it follows that, for all $t\in(0,\,\fz)$,
\begin{eqnarray*}
|E\cap B(x,\,t)|\ge \frac{1}{2}|B(x,\,t)|=\frac{1}{2}\omega_nt^n,
\end{eqnarray*}
where $\omega_n:=|B(x,\,1)|=|B(0,\,1)|$.
Thus, $|E\cap B(y,\,3t)|\ge \frac{1}{2} \omega_n t^n$, which, combined with Fubini's theorem,
implies that
\begin{eqnarray}\label{estimates on B epsilon}
\nonumber&&\iint_{\wz{B}^\epsilon (E^*)}\lf|u(y,\,t)\r|^2\,\frac{dy\,dt}{t}\\
&&\hs\nonumber\ls\iint_{\wz{B}^\epsilon (E^*)}
\lf[\dint_{E\cap B(y,\,3t)}\lf|u(y,\,t)\r|^2
\,dz\r]\,\frac{dy\,dt}{t^{n+1}}\\
&&\hs\nonumber\ls\dint_{\epsilon/2}^{\epsilon}\dint_{E}
\lf[\frac{1}{t^n}\dint_{B(z,\,3t)}\lf|e^{-t^{2m}{L}}(f)(y)\r|^2\,dy\r]
\,\frac{dz\,dt}{t}\\
&&\hs\nonumber\ls \dint_{\epsilon/2}^\epsilon
\dint_{E}\lf|\sup_{(x,\,t)\in\Gamma^3(z)}
\lf\{\frac{1}{\omega_n(3 t)^n} \int_{B(x,\,3 t)}
\lf|e^{-t^{2m} {L}}(f)(y)\r|^2\,dy\r\}^{\frac{1}{2}}\r|^2\,\frac{dz\,dt}{t}\\
&&\hs\sim\dint_{\epsilon/2}^\epsilon
\dint_{E}\lf|\mathcal{N}_{h,\,L}^3(f)(z)\r|^2\,\frac{dz\,dt}{t}\sim
\dint_{E}\lf|\mathcal{N}_{h,\,L}^3(f)(z)\r|^2\,dz.
\end{eqnarray}
Similarly, we have
\begin{eqnarray}\label{estimates on B R}
\iint_{\wz{\mathcal{B}}^R (E^*)}\lf|u(y,\,t)\r|^2\,
\frac{dy\,dt}{t}\ls \dint_{E}\lf|\mathcal{N}_{h,\,L}^3(f)(z)\r|^2\,dz.
\end{eqnarray}
To estimate the integrand on the region $\wz{\mathcal{B}}_0(E^*)$, let
$\{B(x_k,\,r_k)\}_k$ be Whitney's covering of $B^*$,
where $B^*:=(E^*)^{\complement}$. Then we see that
\begin{itemize}
\item [(i)] $\cup_{k} B(x_k,\,r_k)=B^*$;
\item [(ii)] there exist positive constants $C_1$ and $C_2\in(0,\,1)$  such that, for all $k$,
$$C_1d(x_k,\,E^*)\le r_k\le C_2 d(x_k,\,E^*);$$
\item [(iii)] there exists a positive constant $C_3$ such that, for all $x\in B^*$,
$\sum_k\chi_{B(x_k,\,r_k)}(x)\le C_3$.
\end{itemize}
From these, we deduce that
\begin{eqnarray*}
\iint_{\wz{\mathcal{B}}_0(E^*)}\lf|u(y,\,t)\r|^2\,\frac{dy\,dt}{t}&&\ls
\dsum_k\dint_{\frac{2}{3}r_k(\frac{1}{C_2}-1)}^{2r_k(\frac{1}{C_1}+1)}\dint_{B(x_k,\,r_k)}
\lf|u(y,\,t)\r|^2\,\frac{dy\,dt}{t}\\
&&\ls\dsum_k\dint_{\frac{2}{3}r_k(\frac{1}{C_2}-1)}^{2r_k(\frac{1}
{C_1}+1)}r_k^n
\lf[\frac{1}{t^n}\dint_{B(x_k,\,r_k)}\lf|u(y,\,t)\r|^2\,dy\r]\,\frac{dt}{t}.
\end{eqnarray*}
By the fact $E^*\subset E$, we know that $d(x_k,\,E)\le d(x_k,\,E^*)\le \frac{C_2}
{(1-C_2)C_1}t$. Thus, by taking $\gz\in(\frac{C_2}{(1-C_2)C_1},\,\fz)$,
we conclude that
\begin{eqnarray}\label{estimates on B plus}
\iint_{\wz{\mathcal{B}}_0(E^*)}\lf|u(y,\,t)\r|^2\,\frac{dy\,dt}{t}\ls
\dsum_k r_k^n \lf[\dsup_{z\in E} \mathcal{N}_{h,\,L}^\gz (f)(z)\r]^2\ls |B^*|
\lf[\dsup_{z\in E} \mathcal{N}_{h,\,L}^\gz (f)(z)\r]^2.
\end{eqnarray}
Combining the estimates of \eqref{estimates on B epsilon}, \eqref{estimates on B R}
and \eqref{estimates on B plus}, we see that
\begin{eqnarray}\label{estimates of J0}
\mathrm{J}_0\ls \dint_{E}\lf|\mathcal{N}_{h,\,L}^3(f)(z)\r|^2\,dz+
|B^*|\lf[\dsup_{z\in E} \mathcal{N}_{h,\,L}^\gz (f)(z)\r]^2.
\end{eqnarray}

Now, we turn to the estimates of $\mathrm{J}_k$
for all $k\in\{1,\ldots,\,m\}$.
Using integration by parts and H\"older's inequality, we write
\begin{eqnarray}\label{eqn estJk}
\nonumber\mathrm{J}_k&&\sim \lf|\Re e \lf\{\iint_{\rr^{n+1}_+}
t^{2m}\!\dsum_{|\az|=m=|\bz|}\dsum_{|\wz\az|=k,\,\wz\az\le \az}
\wz C_{(\az,\,\wz\az,\,m)}\r.\r.\\
&&\nonumber\hs\times\lf.a_{\az,\,\bz}(y)\pat^{\bz}u(y,\,t)\overline{\pat^{\az-\wz{\az}}
\lf((\pat^{\wz{\az}}\eta)u\r)(y,\,t)}\,\frac{dy\,dt}{t}\Bigg\}\r|\\
&&\nonumber\ls\dsum_{|\az|=m=|\bz|}\dsum_{|\wz\az|=k,\,\wz\az\le \az}
\lf\{ \iint_{\mathcal{R}^{\epsilon/2,\,2R,\,3/2}(E^*)\setminus
\mathcal{R}^{\epsilon,\,R,\,1/2}(E^*)}\lf|t^{m}
\pat^\bz u(y,\,t)\r|^2\,\frac{dy\,dt}{t}\r\}^{\frac{1}{2}}\\
&&\nonumber\hs\times\lf\{\dint_{\mathcal{R}^{\epsilon/2,\,2R,\,3/2}
(E^*)\setminus\mathcal{R}^{\epsilon,\,R,\,1/2}(E^*)}
\lf|t^{m-k}\pat^{\az-\wz\az}\lf(t^k [\pat^{\wz\az}\eta]u\r)
(y,\,t)\r|^2\,\frac{dy\,dt}{t}\r\}^{\frac{1}{2}}\\
&&=:\dsum_{|\az|=m=|\bz|}\dsum_{|\wz\az|=k,\,\wz\az\le \az}
\mathrm{J}_{\az,\,\wz\az,\,\bz,\,1}
\times\mathrm{J}_{\az,\,\wz\az,\,\bz,\,2},
\end{eqnarray}
where, for any $m\in\nn$ and any multi-indices $\az$ and $\wz\az$ as above,
$\wz C_{(\az,\,\wz\az,\,m)}$ denotes a constant depending on $\az$, $\wz\az$ and $m$.

We first control $\mathrm{J}_{\az,\,\wz\az,\,\bz,\,1}$.
Let $\wz{\mathcal{B}}^{\epsilon}(E^*)$, $\wz{\mathcal{B}}^{R}(E^*)$
and $\wz{\mathcal{B}}_0(E^*)$ be, respectively, as in
\eqref{2.x2}, \eqref{2.x3} and \eqref{2.x4}.
Similar to \eqref{estimates on B epsilon}, we obtain
\begin{eqnarray}\label{estimates similar to B epsilon}
\iint_{\wz{\mathcal{B}}^{\epsilon}(E^*)}
\lf|t^{m} \nabla^m u(y,\,t)\r|^2\,\frac{dy\,dt}{t}\ls
\dint_{E}\lf[\dint_{\epsilon/2}^{\epsilon}\frac{1}{t^n}
\dint_{B(z,\,3t)}\lf|t^m\nabla^m u(y,\,t)\r|^2\,\frac{dy\,dt}{t}\r]
\,dz.
\end{eqnarray}
This, together with the parabolic Caccioppoli's inequality
\eqref{caccioppoli's inequality3 with epsilon}, implies that
\begin{eqnarray}\label{estimates on B epsilon for k}
\nonumber\iint_{\wz{\mathcal{B}}^{\epsilon}(E^*)}
\lf|t^{m} \nabla^m u(y,\,t)\r|^2\,\frac{dy\,dt}{t}
&&\ls \dint_{E}\dint_{\epsilon/4}^{2\epsilon}
\frac{1}{t^{n+2m}}\dint_{B(z,\,6t)}\lf[\lf|t^mu(y,\,t)\r|^2 \r]\,\frac{dy\,dt\,dz}{t}\\
&&\ls\dint_{E}\lf[\mathcal{N}_{h,\,L}^6 (f)(y)\r]^2\,dy.
\end{eqnarray}

Similarly, resting on estimates \eqref{estimates on B R},
\eqref{estimates on B plus}, \eqref{estimates on B epsilon for k}
and the parabolic Caccioppoli's inequality \eqref{caccioppoli's inequality3 with epsilon},
we conclude that there exists a positive
constant $\gz\in(0,\,\fz)$ large enough such that
\begin{eqnarray*}
\iint_{\wz{\mathcal{B}}^{R}(E^*)\cup \wz{\mathcal{B}}_0(E^*)}
\lf|t^{m} \nabla^m u(y,\,t)\r|^2\,\frac{dy\,dt}{t}\ls
\dint_{E}\lf[\mathcal{N}_{h,\,L}^\gz (f)(y)\r]^2\,dy+ |B^*|
\lf[\dsup_{x\in E}\mathcal{N}_{h,\,L}^\gz (f)(x)\r]^2.
\end{eqnarray*}
This, combined with \eqref{estimates on B epsilon for k}, implies that
\begin{eqnarray}\label{estimates of Jk1}
\mathrm{J}_{\az,\,\wz\az,\,\bz,\,1}\ls
\lf\{\dint_{E}\lf[\mathcal{N}_{h,\,L}^\gz (f)(y)\r]^2\,dy+ |B^*|
\lf[\dsup_{x\in E}\mathcal{N}_{h,\,L}^\gz (f)(x)\r]^2\r\}^{\frac{1}{2}}.
\end{eqnarray}

The estimate of $\mathrm{J}_{\az,\,\wz\az,\,\bz,\,2}$
can be obtained by using the definition of $\eta$, the interpolation
inequality \eqref{interpolation of sobolev space} and the estimates of
$\mathrm{J}_0$  and $\mathrm{J}_{\az,\,\wz\az,\,\bz,\,1}$. By the estimate of
$\mathrm{J}_{\az,\,\wz\az,\,\bz,\,2}$, \eqref{divide into Jk} and  the estimates of
\eqref{estimates of J0} through \eqref{estimates of Jk1}, we see that
\begin{eqnarray*}
\dint_{\mathcal{R}^{2\epsilon,\,R,\,1/2}(E^*)}
t^{2m}\lf|\nabla^m u(y,\,t)\r|^{2}\,\frac{dy\,dt}{t}\ls
\dint_{E}\lf[\mathcal{N}_{h,\,L}^\gz (f)(y)\r]^2\,dy+ |B^*|
\lf[\dsup_{x\in E}\mathcal{N}_{h,\,L}^\gz (f)(x)\r]^2,
\end{eqnarray*}
where $\gz\in(0,\,\fz)$ is a sufficiently large constant. By this and
an argument similar to that used in \cite[(6.31) through (6.37)]{hm09},
we then complete the proof of Proposition \ref{domination of heat by maximal}.
\end{proof}

We are now in a position to prove our main result of this article.

\begin{proof}[Proof of Theorem
\ref{non-tangential heat maximal functiona characterization}]
The inclusion that $H_{\mathcal{N}_{h,\,L}}^p(\rn)\subset H_{L}^p(\rn)$, for all
$p\in(0,\,2)$, is a direct
consequence of Propositions \ref{domination of square functions}
and \ref{domination of heat by maximal},
and Corollary \ref{pro domination of square functions K}.
We now turn to the proof of the inclusion
$H_{\mathcal{N}_{h,\,L}}^p(\rn)\subset H_{L}^p(\rn)$ for all $p\in[2,\,p_+(L))$.
Using Lemma \ref{lem relation Hardy Lebesgue L0}, we are reduced to proving that,
for all $p\in[2,\,p_+(L))$ and $f\in \mathbb{H}^p_L(\rn)$,
\begin{equation}\label{eqn m1}
\|f\|_{L^p(\rn)}\ls \|\mathcal{N}_{h,\,L}(f)\|_{L^p(\rn)}.
\end{equation}

To show \eqref{eqn m1}, for $p\in[2,\,p_+(L))$, let $\psi\in L^{p'}(\rn)$ satisfying that
$\|\psi\|_{L^{p'}(\rn)}\le 1$ and
\begin{equation}\label{eqn m2}
\|f\|_{\lp}\ls\lf|\int_{\R^n}f(y)\overline{\psi(y)}\,dy\r|,
\end{equation}
here and hereafter, $1/p+1/p'=1$. We first claim that
\begin{equation}\label{cor5}
\lf|\int_{\R^n}f(y)\overline{\psi(y)}\,dy\r|=\lim_{t\to 0^+}\lf|\int_{\R^n}e^{-t^{2m}L}(f)(y)
\lf[\frac{1}{t^n}\int_{B(y,t)}\overline{\psi(x)} dx\r]\,dy\r|,
\end{equation}
here and hereafter, ``$t\to 0^+$" means that ``$t>0$ and $t\to 0$".
Indeed, if the claim \eqref{cor5} holds true, then, from Fubini's theorem, H\"{o}lder's inequality
and Remark \ref{rem def MA}(ii), we deduce that
\begin{align*}
\lf|\int_{\R^n}f(y)\overline{\psi(y)}\,dy\r|&\leq \sup_{t>0}\lf\{\int_{\R^n}|\psi(x)|\lf[\frac{1}{t^n}\int_{B(x,\, t)}
\lf|e^{-t^{2m}L}(f)(y)\r|\,dy\r]dx\r\}\\
&\ls \|\psi\|_{L^{p'}(\R^n)}\sup_{t>0}\lf\{\int_{\R^n}\lf[\frac{1}{t^n}\int_{B(x,\, t)}|e^{-t^{2m}L}(f)(y)|^2\,dy\r]^{\frac{p}{2}}\,dx\r\}^{\frac{1}{p}}\notag\\
&\ls\left \{\int_{\R^n}\lf(\sup_{t>0}\lf[\frac{1}{t^n}\int_{B(x,t)}|e^{-t^{2m}L}(f)(y)|^2 \,dy\r]^{\frac{1}{2}}\r)^p\,dx\right\}^{\frac{1}{p}}\notag\\
&\nonumber\sim\|\mathcal{R}_{h,\,L}(f)\|_{\lp}\sim \|\mathcal{N}_{h,\,L}(f)\|_{\lp},
\end{align*}
which, together with \eqref{eqn m2}, implies that, for all $p\in [2,\,p_{+}(L))$,
\eqref{eqn m1} holds true.

Thus, to finish the proof of \eqref{eqn m1},
it remains to show the claim \eqref{cor5}. By H\"{o}lder's inequality and an
elementary calculation, we see that
\begin{align*}
& \lf|\int_{\R^n}f(y)\overline{\psi(y)}\,dy-\int _{\R^n} e^{-t^{2m}L}(f)(y)\lf[\frac{1}{t^n}
\int_{B(y,\,t)}\overline{\psi(x)}dx\r]\,dy\r|\\
&\hs\leq \lf| \int_{\R^n}\lf\{f(y)-e^{-t^{2m}L}(f)(y)\r\}\overline{\psi(y)}\,dy\r|\\
&\hs\hs+\lf|\int_{\R^n}e^{-t^{2m}L}(f)(y)\lf\{\overline{\psi(y)}
-\frac{1}{t^n}\int_{B(y,\,t)}\overline{\psi(x)}dx\r\}\,dy\r|\\
&\hs\le \lf\{\int_{\R^n}\lf|(e^{-t^{2m}L}-I)(f)(y)\r|^p \,dy\r\}^{\frac{1}{p}}\lf\{\int_{\R^n}
\lf|{\psi(y)}\r|^{p'}\,dy\r\}^{\frac{1}{p'}}\\
&\hs\hs+\lf\{\int_{\R^n}\lf|e^{-t^{2m}L}(f)(y)\r|^p\,dy\r\}^{\frac 1{p}}\lf\{\int_{\R^n}
\lf|\psi(y)-\frac{1}{t^n}\int _{B(y,\,t)}\psi(x)\,dx\r|^{p'}\,dy\r\}^{\frac{1}{p'}}\\
&\hs=: \mathrm{A}_t+\mathrm{B}_t.
\end{align*}

Notice that, by the fact that $\lim_{t\to 0^+}e^{-tz}=1$ for all complex numbers $z$ and
the fact that $L$ has a bounded functional calculus in $L^q(\rn)$ with $q\in (p_-(L),\,p_+(L))$
(which is a simple corollary of Proposition \ref{pro Lp semiL0}(iii) and \cite[Theorem 1.2]{bk03}),
we know that $\{e^{-tL}\}_{t>0}$ has the strong continuity in $L^q(\rn)$ for all $q\in (p_-(L),\,p_+(L))$.
Letting $t\to 0^+$, and using the strong continuity of the semigroup $\{e^{-tL}\}_{t>0}$ in $L^p(\rn)$ for
$p\in[2,\,p_+(L))$ and $\|\psi\|_{L^{p'}(\rn)}\le 1$,
we know that
$$\lim_{t\to 0^+} A_t=0.$$

In what follows, for any locally integrable function $f$, let
$\mathcal{M}(f)$ be the \emph{Hardy-Littlewood maximal function}
defined by setting, for all $x\in\rn$,
$$\mathcal{M}(f)(x):=\dsup_{B\ni x}\frac{1}{|B|}\dint_B \lf|f(x)\r|\,dx,$$
where the supremum is taken over all the balls in $\rn$ containing $x$. Observe that,
for all $y\in\rn$,
$$\lf|\psi(y)-\frac{1}{t^n}\int _{B(y,\,t)}\psi(x)\,dx\r|\le 2\mathcal{M}(\psi)(y)$$
and $\mathcal{M}(\psi)\in L^{p'}(\rn)$. From this, the Lebesgue dominated
convergence theorem and the Lebesgue differentiation theorem, together with
the uniformly boundedness of the semigroup
$\{e^{-tL}\}_{t>0}$ in $L^p(\rn)$ for
$p\in[2,\,p_+(L))$, we deduce that
$$\lim_{t\to 0^+}B_t=0,$$
which completes the  proof of the claim \eqref{cor5}. Thus, $H_{\mathcal{N}_{h,\,L}}^p(\rn)\subset H_{L}^p(\rn)$
for all $p\in(0,\,p_+(L))$.

To prove the inclusion $H_{L}^p(\rn)\subset H_{\mathcal{N}_{h,\,{L}}}^p(\rn)$,
we consider two cases. If $p\in(0,\,1]$,
by Theorem \ref{molecular characterization for Hardy space L0} and
Remark \ref{rem def MA},
we see that it suffices to show that, for all $(p,\,2,\,M,\,\epsilon)_{L}$-molecules
$\az$, $$\|\mathcal{R}_{h,\,{L}} (\az)\|_{L^p(\rn)}\ls 1,$$
where $\mathcal{R}_{h,\,{L}}$ is the radial heat maximal function defined as in
\eqref{radial maximal function for heat}. The latter estimate can be obtained by
using the same method as that used in the proof of \cite[Theorem 6.3]{hm09}, the
details being omitted here. This finishes the proof of Theorem
\ref{non-tangential heat maximal functiona characterization} for $p\in(0,\,1]$.

If $p\in(1,\,p_+(L))$, let $\mathcal{R}_{h,\,L}$ be the radial maximal function defined as in
\eqref{radial maximal function for heat} and, for any ball $B$ and $j\in\nn$,
let $S_j(B):=2^jB\setminus (2^{j-1}B)$ and $S_0(B):=B$.
Then, for any  $q\in (2,\,\fz)$, using Minkowski's inequality,
Proposition \ref{pro Lp semiL0} and the boundedness of $\mathcal{M}$
on $L^{q/2}(\rn)$,  we know that there exists a positive constant $\eta$
such that
\begin{eqnarray*}
&&\lf\|\mathcal{R}_{h,\,L}(f)\r\|_{L^q(\rn)}\\
&&\hs=\lf\|\sup_{t\in(0,\fz)}\lf\{\frac{1}{t^n}\dint_{B(\cdot,\, t)} \lf|e^{-t^{2m}{L}}
\lf(\dsum_{j\in\zz_+}\chi_{S_j(B(\cdot,\, t))}f\r)
(y)\r|^2\,dy\r\}^{\frac{1}{2}}\r\|_{L^q(\rn)}\\
&&\hs\ls\dsum_{j\in\zz_+}
\lf\|\sup_{t\in(0,\fz)}\lf[\frac{1}{t^{\frac n2}}\exp\lf\{-\frac{[d(B(\cdot,\, t),\,S_j(B(\cdot,\, t)))]^{2m/(2m-1)}}
{t^{{2m}/(2m-1)}}\r\}\lf\|f\r\|_{L^2(S_j(B(\cdot,\, t)))}\r]\r\|_{L^q(\rn)}\\
&&\hs\ls\dsum_{j\in\zz_+}2^{-j\eta}
\lf\|\sup_{t\in(0,\fz)}\lf[\frac{1}{(2^jt)^n}\dint_{2^jB(\cdot,\, t)}
\lf|f(x)\r|^2\,dx\r]^{\frac{1}{2}}\r\|_{L^q(\rn)}\\
&&\hs\ls\lf\|\lf[\mathcal{M}\lf(|f|^2\r)\r]^{1/2}\r\|_{L^q(\rn)}
\ls\lf\|f\r\|_{L^q(\rn)},
\end{eqnarray*}
which implies that $\mathcal{R}_{h,\,L}$
is bounded on $L^q(\rn)$. This, together with Remark \ref{rem def MA},
further shows that the non-tangential maximal function $\mathcal{N}_{h,\,{L}}$
is also bounded on $L^q(\rn)$ for all $q\in(2,\,\fz)$. By the case $p\in(0,\,1]$, Lemma
\ref{lem relation Hardy Lebesgue L0} and
the complex interpolation of $H_L^p(\rn)$ (see
Proposition \ref{pro interpolation Hardy L0} {together with \cite[p.\,52, Theorem]{Jan}}),
we see that, for all $p\in(0,\,p_+(L))$, $\mathcal{N}_{h,\,{L}}$ is bounded from
$H_L^p(\rn)$ to $L^p(\rn)$. This implies the inclusion
$H_{L}^p(\rn)\subset H_{\mathcal{N}_{h,\,L}}^p(\rn)$ for all $p\in(0,\,p_+(L))$
and hence finishes the proof of
Theorem \ref{non-tangential heat maximal functiona characterization}.
\end{proof}

With the help of Theorem \ref{non-tangential heat maximal functiona characterization},
we are able to prove Theorem \ref{non-tangential heat maximal functiona characterization gradient}.

\begin{proof}[Proof of Theorem \ref{non-tangential heat maximal functiona characterization gradient}]

We first point out that,  similar to Remark \ref{rem def MA}, we have
$$H_{\wz{\mathcal{N}}_{h,\,L}}^p(\rn)=H_{\wz{\mathcal{R}}_{h,\,L}}^p(\rn)$$
with equivalent quasi-norms. Moreover, by \eqref{radial maximal function for heat}, \eqref{non-tangential maximal function for heat},
\eqref{radial maximal function for heat q} and \eqref{non-tangential maximal function for heat q},
we immediately see that, for all $f\in L^2(\rn)$ and $x\in\rn$,
\begin{eqnarray*}
{\mathcal{N}}_{h,\,L}(x)\le \wz{\mathcal{N}}_{h,\,L}(f)(x),
\end{eqnarray*}
which, together with Theorem \ref{non-tangential heat maximal functiona characterization},
implies the inclusion that $H_{\wz{\mathcal{N}}_{h,\,L}}^p(\rn)
\subset H_L^p(\rn)$.
The proof of the inclusion that $H_L^p(\rn)\subset H_{\wz{\mathcal{R}}_{h,\,L}}^p(\rn)$
is similar to the corresponding part of the proof of Theorem \ref{non-tangential heat maximal functiona characterization}. Here, we only need to use the $L^2$ off-diagonal estimates
of the gradient semigroup $\{\sqrt{t}\nabla^me^{-t^{2m}L}\}_{t>0}$
to replace the $L^2$ off-diagonal estimates
of the semigroup $\{e^{-t^{2m}L}\}_{t>0}$ therein, which completes the proof
of Theorem \ref{non-tangential heat maximal functiona characterization gradient}.
\end{proof}

Now we turn to the proof of Proposition \ref{pro NTMFC2 L0}.

\begin{proof}[Proof of Proposition \ref{pro NTMFC2 L0}]
For any $(x,\,t)\in\rr^{n+1}_+$, following \cite[pp.\,59-60, 3.2(c)]{af03},
let $W_0^{m,\,2}(B(x,\,2t))$ be the \emph{Sobolev space over} $B(x,\,2t)$,
defined as the completion of $C_c^\fz(B(x,\,2t))$ with respect to
the \emph{norm} $\|\cdot\|_{W^{m,\,2}(B(x,\,2t))}$,
where, for all $\fai\in C_c^\fz(B(x,\,2t))$,
\begin{eqnarray*}
\|\fai\|_{W^{m,\,2}(B(x,\,2t))}:
=\lf[\dsum_{0\le |\az|\le m}\|\pat^\az \fai\|_{L^2(B(x,\,2t))}^2\r]^{1/2}.
\end{eqnarray*}
Since $f\in L^2(\rn)$, it follows that $e^{-t^{2m}L}(f)\in W^{m,\,2}(\rn)$. Thus,
we have
$$\psi_{x,\,t}e^{-t^{2m}L} (f)\in W_0^{m,\,2}(B(x,\,2t))
\subset W_0^{m-1,\,2}(B(x,\,2t)).$$

Recall the following Poincar\'e's inequality from \cite[p.\,69, Theorem 3.2.1]{Mo08}:
for all $k\in\{0,\,\ldots,\,m-1\}$ and $v\in W_0^{m-1,\,2}(B(x,\,2t))$,
it holds true that, for all $(x,\,t)\in\rr^{n+1}_+$,
\begin{eqnarray*}
\dint_{B(x,\,2t)}|\nabla^k v(y)|^2\,dy\le 2^{k-m+1}(2t)^{(m-1-k)2}\dint_{B(x,\,2t)}
\lf|\nabla^{m-1}v(y)\r|^2\,dy.
\end{eqnarray*}
Thus, by this and properties of $\psi$, we see that, for all $(x,\,t)\in\rr^{n+1}_+$,
\begin{eqnarray}\label{eqn est1 proNMFC2}
\nonumber&&\lf\{\frac{1}{t^n} \dint_{B(x,\,t)} \lf|e^{-t^{2m}L}(f)(z)\r|^2\,dz\r\}^{1/2}\\
&&\nonumber\hs\le \lf\{\frac{1}{t^n} \dint_{B(x,\,2t)} \lf|\psi_{x,\,t}(z)e^{-t^{2m}L}(f)(z)\r|^2\,dz\r\}^{1/2}\\
&&\hs\ls \lf\{\frac{1}{t^n} \dint_{B(x,\,2t)} \lf|\lf(t\nabla\r)^{m-1}\lf(\psi_{x,\,t}e^{-t^{2m}L} (f)\r)(z)\r|^2\,dz\r\}^{1/2},
\end{eqnarray}
which, combined with \eqref{eqn est1 proNMFC2}, further shows that, for all $x\in\rn$,
\begin{eqnarray}\label{eqn est2 proNMFC2}
\mathcal{N}_{h,\,L}(f)(x)\ls \mathcal{N}^2_{h,\,\psi,\,L}(f)(x).
\end{eqnarray}

As for the converse direction, by Leibniz's rule and properties of $\psi$,
for all $(x,\,t)\in\rr^{n+1}_+$, we have
\begin{eqnarray*}
&&\lf\{\frac{1}{t^n} \dint_{B(x,\,2t)} \lf|\lf(t\nabla\r)^{m-1}
\lf(\psi_{x,\,t}e^{-t^{2m}L}(f)\r)(z)\r|^2\,dz\r\}^{1/2}\\
&&\hs\ls\lf\{\frac{1}{t^n} \dint_{B(x,\,2t)} \dsum_{k=0}^{m-1}
\lf|\lf(t\nabla\r)^{k}\lf(e^{-t^{2m}L}
(f)\r)(z)\r|^2\,dz\r\}^{1/2},
\end{eqnarray*}
which implies that, for all $x\in\rn$,
\begin{eqnarray}\label{eqn est3 proNMFC2}
\mathcal{N}^2_{h,\,\psi,\,L}(f)(x)\ls \wz{\mathcal{N}}^{2}_{h,\,L}(f)(x).
\end{eqnarray}

Combining \eqref{eqn est2 proNMFC2}, \eqref{eqn est3 proNMFC2},
Theorems \ref{non-tangential heat maximal functiona characterization}
and \ref{non-tangential heat maximal functiona characterization gradient},
we conclude that
$H_{L}^p(\rn)=H_{\mathcal{N}_{h,\,\psi,\,L}}^p(\rn)$
with equivalent quasi-norms, which completes
the proof of Proposition \ref{pro NTMFC2 L0}.
\end{proof}

Now, we prove Theorem \ref{thm SFLAF Hardy}. To this end,
we need the following proposition.

\begin{proposition}\label{pro MI}
Let $k\in\nn$, $L$ be as in \eqref{L homo} and satisfy the Ellipticity condition
$(\mathcal{E}_0)$. Let $S_{L,\,k}$ and $S_{h,\,L,\,k}$ be the same, respectively, as
in \eqref{eqn KSF} and \eqref{gradient square function for heat}. Then, for all
$q\in (0,\, {p_+(L)})$, there exists a positive constant $C_{(k,\,q)}$ such that
\begin{enumerate}
\item [{\rm (i)}]  for all $f\in {L^2(\rn)}$,
\begin{eqnarray*}
\lf\|S_{L,\,k} (f)\r\|_{L^q(\rn)}\le
C_{(k,\,q)}\|S_{L,\,1}(f)\|_{L^q(\rn)};
\end{eqnarray*}
\item [{\rm (ii)}] for all $f\in {L^2(\rn)}$,
\begin{eqnarray*}
\lf\|S_{h,\,L,\,k} (f)\r\|_{L^q(\rn)}\le
C_{(k,\,q)}\|S_{L,\,1}(f)\|_{L^q(\rn)}.
\end{eqnarray*}
\end{enumerate}
\end{proposition}

\begin{proof}
We prove Proposition \ref{pro MI} by mathematical induction.

If $k=1$, Proposition \ref{pro MI}(i) automatically holds true.
To prove Proposition \ref{pro MI}(ii),
by \eqref{estimates of IIm}, we know that,
for all $f\in L^2(\rn)$ and $x\in\rn$,
\begin{eqnarray}\label{eqn 3.25}
\lf[S^{\epsilon/2,\,2R,\,3/2}_{h,\,L,\,1}(f)(x)\r]^{2}\le
\lf[S^{\epsilon/4,\,3R,\,2}_{L,\,1}(f)(x)\r]^2.
\end{eqnarray}
By this, together with letting $\epsilon\to 0$ and $R\to \fz$,
\eqref{3.xx18} and \eqref{3.xxx18},
we further conclude that, for all $q\in (0,\,\fz)$ and $f\in L^2(\rn)$,
\begin{eqnarray}\label{3.53x}
\lf\|S_{h,\,L,\,1}(f)\r\|_{L^q(\rn)}\ls \lf\|S_{L,\,1}(f)\r\|_{L^q(\rn)},
\end{eqnarray}
which implies that Proposition \ref{pro MI}(ii) holds true in this case.

If $k=2$, we prove (i) by first establishing a desired estimate for
$\|S_{L,\,2}(f)\|_{L^q(\rn)}$ with $q\in (0,\, \fz)$
(see \eqref{3.31xx} below).
To this end, for all $0<\epsilon \ll R<\fz$, $\lz\in(0,\,\fz)$
and $x\in\rn$, let $\bgz^{\epsilon,\,R,\,\lz}(x)$ be the
{truncated cone} as in \eqref{contracted cone}.
Also, let $\eta\in C_c^\fz(\bgz^{\epsilon/2,\,2R,\,3/2}(x))$ satisfy $\eta\equiv1$
on $\bgz^{\epsilon,\,R,\,1}(x)$, $0\le \eta\le 1$ and, for all $l\in\nn$ with $l\le m$
and $(y,\,t)\in \bgz^{\epsilon/2,\,2R,\,3/2}(x)$,
$$\lf|\nabla^l\eta(y,\,t)\r|\ls \frac{1}{t^{l}}.$$
From properties of $\eta$,
the definition of ${L}$, Leibniz's rule and Minkowski's inequality, we deduce that,
for all $f\in L^2(\rn)$ and $x\in\rn$,
\begin{eqnarray}\label{eqn kI0-Im}
\nonumber S_{L,\,2}^{\epsilon,\,R,\,1}(f)(x)
&&=\lf[\iint_{\bgz^{\epsilon,\,R,\,1}(x)}
\lf|\lf(t^{2m}{L}\r)^2e^{-t^{2m}{L}}(f)(y)\r|^2
\frac{dy\,dt}{t^{n+1}}\r]^{\frac{1}{2}}\\
&&\nonumber\le \lf[\iint_{\bgz^{\epsilon/2,\,2R,\,3/2}(x)}
\lf(t^{2m}{L}\r)^2e^{-t^{2m}{L}}(f)(y)
\overline{\lf(t^{2m}{L}\r)^2e^{-t^{2m}{L}}(f)(y)}\eta(y,\,t)\,
\frac{dy\,dt}{t^{n+1}}\r]^{\frac{1}{2}}\\
&&\nonumber= \lf|\dsum_{|\az|=m=|\bz|}
\iint_{\bgz^{\epsilon/2,\,2R,\,3/2}(x)} a_{\az,\,\bz}(y)
t^{m}\pat^{\bz}\lf(t^{2m}Le^{-t^{2m}{L}}(f)\r)(y)\r.\\
&&\nonumber\hs\times
\overline{t^m\pat^{\az}\lf(\lf(t^{2m}{L}\r)^{2}e^{-t^{2m}{L}}(f)\eta\r)(y,\,t)}
\,\frac{dy\,dt}{t^{n+1}}\Bigg|^{\frac{1}{2}}\\
&&\nonumber\ls \dsum_{l=0}^{m}\lf|\dsum_{|\az|=m=|\bz|}
\iint_{\bgz^{\epsilon/2,\,2R,\,3/2}(x)} a_{\az,\,\bz}(y)
t^{m}\pat^{\bz}\lf(t^{2m}Le^{-t^{2m}{L}}(f)\r)(y)\r.\\
&&\nonumber\hs\times\lf.\lf[t^m \dsum_{|\wz\az|=l,\,\wz\az\le \az}
C_{(\az,\,\wz\az)}\overline{\pat^{\wz\az}\lf(\lf(t^{2m}{L}\r)^2
e^{-t^{2m}{L}}(f)\r)(y)\, \pat^{\az-\wz\az}\eta(y,\,t)}\r]
\,\frac{dy\,dt}{t^{n+1}}\r|^{\frac{1}{2}}\\
&&=: \dsum_{l=0}^{m}\mathrm{I}_l,
\end{eqnarray}
where, for all multi-indices $\az$ and $\wz\az$,
$C_{(\az,\,\wz\az)}$ denotes a positive constant depending on $\az$ and $\wz \az$.

For $\mathrm{I}_0$, by H\"older's inequality,
the size condition of $\eta$ and the Ellipticity condition $(\mathcal{E}_0)$,
we see that, for all $0<\epsilon \ll R<\fz$ and $x\in\rn$,
\begin{eqnarray}\label{eqn estimates of I0k}
\nonumber\mathrm{I}_0
&&\ls \lf\{\iint_{\bgz^{\epsilon/2,\,2R,\,3/2}(x)}
\lf|t^{m}\nabla^{m}\lf(t^{2m}Le^{-t^{2m}{L}}(f)\r)(y)\r|^2\frac{dy\,dt}
{t^{n+1}}\r\}^{\frac{1}{4}}\\
&&\nonumber\hs\times\lf\{\iint_{\bgz^{\epsilon/2,\,2R,\,3/2}(x)}
\lf|\lf(t^{2m}{L}\r)^2e^{-t^{2m}{L}}(f)(y)\r|^2
\frac{dy\,dt}{t^{n+1}}\r\}^{\frac{1}{4}}\\
&&\sim\lf[S^{\epsilon/2,\,2R,\,3/2}_{h,\,L,\,1}(f)(x)\r]^{\frac{1}{2}}
\lf[S^{\epsilon/2,\,2R,\,3/2}_{L,\,2}(f)(x)\r]^{\frac{1}{2}},
\end{eqnarray}
where $S^{\epsilon/2,\,2R,\,3/2}_{h,\,L,\,1}(f)$ and
$S^{\epsilon/2,\,2R,\,3/2}_{L,\,2}(f)$
are defined, respectively, similar to $S_{h,\,L,\,1}(f)$
in \eqref{gradient square function for heat} and $S_{L,\,2}(f)$ in \eqref{eqn KSF},
with $\bgz(x)$ in \eqref{1.x1} replaced by $\bgz^{\epsilon/2,\,2R,\,3/2}(x)$
in \eqref{contracted cone}.

For $\mathrm{I}_m$, using an argument similar to that used in the proof of
\eqref{estimates of Im}, we conclude that, for all $0<\epsilon \ll R<\fz$ and $x\in\rn$,
\begin{eqnarray}\label{eqn estimates of ImK}
\hs\hs\mathrm{I}_m&&\ls\lf[S^{\epsilon/2,\,2R,\,3/2}_{h,\,L,\,1}(f)(x)\r]^{\frac{1}{2}}
\lf[S^{\epsilon/2,\,2R,\,3/2}_{h,\,L,\,2}(f)(x)\r]^{\frac{1}{2}}.
\end{eqnarray}

Moreover, by an argument similar to that used in the proof of \eqref{estimates of IIm}
(see also \eqref{eqn 3.25}),
we find that, for all $0<\epsilon \ll R<\fz$ and $x\in\rn$,
\begin{eqnarray}\label{eqn 3.29}
\lf[{S}^{\epsilon/2,\,2R,\,3/2}_{h,\,L,\,2}(f)(x)\r]^2\ls
\lf[S^{\epsilon/4,\,3R,\,2}_{L,\,2}(f)(x)\r]^2.
\end{eqnarray}

Also, similar to \eqref{estimates of Ik}, we know that, for all $l\in\{1,\,\ldots,\,m-1\}$,
$0<\epsilon \ll R<\fz$ and $x\in\rn$,
\begin{eqnarray}\label{3.xx30}
\nonumber\mathrm{I}_l&&\ls\lf[S^{\epsilon/2,\,2R,\,3/2}_{h,\,L,\,1}(f)(x)\r]^{\frac{1}{2}}
\lf\{\iint_{\bgz^{\epsilon/2,\,2R,\,3/2}(x)}
\lf|t^l\nabla^l\lf(\lf[t^{2m}{L}\r]^2e^{-t^{2m}{L}}(f)\r)(y)\r|^2
\frac{dy\,dt}{t^{n+1}}\r\}^{\frac{1}{4}}\\
&&\ls \lf[S^{\epsilon/2,\,2R,\,3/2}_{h,\,L,\,1}(f)(x)\r]^{\frac{1}{2}}
\lf[{S}^{\epsilon/2,\,2R,\,3/2}_{h,\,L,\,2}(f)(x)\r]^{l/(2m)}
\lf[S^{\epsilon/2,\,2R,\,3/2}_{L,\,2}(f)(x)\r]^{(m-l)/(2m)}.
\end{eqnarray}

By combining \eqref{eqn kI0-Im} through \eqref{3.xx30} and then letting $\epsilon\to 0$ and $R\to \fz$,
we conclude that, there exists a positive constant $C$ such that, for all $f\in L^2(\rn)$ and
almost every $x\in\rn$,
\begin{eqnarray}\label{meqn3.60}
S_{L,\,2}(f)(x)\le C \lf[S^{2}_{h,\,L,\,1}(f)(x)\r]^{\frac{1}{2}}
\lf[S^{2}_{L,\,2}(f)(x)\r]^{\frac{1}{2}}.
\end{eqnarray}

Similarly, by following the same line of the proof of \eqref{meqn3.60},
we conclude that, for all $k\in\zz_+$, $f\in L^2(\rn)$ and almost every $x\in\rn$,
\begin{eqnarray*}
S^{2^k}_{L,\,2}(f)(x)\le  C\lf[ S_{h,\,L,\,1}^{2^{k+1}}(f)(x)\r]^{\frac{1}{2}}
\lf[S^{2^{k+1}}_{L,\,2}(f)(x)\r]^{\frac{1}{2}},
\end{eqnarray*}
which, combined with the definition of $\mathcal{A}_k$ in \eqref{meqnAk},
implies that, for all $f\in L^2(\rn)$ and almost every $x\in\rn$,
\begin{eqnarray*}
\mathcal{A}_k(F)(x)\le  C\lf[ \mathcal{A}_{k+1}(G)(x)\r]^{\frac{1}{2}}
\lf[\mathcal{A}_{k+1}(F)(x)\r]^{\frac{1}{2}},
\end{eqnarray*}
where $F:=(t^{2m}L)^2e^{-t^{2m}L}(f)$ and $G:=(t\nabla)^m t^{2m}Le^{-t^{2m}L}(f)$.
Moreover, since $f\in L^2(\rn)$, we know that $F$ and $G\in T^2(\rr^{n+1}_+)$.
Thus, by Lemma \ref{mlem 1}, we conclude that, for all $q\in (0,\,\fz)$ and $f\in L^2(\rn)$,
$\|F\|_{T^q(\rr^{n+1}_+)}\ls \|G\|_{T^q(\rr^{n+1}_+)}$, which implies that,
for all $q\in (0,\fz)$ and $f\in L^2(\rn)$,
\begin{equation}\label{3.31xx}
\|S_{L,\,2}(f)\|_{L^q(\rn)}\ls \|S_{h,\,L,\,1}(f)\|_{L^q(\rn)}.
\end{equation}
This, combined with \eqref{3.53x}, shows that Proposition \ref{pro MI}(i)
holds true for $k=2$.

Moreover, Proposition \ref{pro MI}(ii) when $k=2$
follows from \eqref{eqn 3.29}, \eqref{3.xx18}, \eqref{3.xxx18} and Proposition \ref{pro MI}(i) when $k=2$ .

Now, let $\wz k \in\nn\cap [3,\,\fz)$.
Assume that Proposition \ref{pro MI} holds true
for all $k\in\{1,\,\ldots,\,{\wz{k}}\}$. Thus, by mathematical induction,
to finish the proof of Proposition \ref{pro MI}, it suffices to show that
Proposition \ref{pro MI} also holds true for ${\wz{k}}+1$.

Similar to \eqref{eqn kI0-Im}, for all $0<\epsilon \ll R<\fz$, $f\in L^2(\rn)$
and $x\in\rn$, we have
\begin{eqnarray}\label{3}
\nonumber S^{\epsilon,\,R,\,1}_{L,\,{\wz{k}}+1}(f)(x)
&&=\lf[\iint_{\Gamma^{\epsilon,\,R,\,1}(x)}
\lf|\lf(t^{2m}{L}\r)^{{\wz{k}}+1}e^{-t^{2m}{L}}(f)(y)\r|^2
\frac{dy\,dt}{t^{n+1}}\r]^{\frac{1}{2}}\\
&&\nonumber\le \lf[\iint_{\Gamma^{\epsilon/2,\,2R,\,3/2}(x)}
\lf(t^{2m}{L}\r)^{{\wz{k}}+1}e^{-t^{2m}{L}}(f)(y)
\overline{\lf(t^{2m}{L}\r)^{{\wz{k}}+1}e^{-t^{2m}{L}}(f)(y)}\eta(y,\,t)\,
\frac{dy\,dt}{t^{n+1}}\r]^{\frac{1}{2}}\\
&&\nonumber= \lf|\dsum_{|\az|=m=|\bz|}
\iint_{\Gamma^{\epsilon/2,\,2R,\,3/2}(x)} a_{\az,\,\bz}(y)
t^{m}\pat^{\bz}\lf(\lf[t^{2m}L\r]^{{\wz{k}}}e^{-t^{2m}{L}}(f)\r)(y)\r.\\
&&\nonumber\hs\times
\overline{t^m\pat^{\az}\lf(\lf(t^{2m}{L}\r)^{{\wz{k}}+1}e^{-t^{2m}{L}}(f)\eta\r)(y,\,t)}
\,\frac{dy\,dt}{t^{n+1}}\Bigg|^{\frac{1}{2}}\\
&&\nonumber\ls \dsum_{l=0}^{m}\lf|\dsum_{|\az|=m=|\bz|}
\iint_{\Gamma^{\epsilon/2,\,2R,\,3/2}(x)} a_{\az,\,\bz}(y)
t^{m}\pat^{\bz}\lf(\lf[t^{2m}L\r]^{{\wz{k}}}e^{-t^{2m}{L}}(f)\r)(y)\r.\\
&&\nonumber\hs\times\lf.\lf[t^m \dsum_{|\wz\az|=l,\,\wz\az\le \az}
C_{(\az,\,\wz\az)}\overline{\pat^{\wz\az}\lf(\lf(t^{2m}{L}\r)^{{\wz{k}}+1}
e^{-t^{2m}{L}}(f)\r)(y)\, \pat^{\az-\wz\az}\eta(y,\,t)}\r]
\,\frac{dy\,dt}{t^{n+1}}\r|^{\frac{1}{2}}\\
&&=: \dsum_{l=0}^{m}\wz{\mathrm{I}}_l.
\end{eqnarray}

For $\wz{\mathrm{I}}_0$, by H\"older's inequality,
the size condition of $\eta$ and the Ellipticity condition $(\mathcal{E}_0)$,
we see that, for all $0<\epsilon \ll R<\fz$, $f\in L^2(\rn)$ and $x\in\rn$,
\begin{eqnarray}\label{4}
\nonumber\wz{\mathrm{I}}_0
&&\ls \lf\{\iint_{\Gamma^{\epsilon/2,\,2R,\,3/2}(x)}
\lf|t^{m}\nabla^{m}\lf(\lf[t^{2m}L\r]^{{\wz{k}}}
e^{-t^{2m}{L}}(f)\r)(y)\r|^2\frac{dy\,dt}
{t^{n+1}}\r\}^{\frac{1}{4}}\\
&&\nonumber\hs\times\lf\{\iint_{\Gamma^{\epsilon/2,\,2R,\,3/2}(x)}
\lf|\lf(t^{2m}{L}\r)^{{\wz{k}}+1}e^{-t^{2m}{L}}(f)(y)\r|^2
\frac{dy\,dt}{t^{n+1}}\r\}^{\frac{1}{4}}\\
&&\sim\lf[S^{\epsilon/2,\,2R,\,3/2}_{h,\,L,\,{\wz{k}}}(f)(x)\r]^{\frac{1}{2}}
\lf[S^{\epsilon/2,\,2R,\,3/2}_{L,\,{\wz{k}}+1}(f)(x)\r]^{\frac{1}{2}},
\end{eqnarray}
where $S^{\epsilon/2,\,2R,\,3/2}_{h,\,L,\,{\wz{k}}}(f)$ and
$S^{\epsilon/2,\,2R,\,3/2}_{L,\,{\wz{k}}+1}(f)$
are defined, respectively, similar to $S_{h,\,L,\,{\wz{k}}}(f)$
in \eqref{gradient square function for heat} and $S_{L,\,{\wz{k}}+1}(f)$ in \eqref{eqn KSF},
with $\bgz(x)$ in \eqref{1.x1} replaced by $\bgz^{\epsilon/2,\,2R,\,3/2}(x)$
in \eqref{contracted cone}.

For $\wz{\mathrm{I}}_m$, using an argument similar to that used in the proof of
\eqref{eqn estimates of ImK}, we conclude that, for all $0<\epsilon \ll R<\fz$, $f\in L^2(\rn)$ and $x\in\rn$,
\begin{eqnarray}\label{5}
\wz{\mathrm{I}}_m&&\ls \lf[S^{\epsilon/2,\,2R,\,3/2}_{h,\,L,\,{\wz{k}}}(f)(x)\r]^{\frac{1}{2}}
\lf[S^{\epsilon/2,\,2R,\,3/2}_{h,\,L,\,{\wz{k}+1}}(f)(x)\r]^{\frac{1}{2}}.
\end{eqnarray}
By an argument similar to that of \eqref{eqn 3.25},
we know that, for all $0<\epsilon \ll R<\fz$, $f\in L^2(\rn)$ and $x\in\rn$,
\begin{eqnarray*}
\lf[S^{\epsilon/2,\,2R,\,3/2}_{h,\,L,\,{\wz{k}}+1}(f)\r]^2\ls
\lf[S^{\epsilon/4,\,3R,\,2}_{L,\,{\wz{k}}+1}(f)(x)\r]^2.
\end{eqnarray*}
By this, together with letting $\epsilon\to 0$ and $R\to \fz$,
\eqref{3.xx18} and \eqref{3.xxx18},
we further conclude that, for all $q\in (0,\,\fz)$ and $f\in L^2(\rn)$,
\begin{eqnarray}\label{3.xx40}
\lf\|S_{h,\,L,\,\wz k+1}(f)\r\|_{L^q(\rn)}\ls \lf\|S_{L,\,\wz k+1}(f)\r\|_{L^q(\rn)}.
\end{eqnarray}

Also, similar to \eqref{3.xx30},
we know that, for all $l\in\{1,\,\ldots,\,m-1\}$, $0<\epsilon \ll R<\fz$, $f\in L^2(\rn)$ and $x\in\rn$,
\begin{eqnarray}\label{6}
\wz{\mathrm{I}}_l&&\le \lf[S^{\epsilon/2,\,2R,\,3/2}_{h,\,L,\,{\wz{k}}}(f)(x)\r]^{\frac{1}{2}}
\!\lf[S^{\epsilon/2,\,2R,\,3/2}_{h,\,L,\,{\wz{k}+1}}(f)(x)\r]^{l/(2m)}
\lf[S^{\epsilon/2,\,2R,\,3/2}_{L,\,{\wz{k}}+1}(f)(x)\r]^{(m-l)/(2m)}.
\end{eqnarray}

Thus, combining \eqref{3} through \eqref{6}, and
then letting $\epsilon\to 0$ and $R\to \fz$,
we conclude that, there exists a positive constant $C$ such that,
for all $f\in L^2(\rn)$ and almost every $x\in\rn$,
\begin{eqnarray}\label{meqn3.61}
S_{L,\,\wz k}(f)(x)\le C \lf[S^{2}_{h,\,L,\,\wz k-1}(f)(x)\r]^{\frac{1}{2}}
\lf[S^{2}_{L,\,\wz k}(f)(x)\r]^{\frac{1}{2}}.
\end{eqnarray}

Similarly, by following the same line of the proof of \eqref{meqn3.61},
we conclude that, for all $l\in\zz_+$, $f\in L^2(\rn)$ and
almost every $x\in\rn$,
\begin{eqnarray*}
S^{2^l}_{L,\,\wz k}(f)(x)\le  C\lf[ S_{h,\,L,\,\wz k-1}^{2^{l+1}}(f)(x)\r]^{\frac{1}{2}}
\lf[S^{2^{l+1}}_{L,\,\wz k}(f)(x)\r]^{\frac{1}{2}}.
\end{eqnarray*}
This, together with Lemma \ref{mlem 1} and
the assumption that Proposition \ref{pro MI} holds true for all
$k\in\{1,\,\ldots,\,{\wz{k}}\}$,
implies that Proposition \ref{pro MI}(i) also holds true in the case ${\wz{k}}+1$,
the details being omitted.

Finally, we see that  Proposition \ref{pro MI}(ii) in the case ${\wz{k}}+1$ follows from
\eqref{3.xx40}  and Proposition \ref{pro MI}(i) in the case ${\wz{k}}+1$,
which completes the proof of Proposition \ref{pro MI}.
\end{proof}

From the proof of Proposition \ref{pro MI}, we immediately deduce the following
conclusions.

\begin{corollary}\label{cor 3.9}
Let $k\in\nn$, $L$ be as in \eqref{L homo} and satisfy the Ellipticity condition
$(\mathcal{E}_0)$. Let $S_{L,\,k}$ and $S_{h,\,L,\,k}$ be the same, respectively, as
in \eqref{eqn KSF} and \eqref{gradient square function for heat}. Then, for all $p\in(0,\,p_+(L))$,
there exists  a positive constant $C_{(p,\,k)}$,
depending on $p$ and $k$, such that, for all $f\in L^2(\rn)$,
\begin{eqnarray}\label{3.x54}
\frac{1}{C_{(p,\,k)}}\lf\|S_{h,\,L,\,k}(f)\r\|_{L^p(\rn)}
\le \lf\|S_{L,\,k}(f)\r\|_{L^p(\rn)}
\le {C_{(p,\,k)}}\lf\|S_{h,\,L,\,k-1}(f)\r\|_{L^p(\rn)}.
\end{eqnarray}
\end{corollary}

\begin{proof}
The first inequality of \eqref{3.x54} follows immediately from  \eqref{3.xx40}
with $\wz k+1$ replaced by $k$ in the proof of Proposition \ref{pro MI}, while the second
inequality of \eqref{3.x54} is proved in the proof of Proposition \ref{pro MI},
which completes the proof of Corollary \ref{cor 3.9}.
\end{proof}

Combining Corollary \ref{cor 3.9} and Lemma \ref{lem 3.3}(i),
we immediately obtain the following corollary, which improves
Lemma \ref{lem 3.3}(ii) by extending the range of $q$ from $(q_-(L),\, q_+(L))$
to $(p_-(L),\, p_+(L))$.

\begin{corollary}\label{cor3.3}
Let $L$ be as in \eqref{L homo} and satisfy the Ellipticity condition
$(\mathcal{E}_0)$. Let $k\in\nn$, $q\in (p_-(L),\, p_+(L))$ and
$S_{h,\,L,\,k}$ be as in \eqref{gradient square function for heat}.
Then there exists a positive constant $C$ such that, for all
$f\in L^2(\rn)\cap L^q(\rn)$,
\begin{eqnarray*}
\lf\|S_{h,\,L,\,k} (f)\r\|_{L^q(\rn)}\le C\|f\|_{L^q(\rn)}.
\end{eqnarray*}
\end{corollary}

We are now in a position to prove Theorem \ref{thm SFLAF Hardy}.

\begin{proof}[Proof of Theorem \ref{thm SFLAF Hardy}]
We first prove Theorem \ref{thm SFLAF Hardy}(i).
If $p\in(0,\,2]$, by Proposition \ref{pro GSFC}, we know that
$H_{S_{L,\,k}}^p(\rn)= H_L^p(\rn)$. Thus, to finish the proof of
Theorem \ref{thm SFLAF Hardy}(i), it suffices to consider the
case $p\in(2,\,p_+(L))$.

Moreover, by Lemma \ref{lem 3.3}(i) and a density argument, we see that,
for all $p\in(2,\,p_+(L))$,
\begin{eqnarray*}
H_L^p(\rn)\subset H_{S_{L,\,k}}^p(\rn).
\end{eqnarray*}

On the other hand, for all $p\in (2,\,p_+(L))$,
by an argument similar to that used in the proof of \eqref{eqn DA}, we know that,
for all $k\in\nn$ and $f\in L^2(\rn)$,
\begin{eqnarray*}
\lf\|f\r\|_{L^p(\rn)}\ls \|S_{L,\,k}(f)\|_{L^p(\rn)},
\end{eqnarray*}
which, combined with Lemma \ref{lem relation Hardy Lebesgue L0} and
 a density argument, implies that, for all $p\in(2,\,p_+(L))$,
\begin{eqnarray*}
H_{S_{L,\,k}}^p(\rn)\subset H_L^p(\rn).
\end{eqnarray*}
This shows that Theorem \ref{thm SFLAF Hardy}(i) holds true.

We now prove Theorem \ref{thm SFLAF Hardy}(ii).
To show the inclusion that $H_{S_{h,\,L,\,k}}^p(\rn)\subset H_L^p(\rn)$, for all
$p\in(0,\,p_+(L))$, by Theorem \ref{thm SFLAF Hardy}(i) and Corollary \ref{cor 3.9},
we conclude that, for all $f\in L^2(\rn)$,
\begin{eqnarray*}
\lf\|f\r\|_{H_L^p(\rn)}\sim \lf\|S_{L,\,k+1}(f)\r\|_{L^p(\rn)}\ls
\lf\|S_{h,\,L,\,k}(f)\r\|_{L^p(\rn)},
\end{eqnarray*}
which, together with a density argument, implies that
$H_{S_{h,\,L,\,k}}^p(\rn)\subset H_L^p(\rn)$.

For the converse inclusion, if $p\in(p_-(L),\,p_+(L))$, by
Propositions \ref{pro MI}(ii) and \ref{cor HLP}, we see that, for all $f\in L^2(\rn)\cap L^p(\rn)$,
\begin{eqnarray*}
\|S_{h,\,L,\,k}(f)\|_{L^p(\rn)}\ls\|S_{L}(f)\|_{L^p(\rn)}\sim \|f\|_{L^p(\rn)},
\end{eqnarray*}
which, combined with Lemma \ref{lem relation Hardy Lebesgue L0}
and a density argument, implies that
$H_L^p(\rn)\subset H_{S_{h,\,L,\,k}}^p(\rn)$
holds true in the range $p\in(p_-(L),\,p_+(L))$.

If $p\in(0,\,1]$, by considering the action of $S_{h,\,L,\,k}$
on each $(p,\,2,\,M,\,\epsilon)_{L}$-molecule (see, for
example, \cite[(4.4)]{cy12} for a proof of a similar conclusion)
and Theorem \ref{molecular characterization for Hardy space L0},
we know that, for all $p\in(0,\,1]$ and $f\in H_L^p(\rn)$,
\begin{eqnarray*}
\lf\|S_{h,\,L,\,k} (f)\r\|_{L^p(\rn)}\ls
\lf\|f\r\|_{L^p(\rn)},
\end{eqnarray*}
which, together with the conclusion in the case $p \in(p_-(L),\,p_+(L))$,
Lemma \ref{lem relation Hardy Lebesgue L0}, the interpolation {(see
Proposition \ref{pro interpolation Hardy L0} together with \cite[p.\,52, Theorem]{Jan})}
and a density argument, implies that
$H_L^p(\rn)\subset H_{S_{h,\,L,\,k}}^p(\rn)$ holds true for all $p\in(0,\,q_+(L))$.
This finishes the proof of Theorem \ref{thm SFLAF Hardy}.
\end{proof}

\noindent{\bf Acknowledgements.} The authors would like to thank
the referee for his very carefully reading and many stimulating remarks
which do improve the presentation
of this article.

\bigskip

\noindent Jun Cao and Dachun Yang (Corresponding author)

\medskip

\noindent School of Mathematical Sciences, Beijing Normal
University, Laboratory of Mathematics and Complex Systems, Ministry
of Education, Beijing 100875, People's Republic of China

\smallskip

\noindent{\it E-mails:} \texttt{caojun1860@mail.bnu.edu.cn} (J. Cao)

\hspace{0.96cm}\texttt{dcyang@bnu.edu.cn} (D. Yang)

\bigskip

\noindent Svitlana Mayboroda

\medskip

\noindent School of Mathematics, University of Minnesota, Minneapolis, MN,
55455 USA

\smallskip

\noindent{\it E-mail:} \texttt{svitlana@math.umn.edu}

\end{document}